\documentclass{article}
\usepackage[utf8]{inputenc}
\usepackage[all,cmtip]{xy}
\usepackage{amssymb,amsmath,amsthm}
\usepackage{stmaryrd}
\usepackage{mathtools,mathdots}
\usepackage{mathrsfs,calligra, bbm}
\usepackage{pifont}
\usepackage[top=1.4in, bottom=1.4in, left=1in, right=1in]{geometry}
\usepackage{graphicx}
\usepackage{CJKutf8}
\usepackage{dirtytalk}
\usepackage{hyperref}
\usepackage{tikz-cd, tikz}
\usepackage{csquotes}
\usepackage{enumitem}
\usepackage[T1]{fontenc}
\usepackage{titlesec}
\DeclareMathAlphabet{\mathpzc}{OT1}{pzc}{m}{it}
\usepackage[british]{babel}
\usepackage[backend=biber,style=alphabetic,sorting=anyt, maxbibnames=99,maxcitenames=99,maxalphanames=99,
]{biblatex}
\usepackage{lipsum}
\usepackage{bookmark}
\bookmarksetup{
  numbered,
  open
}

\DeclareMathAlphabet{\mathpzc}{OT1}{pzc}{m}{it}

\newtheorem{Definition}{Definition}[subsection]
\newtheorem{Theorem}[Definition]{Theorem}
\newtheorem{Lemma}[Definition]{Lemma}
\newtheorem{Proposition}[Definition]{Proposition}
\newtheorem{Corollary}[Definition]{Corollary}
\newtheorem{Example}[Definition]{Example}

\newtheorem{Conjecture}[]{Conjecture}

\newtheorem{Notation}[Definition]{Notation}
\newtheorem{Hypothesis}[]{Hypothesis}
\newtheorem*{Conditions}{Conditions}

\theoremstyle{definition}
\newtheorem{Remark}[Definition]{Remark}

\newcommand{\blackqed}{\hfill$\blacksquare$}
\newcommand\scalemath[2]{\scalebox{#1}{\mbox{\ensuremath{\displaystyle #2}}}}

\DeclareMathOperator\C{\mathbb{C}}

\DeclareMathOperator\F{\mathbb{F}}

\DeclareMathOperator\Q{\mathbb{Q}}
\DeclareMathOperator\R{\mathbb{R}}

\DeclareMathOperator\Z{\mathbb{Z}}

\DeclareMathOperator\bfitx{\textbf{\textit{x}}}

\DeclareMathOperator\sfD{\mathsf{D}}

\DeclareMathOperator\bbE{\mathbb{E}}

\DeclareMathOperator\calD{\mathcal{D}}

\DeclareMathOperator\calL{\mathcal{L}}

\DeclareMathOperator\calO{\mathcal{O}}

\DeclareMathOperator\calS{\mathcal{S}}
\DeclareMathOperator\calT{\mathcal{T}}
\DeclareMathOperator\calU{\mathcal{U}}

\DeclareMathOperator\calW{\mathcal{W}}
\DeclareMathOperator\calX{\mathcal{X}}
\DeclareMathOperator\calY{\mathcal{Y}}
\DeclareMathOperator\calZ{\mathcal{Z}}

\DeclareMathOperator\scrE{\mathscr{E}}
\DeclareMathOperator\scrF{\mathscr{F}}
\DeclareMathOperator\scrG{\mathscr{G}}
\DeclareMathOperator\scrH{\mathscr{H}}

\DeclareMathOperator\scrM{\mathscr{M}}
\DeclareMathOperator\scrN{\mathscr{N}}
\DeclareMathOperator\scrO{\mathscr{O}}

\DeclareMathOperator\frakD{\mathfrak{D}}

\DeclareMathOperator\frakS{\mathfrak{S}}
\DeclareMathOperator\frakT{\mathfrak{T}}
\DeclareMathOperator\frakU{\mathfrak{U}}

\DeclareMathOperator\frakX{\mathfrak{X}}
\DeclareMathOperator\frakY{\mathfrak{Y}}
\DeclareMathOperator\frakZ{\mathfrak{Z}}

\DeclareMathOperator\frakm{\mathfrak{m}}

\DeclareMathOperator\frakp{\mathfrak{p}}

\DeclareMathOperator\codim{codim}
\DeclareMathOperator\Ext{Ext}
\DeclareMathOperator\GL{GL}

\DeclareMathOperator\Aut{Aut}
\DeclareMathOperator\End{End}
\DeclareMathOperator\Hom{Hom}

\DeclareMathOperator\sheafHom{\scrH\!\!\!om}

\DeclareMathOperator\AJ{AJ}

\DeclareMathOperator\coker{coker}
\DeclareMathOperator\Cone{Cone}

\DeclareMathOperator\disc{disc}

\DeclareMathOperator\dlog{dlog}
\DeclareMathOperator\DR{DR}
\DeclareMathOperator\dR{dR}

\DeclareMathOperator\et{\text{\'{e}t}}

\DeclareMathOperator\Fil{Fil}

\DeclareMathOperator\fp{fp}

\DeclareMathOperator\HK{HK}
\DeclareMathOperator\holim{holim}
\DeclareMathOperator\hocolim{hocolim}

\DeclareMathOperator\id{id}

\DeclareMathOperator\image{image}

\DeclareMathOperator\nr{nr}
\DeclareMathOperator\OC{OC}

\DeclareMathOperator\one{\mathbbm{1}}

\DeclareMathOperator\Poly{\mathbf{Poly}}
\newcommand{\pr}{pr}

\DeclareMathOperator\Res{Res}
\DeclareMathOperator\rig{rig}

\DeclareMathOperator\specialisation{sp}
\DeclareMathOperator\Spa{Spa}
\DeclareMathOperator\Spadagger{Spa^{\normalfont \dagger}}

\DeclareMathOperator\Spec{Spec}
\DeclareMathOperator\Spf{Spf}
\DeclareMathOperator\Spwf{Spwf}

\DeclareMathOperator\Sym{Sym}

\DeclareMathOperator\syn{syn}

\DeclareMathOperator\tr{tr}

\DeclareMathOperator\unip{unip}
\DeclareMathOperator\univ{univ}

\DeclareMathOperator\Mod{\mathbf{\mathsf{Mod}}}

\DeclareMathOperator\Sch{\mathbf{\mathsf{Sch}}}
\DeclareMathOperator\FSch{\mathbf{\mathsf{FSch}}}

\DeclareMathOperator\Rig{\mathbf{\mathsf{Rig}}}

\DeclareMathOperator\Sets{\mathbf{\mathsf{Sets}}}

\DeclareMathOperator\FIsoc{\mathbf{\mathsf{FIsoc}}}
\DeclareMathOperator\Isoc{\mathbf{\mathsf{Isoc}}}
\DeclareMathOperator\Shv{\mathbf{\mathsf{Shv}}}
\DeclareMathOperator\Syn{\mathbf{\mathsf{Syn}}}

\DeclareMathOperator\bfepsilon{\boldsymbol{\varepsilon}}

\DeclareMathOperator\llbrack{\![\![\!}
\DeclareMathOperator\rrbrack{\!]\!]}

\makeatletter
\renewcommand{\maketitle}{\bgroup\setlength{\parindent}{0pt}
\begin{flushleft}
  \LARGE{\textbf{\@title}}
  
  \vspace{4mm}
  
  \large{\textsc{\@author}} \hfill \normalfont{\text{\@date}}
  
  \vspace{4mm}
\end{flushleft}\egroup
}
\makeatother

\title{Finite polynomial cohomology with coefficients}
\author{Ting-Han Huang and Ju-Feng Wu}
\date{}

\begin{document}

\maketitle

{\footnotesize \noindent \textbf{Abstract.} We introduce a theory of finite polynomial cohomology with coefficients in this paper. We prove several basic properties and introduce an Abel--Jacobi map with coefficients. 
As applications, we use such a cohomology theory to study arithmetics of compact Shimura curves over $\Q$, and simplify proofs of the works of Darmon--Rotger and Bertolini--Darmon--Prasanna.}

\tableofcontents

\section{Introduction}

In this paper, we provide and study a theory of \emph{finite polynomial cohomology with coefficients}. Such a cohomology theory with trivial coefficient was first introduced by A. Besser in \cite{Besser-integral} for a proper smooth scheme.
In what follows, in order to motivate our study, we begin with a brief review of Besser's work. An overview of the present paper will be provided thereafter. In the end of this introduction, we discuss several further research directions and possible applications of our theory. 

\subsection{Besser's finite polynomial cohomology and the \texorpdfstring{$p$}{p}-adic Abel--Jacobi map}\label{subsection: review of Besser's work}
Fix a rational prime $p$ and let $K$ be a finite extension of $\Q_p$. Let $\calO_K$ be the ring of integer of $K$. Let $X$ be a proper smooth scheme over $\calO_K$ of relative dimension $d$. Given $m, n\in \Z$, Besser defined a cohomology theory $R\Gamma_{\fp, m}(X, n)$, whose $i$-th cohomology group is denoted by $H_{\fp, m}^i(X, n)$.
The novelties of this cohomology theory are the following:\begin{enumerate}
    \item[$\bullet$] One can view the theory of finite polynomial cohomology as a generalisation of Coleman's integration. 
    \item[$\bullet$] One can use finite polynomial cohomology to understand the $p$-adic Abel--Jacobi map explicitly.
\end{enumerate} Let us explain these in more details. 

From the definition of finite polynomial cohomology, one can easily deduce a fundamental exact sequence \begin{equation}\label{eq: fundamental exact seq. for fp cohomology}
    0 \rightarrow \frac{H_{\dR}^{i-1}(X_K)}{F^nH_{\dR}^{i-1}(X_K)} \xrightarrow{i_{\fp}} H_{\fp}^i(X, n) \xrightarrow{\pr_{\fp}} F^n H_{\dR}^i(X_K) \rightarrow 0,
\end{equation} where $H_{\fp}^i(X, n)$ stands for $H_{\fp, i}^i(X, n)$, and `$F^n$' stands for the $n$-th filtration of the de Rham cohomology groups. When $i = n = 1$, Besser explained in \cite[Theorem 1.1]{Besser-integral} that, given $\omega\in F^1H_{\dR}^1(X_K)$, any of its lift $\widetilde{\omega}\in H_{\fp}^1(X, 1)$ can be viewed as a Coleman integration of $\omega$. In particular, for any $x\in X(\calO_K)$, the map $x \mapsto x^*\widetilde{\omega} \in H^1_{\fp}(\Spec \calO_K,1) \cong K$ is the evaluation at $x$ of a Coleman integral $F_{\omega}$ of $\omega$.
Hence the finite polynomial cohomology can be viewed as a generalisation of Coleman's integration theory.

Let us turn our attention to the $p$-adic Abel--Jacobi map. We denote by $Z^i(X)$ the set of smooth irreducible closed subschemes of $X$ of codimension $i$ and let \[
    A^i(X) := \text{ the free abelian group generated by }Z^i(X).
\] Besser constructed a cycle class map \[
    \eta_{\fp}: A^i(X) \rightarrow H_{\fp}^{2i}(X, i), 
\] which, after composing with the projection $\pr_{\fp}: H_{\fp}^{2i}(X, i) \rightarrow F^i H_{\dR}^{2i}(X_K)$, agrees with the usual de Rham cycle class map. The $p$-adic Abel--Jacobi map is then defined to be \[
    \AJ : A^i(X)_0 := \ker \pr_{\fp}\circ \eta_{\fp} \xrightarrow{\eta_{\fp}} \frac{H_{\dR}^{2i-1}(X_K)}{ F^{i} H_{\dR}^{2i-1}(X_K) } \cong \left( F^{d-i+1} H_{\dR}^{2d-2i+1}(X_K) \right)^{\vee},
\] where the last isomorphism is given by the Poincaré duality. Besser then proved the following theorem.

\begin{Theorem}[$\text{\cite[Theorem 1.2]{Besser-integral}}$]\label{Theorem: Besser's AJ map}
For any $Z = \sum_{j} n_j Z_j\in A^i(X)_0$, $\AJ(Z)$ is the functional on $F^{d-i+1} H_{\dR}^{2d-2i+1}(X_K)$ such that, for any $\omega\in F^{d-i+1} H_{\dR}^{2d-2i+1}(X_K)$, \[
    \AJ(Z)(\omega) = \int_Z \omega := \sum_{j} n_j \tr_{Z_j}\iota_{Z_j}^*\widetilde{\omega},
\] where \begin{enumerate}
    \item[$\bullet$] $\widetilde{\omega}\in H_{\fp}^{2d-2i+1}(X, d-i+1)$ is a lift of $\omega$ via the projection $\pr_{\fp}$;
    \item[$\bullet$] $\iota_{Z_j}: Z_j \hookrightarrow X$ is the natural closed immersion; 
    \item[$\bullet$] $\tr_{Z_j}: H_{\fp}^{2d-2i+1}(Z_j, d-i+1) \cong H_{\dR}^{2d-2i}(Z_{j, K})\cong K$ is the canonical isomorphism induced from \eqref{eq: fundamental exact seq. for fp cohomology}.
\end{enumerate}
\end{Theorem}

\subsection{An overview of the present paper}
Immediately from the original construction, there are two natural questions: \begin{enumerate}
    \item[$\bullet$] Is there a theory of \emph{finite polynomial cohomology for general varieties}?
    \item[$\bullet$] Is there a theory of \emph{finite polynomial cohomology with non-trivial coefficients}?
\end{enumerate} The former is studied by Besser, D. Loeffler and S. Zerbes in \cite{BLZ-FPCoh} for varieties that is not required to have good reduction, by using methods developed by J. Nekovář and W. Nizioł in \cite{NN-syntomic}.

However, to the authors's knowledge, there seems to be no literature on finite polynomial cohomology with (non-trivial) coefficients. 
Note that there is an integration theory for overconvergent $F$-isocrystals developed by Coleman in \cite{C-pShimura}.
Moreover, the authors of \cite{BDP, DR} did apply Coleman integration theory to certain non-trivial sheaves on the modular curve.
Though they did not formalise a theory of finite polynomial cohomology with coefficients. 
Instead, they worked with finite polynomial cohomology (with trivial coefficient) of the Kuga--Sato variety over the modular curve. Nonetheless, the methods in \emph{loc. cit.} already use implicitly the idea of finite polynomial cohomology with non-trivial coefficients. It is then reasonable to expect the existence of a theory of finite polynomial cohomology with non-trivial coefficients.  

The purpose of this paper is to initiate the study of finite polynomial cohomology with non-trivial coefficients. The first question one encounters is what the eligible coefficients are. Fortunately, the work of K. Yamada (\cite{Yamada}) provides a suitable candidate. Indeed, given a proper weak formal scheme $\frakX$ over $\calO_K$ such that its special fibre $X_0$ is strictly semistable and its dagger generic fibre is smooth over $K$, Yamada defined a category $\Syn(X_0, \frakX, \calX)$ of \emph{syntomic coefficients}. Objects of this category are certain \emph{overconvergent $F$-isocrystals} that admit a filtration that satisfies Griffiths's transversality. Then, given an object $(\scrE, \Phi, \Fil^{\bullet})\in \Syn(X_0, \frakX, \calX)$, $n\in \Z$, and a (suitable) polynomial $P$, we are able provide a definition of \emph{syntomic $P$-cohomology with coefficients in $\scrE$}, denoted by $R\Gamma_{\syn, P}(\calX, \scrE, n)$.

Now, suppose that there is a proper smooth scheme $X$ over $\calO_K$ such that its $\varpi$-adic weak completion agrees with $\frakX$ (with trivial log structure). In this case, we are able to provide a definition of \emph{finite polynomial cohomology groups} $H_{\fp}^i(\calX, \scrE, n)$ for any $(\scrE, \Phi, \Fil^{\bullet})\in \Syn(X_0, \frakX, \calX)$ and any $n\in \Z$. We summarise some of its basic properties in the following theorem. 

\begin{Theorem}[Corollary \ref{Corollary: exact sequence when X is of good reduction}, Corollary \ref{Corollary: perfect pairing for good reduction case}, and Proposition \ref{Proposition: proper pushforward}]\label{Theorem: basic properties of fp cohomology; intro}
Let $X$ be a proper smooth scheme over $\calO_K$ of relative dimension $d$ with $\varpi$-adic weak completion $\frakX$. Let $(\scrE, \Phi, \Fil^{\bullet})\in \Syn(X_0, \frakX, \calX)$.
\begin{enumerate}
    \item[(i)]  For any $i\in \Z_{\geq 0}$ and any $n\in \Z$, we have a fundamental short exact sequence as in \eqref{eq: fundamental exact seq. for fp cohomology} \[
        0 \rightarrow \frac{H_{\dR}^{i-1}(\calX, \scrE)}{F^n H_{\dR}^{i-1}(\calX, \scrE)} \xrightarrow{i_{\fp}} H_{\fp}^i(\calX, \scrE, n) \xrightarrow{\pr_{\fp}} F^n H_{\dR}^i(\calX, \scrE) \rightarrow 0.
    \]
    \item[(ii)] There is a perfect pairing \[
        H_{\fp}^i(\calX, \scrE, n) \times H_{\fp}^{2d-i+1}(\calX, \scrE^{\vee}, d-n+1) \rightarrow K.
    \]
    \item[(iii)] For any irreducible closed subscheme $\iota: Z \hookrightarrow X$ which is smooth over $\calO_K$ and of codimension $i$, we have a pushforward map \[
        \iota_*: H_{\fp}^i(\calZ, \iota^* \scrE, n) \rightarrow H_{\fp}^{j+2i}(\calX, \scrE, n+i),
    \] where $\calZ$ is the dagger space associated with the $\varpi$-adic weak completion of $Z$.
\end{enumerate}
\end{Theorem}

The next question we asked is whether there exists an \emph{Abel--Jacobi map} for finite polynomial cohomology with coefficients. 
We first describe the relation between finite polynomial cohomology with coefficients and Coleman's integration of modules with connections.
After this is accomplished, we try to interpret the Abel--Jacobi map as certain kind of integration similar to the complex geometry case.
An immediate problem one encounters is which group should take the place of $A^i(X)$ when non-trivial coefficients are involved. More precisely, to the authors's knowledge, there does not exist a notion of `cycle class group with coefficients'. In this paper, we proposed a candidate for this purpose: \[
    A^i(X, \scrE) := \bigoplus_{Z\in Z^i(X)} H_{\dR}^0(Z, \scrE).
\] We then consider certain subgroup $A^i(X, \scrE)_0 \subset A^i(X, \scrE)$, which is an analogue to $A^i(X)_0$.

We have the following result.

\begin{Theorem}[Theorem \ref{Theorem: Abel--Jacobi map}]
There exists a finite polynomial Abel--Jacobi map for $(\scrE, \Phi, \Fil^{\bullet})$, \[
    \AJ_{\fp} = \AJ_{\fp, \scrE}: A^i(X, \scrE)_0 \rightarrow \left(F^{d-i+1} H_{\dR}^{2d-2i+1}(\calX, \scrE^{\vee})\right)^{\vee}
\] such that for any $(\theta_Z)_Z\in A^i(X, \scrE)_0$ and any $\omega\in F^{d-i+1} H_{\dR}^{2d-2i+1}(\calX, \scrE^{\vee})$, \[
    \AJ_{\fp}((\theta_Z)_Z)(\omega) = \sum_{Z} \left (\int_Z \omega \right ) (\theta_Z).
\] 
\end{Theorem} 

The theory of finite polynomial cohomology with coefficients has potential applications to study the arithmetic of automorphic forms. In the final section of this paper, we illustrate this by reproducing the formula in \cite[Theorem 3.8]{DR} and the formula in \cite[Proposition 3.18 \& Proposition 3.21]{BDP} in the case of compact Shimura curves over $\Q$. More precisely, let $X$ be the compact Shimura cruve over $\Q$, parametrising false elliptic curves of level away from $p$ and let $\pi: A^{\univ} \rightarrow X$ be the universal false elliptic curve over $X$. After fixing an idempotent $\bfepsilon\in M_2(\Z_p)\smallsetminus \{1, 0\}$, we consider $\underline{\omega} := \bfepsilon \pi_* \Omega_{A^{\univ}/X}^1$, $\scrH := \bfepsilon R^1\pi_* \Omega_{A^{\univ}/X}^{\bullet}$ and \[
    \underline{\omega}^k := \underline{\omega}^{\otimes k}, \quad \scrH^k := \Sym^k \scrH. 
\] 

The following theorem summarises the application in the direction of diagonal cycles:

\begin{Theorem}[Theorem \ref{Theorem: formula for diagonal cycles} and Corollary \ref{Corollary: relation with DR}]\label{Theorem: AJ(diagonal cycle); intro}
Let $(k, \ell, m)\in \Z^3$ such that \begin{enumerate}
    \item[$\bullet$] $k+ \ell +m\in 2\Z$, 
    \item[$\bullet$] $2<k\leq \ell \leq m$ and $m < k+\ell$, 
    \item[$\bullet$] $\ell + m - k = 2t + 2$ for some $t\in \Z_{\geq 0}$.
\end{enumerate} and write \[
    r_1:= k-2, \quad r_2 := \ell-2, \quad r_3:=m-2.
\] Then, there exists a cycle $\Delta_{2,2,2}^{k, \ell, m}\in A^2(X^3, (\scrH^{r_1}\otimes \scrH^{r_2}\otimes \scrH^{r_3})^{\vee})_0$, which is called the diagonal cycle with coefficients in $(\scrH^{r_1}\otimes \scrH^{r_2}\otimes \scrH^{r_3})^{\vee}$, such that for any \[
    \eta\in H_{\dR}^1(X, \scrH^{r_1}), \quad \omega_2\in F^1H_{\dR}^1(X, \scrH^{r_2}) = H^0(X, \underline{\omega}^{\ell}), \quad \text{ and }\quad \omega_3\in F^1H_{\dR}^1(X, \scrH^{r_3}) = H^0(X, \underline{\omega}^{m}),
\] we have the formula \[
    \AJ_{\fp}(\Delta_{2,2,2}^{k, \ell, m})(\eta\otimes\omega_2\otimes \omega_3) = \langle \eta, \xi(\omega_2, \omega_3)\rangle_{\dR},
\] where $\xi(\omega_2, \omega_3)$ is an element in $H_{\dR}^1(X, \scrH^{r_1}(-t))$ whose definition only depends on $\omega_2$ and $\omega_3$.
\end{Theorem}

The following theorem summarises the direction of application to cycles attached to isogenies:

\begin{Theorem}[Theorem \ref{Theorem: cycle attached to isogenies}]\label{Theorem: cycles attached to isogenies; intro}
Let $x=(A, i, \alpha)$ and $y=(A', i', \alpha')$ be two $K$-rational points on $X$ and suppose there exists an isogeny $\varphi: (A, i, \alpha) \rightarrow (A', i', \alpha')$. For any $r\in \Z_{>0}$, define the sheaf $\scrH^{r,r}:= \scrH^r \otimes \Sym^r \bfepsilon H_{\dR}^1(A)$ on $X$. Then, there exists a unique cycle $\Delta_{\varphi} \in A^1(X, \scrH^{r,r}(r))_0$ such that for any $\omega \in H^0(X, \underline{\omega}^{r+2})$ and any $\alpha\in \Sym^r \bfepsilon H_{\dR}^1(A)$, we have \[
    \AJ_{\fp}(\Delta_{\varphi})(\omega \otimes \alpha) = \langle \varphi^* (F_{\omega}(y)), \alpha\rangle,
\] 
where $F_{\omega}$ is the Coleman integral of the form $\omega$ and the pairing is the Poincar\'{e} pairing on $\Sym^r \bfepsilon H_{\dR}^1(A)$.
\end{Theorem}

We remark that the formulae we obtained in Theorem \ref{Theorem: AJ(diagonal cycle); intro} and Theorem \ref{Theorem: cycles attached to isogenies; intro} can be linked to special $L$-values as in \cite{DR} and \cite{BDP} respectively. 
As our purpose is to demonstrate how finite polynomial cohomology with coefficients can be used in practice, we do not pursue such a relation in this paper.

\subsection{Some further remarks}
Our theory of finite polynomial cohomology with coefficients leads to several further questions that we would like to study in future projects. Let us briefly discuss them in the following remarks. 

\paragraph{A comparison conjecture.} When the polynomial $P = 1-q^{-n}T$, where $q$ is the cardinality of the residue field of $\calO_K$, we follow the traditional terminology and call $R\Gamma_{\syn}(\calX, \scrE, n) := R\Gamma_{\syn, 1-q^{-n}T}(\calX, \scrE, n)$ the \emph{syntomic cohomology of $\calX$ with coefficients in $\scrE$ twisted by $n$}. When the coefficient $\scrE$ is the trivial coefficient, it follows from the result of J.-M. Fontaine and W. Messing (\cite{Fontaine--Messing}) that there is a canonical quasi-isomorphism \[
    \tau_{\leq n} R\Gamma_{\syn}(\calX, n) \simeq \tau_{\leq n}R\Gamma_{\et}(\calX_{\C_p}, \Q_p(n)).
\] Inspired by this classical result, we make the following conjecture. 

\begin{Conjecture}[Syntomic--étale comparison with coefficients]\label{Conj: syntomic--etale}
There exists a suitable subcategory of étale local systems on $\calX$, such that for any such étale local system $\bbE$, \begin{enumerate}
    \item[$\bullet$] there exists an associated overconvergent $F$-isocrystal $(\scrE, \Phi)$;
    \item[$\bullet$] for any $n\in \Z$, there exists a well-defined syntomic cohomology $R\Gamma_{\syn}(\calX, \scrE, n)$; and 
    \item[$\bullet$] there is a caonical quasi-isomorphism \[
        \tau_{\leq n} R\Gamma_{\syn}(\calX, \scrE, n) \simeq \tau_{\leq n} R\Gamma_{\normalfont \et}(\calX_{\C_p}, \bbE(n)).
    \]
\end{enumerate}
\end{Conjecture}

We remark that, although we employed the category of syntomic coefficients introduced in \cite{Yamada}, we believe this category is too restrictive. More precisely, Yamada's syntomic coefficients required the sheaves to be unipotent. However, for applications to the arithmetic of automorphic forms, we encounter sheaves that we don't know whether they are unipotent but still with some nice properties. Therefore, we do not directly conjecture a `syntomic étale local system' that corresponds to Yamada's syntomic coefficients. Instead, we formulate the above conjecture in a rather vague manner.
But we hope that it will be more useful for arithmetic applications.

\paragraph{A motivic expectation.} Readers might feel that our candidate $A^i(X, \scrE)$ of cycle class group with coefficients is too ad hoc. However, when $\scrE = \scrO_{X_0/\calO_{K_0}}$ is the trivial coefficient (\emph{i.e.}, the structure sheaf of the overconvergent site of the special fibre $X_0$), one can check that the classical $p$-adic Abel--Jacobi map factors as \[
    \AJ: A^i(X)_0 \rightarrow A^i(X, \scrO_{X_0/\calO_{K_0}})_0 \rightarrow \left( F^{d-i+1} H_{\dR}^{2d-2i+1}(\calX) \right)^{\vee}.
\] This justifies the use of $A^i(X, \scrE)$. 

However, the classical cycle class group $A^i(X)$ is strongly linked to the theory of motivic cohomology and algebraic $K$-theory, while it is not obvious that our cycle class group $A^i(X, \scrE)$ has such a link at the first glance. Nevertheless, we still believe such a link may exist as stated in the next conjecture. 

\begin{Conjecture}\label{Conj: motivic}
There exists a suitable category of coefficients for the motivic cohomology of $X$ such that for any coefficient $\scrE_{\mathrm{mot}}$ in this category \begin{enumerate}
    \item[$\bullet$] there exists an associated overconvergent $F$-isocrystal $(\scrE, \Phi)$; 
    \item[$\bullet$] there exists a canonical map from the motivic cohomology of $X$ with coefficients in $\scrE_{\mathrm{mod}}$ to $A^i(X, \scrE)_0$, i.e., $H_{\mathrm{mot}}^i(X, \scrE_{\mathrm{mot}}) \rightarrow A^i(X, \scrE)_0$; and 
    \item[$\bullet$] there exists an Abel--Jacobi map \[
        \AJ: H^i_{\mathrm{mot}}(X, \scrE_{\mathrm{mot}}) \rightarrow \left(F^{d-i+1} H_{\dR}^{2d-2i+1}(\calX, \scrE^{\vee})\right)^{\vee}
    \] such that it factors as \[
        \begin{tikzcd}
            H^i_{\mathrm{mot}}(X, \scrE_{\mathrm{mot}}) \arrow[rr, "\AJ"]\arrow[rd] && \left(F^{d-i+1} H_{\dR}^{2d-2i+1}(\calX, \scrE^{\vee})\right)^{\vee}\\
            & A^i(X, \scrE)_0\arrow[ru, "\AJ_{\fp}"']
        \end{tikzcd}.
    \]
\end{enumerate}
\end{Conjecture}

We remark that the relation between syntomic cohomology and motivic cohomology without coefficients is directly established by Ertl--Nizioł in \cite{EN}. It is shown in \emph{loc. cit.} that such a link between syntomic cohomology and motivic cohomology is tightened with the relation between syntomic cohomology and étale cohomology. It could be possible that the method of Ertl--Nizioł is generalisable to the setting with non-trivial coefficients and provides satisfactory answers to Conjecture \ref{Conj: syntomic--etale} and Conjecture \ref{Conj: motivic} simultaneously. We wish to come back to this in our future study.

\paragraph{Arithmetic applications.} To the authors's knowledge, there are at least two possible directions for arithmetic applications: \begin{enumerate}
    \item[(I)] As mentioned above, we applied our theory to study the arithmetic of compact Shimura curves over $\Q$. However, the relation between the Abel--Jacobi map and arithmetic has been wildly investigated for more general Shimura varieties. Note that, without the theory of finite polynomial cohomology with coefficients, the process of obtaining arithmetic information is quite indirect: often one starts with a non-PEL-type Shimura variety, then one needs to use the Jacquet--Langlands correspondence to move to a PEL-type Shimura variety and use `Liebermann's trick'. We believe that, with the theory of finite polynomial cohomology with coefficients and with some mild modification if necessary, one should be able to gain arithmetic information directly without bypassing the aforementioned process.
    
    \item[(II)] In \cite{BdS}, Besser and E. de Shalit used techniques of finite polynomial cohomology to construct $\calL$-invariants for $p$-adically uniformised varieties. When such $p$-adically uniformised variety is a Shimura variety, it is a natural question to ask whether the method of Besser--de Shalit can be generalised to obtain $\calL$-invariants of automorphic forms on this Shimura variety of higher weights. The first step to answer this question is, of course, to introduce coefficients to the cohomology groups they study. We hope that our work will shed some light in this direction and we wish to come back to this question in our future study. 
\end{enumerate}

\subsection*{Structure of the paper} 
In \S \ref{section: Geometry} and \S \ref{section: HK cohomology}, we recall some preliminaries of the \emph{rigid Hyodo--Kato theory} developed by V. Ertl and Yamada in \cite{EY, Yamada}. More precisely, we recall the definition of \emph{weakl formal schemes} in \S \ref{subsection: weakly formal schemes} and its relation with \emph{dagger spaces} in \S \ref{subsection: dagger spaces}. For the convenience of the readers, we briefly discuss the \emph{log structures} that one can put on weak formal schemes. Then, we follow the strategy of \cite{Yamada} to introduce the category of overconvergent $F$-isocrystals in \S \ref{subsection: overconvergent F-isocrystals} and the Hyodo--Kato theory with coefficients in \S \ref{subsection: HK theory with coeff}.

In \S \ref{section: FP cohomology}, we provide a definition of \emph{syntomic $P$-cohomology} with coefficients. Our definition works not only for proper smooth weak formal schemes over $\calO_K$, but also for proper weak formal schemes over $\calO_K$ with strictly semistable reduction. Such a definition is inspired by \cite{BLZ-FPCoh}. In \ref{subsection: main definition and first properties}, we prove some basic properties of this cohomology theory. By applying the results of Ertl--Yamada in \cite{EY-Poincare}, we discuss the cup products and pushforward maps for syntomic $P$-cohomology with coefficients in \ref{subsection: Cup products}. In \S \ref{subsection: finite polynomial cohomology}, we define \emph{finite polynomial cohomology with coefficients} for proper smooth scheme over $\calO_K$.

In \S \ref{subsection: Coleman's p-adic integration}, we motivate the definition of the Abel--Jacobi map by explaining the relation between the finite polynomial cohomology and Coleman's integration, in the case that $X$ is a proper smooth curve over $\calO_K$. Our Abel--Jacobi map with coefficients is constructed in \S \ref{subsection: AJ map}. For this construction, we assume that our weak formal scheme is the weak formal completion of a proper smooth scheme $X$ over $\calO_K$. In \S \ref{subsection: towards a semistable theory}, we briefly discuss a possible direction on how the Abel--Jacobi maps can be generalised to the case when $X$ has semistable reduction. We remark here that such a theory is far from satisfactory. 

Finally, we apply our theory to study arithmetic of compact Shimura curves over $\Q$ in \S \ref{section: Applications}. We will recall the definition and basic properties of compact Shimura curves over $\Q$ in \S \ref{subsection: compact Shimura curves over Q}. The construction of the diagonal cycle $\Delta_{2,2,2}^{k, \ell, m}$ and the formula in Theorem \ref{Theorem: AJ(diagonal cycle); intro} are provided in \S \ref{subsection: diagonal cycle} while the construction of the cycle $\Delta_{\varphi}$ and the formulae in Theorem \ref{Theorem: cycles attached to isogenies; intro} are shown in \S \ref{subsection: cycles; isogenies}. Note that, as mentioned above, we do not know whether the sheaves that appear in such arithmetic applications are the coefficients introduced in \cite{Yamada}. Nevertheless, our theory can work for more general coefficients that satisfy certain conditions. These conditions are stated at the beginning of \S \ref{section: Applications}. In fact, the syntomic coefficients introduced in \cite{Yamada} and $\scrH^k$ satisfy these conditions.

\subsection*{Acknowledgement} The project initially grew from the authors' curiosity on whether a theory of finite polynomial cohomology with coefficients exists during a working seminar with Giovanni Rosso and Martí Roset. We would like to thank them for intimate discussions. We thank Adrian Iovita for suggesting the computation in \S \ref{subsection: cycles; isogenies}. Many thanks also go to Antonio Cauchi for explaining the motivic viewpoint of our cycle class group, which led to the formulation of Conjecture \ref{Conj: motivic}.  We are grateful to Kazuki Yamada for answering our questions about his work and for many helpful discussions. We also thank David Loeffler for helpful discussions. Finally, we thank the anonymous referee for pointing out mistakes in the previous version and for constructive suggestions, which led to the improvement of the exposition of the article. 

The first draft of the paper was written when both authors were Ph.D. students at Concordia University. While revising the paper, T.-H. H. was a postdoctoral researcher in LAGA Universit\'{e} Paris XIII and was supported by the grant G5392391GRA3C565LAGAXR from the ANR-DFG Project HEGAL;
J.-F. W. was supported by the ERC
Consolidator grant `Shimura varieties and the BSD conjecture' and the Irish Research Council under grant number IRCLA/2023/849 (HighCritical).

\subsection*{Conventions and notations} Throughout this paper, we fix the following: \begin{enumerate}
    \item[$\bullet$] Let $p$ be a positive rational prime number and let $K$ be a finite field extension of $\Q_p$ with ring of integers $\calO_K$ and residue field $k$. We let $q:= \# k$ and so $q = p^e$ for some $e\in \Z_{>0}$. We also fix a uniformiser $\varpi\in \calO_K$. 
    \item[$\bullet$] We denote by $K_0:= W(k)[1/p]$ the maximal unramified field extension of $\Q_p$ inside $K$ and write $\calO_{K_0} = W(k)$ for its ring of integers. 
    \item[$\bullet$] Choose once and forever an algebraic closure $\overline{K}$ of $K$ and denote by $\C_p$ the $p$-adic completion of $\overline{K}$. We normalise the $p$-adic norm $|\cdot|$ on $\C_p$ so that $|p| = 1/p$.
    \item[$\bullet$] For any commutative square \[
        \begin{tikzcd}
            A^{\bullet} \arrow[r, "\alpha"]\arrow[d] & B^{\bullet}\arrow[d]\\
            C^{\bullet} \arrow[r, "\gamma"] & D^{\bullet} 
        \end{tikzcd}
    \] of complexes (of abelian groups), we write \[
        \left[\begin{tikzcd}
            A^{\bullet} \arrow[r, "\alpha"]\arrow[d] & B^{\bullet}\arrow[d]\\
            C^{\bullet} \arrow[r, "\gamma"] & D^{\bullet} 
        \end{tikzcd}\right] := \Cone\left(\Cone(\alpha)[-1] \rightarrow \Cone(\gamma)[-1]\right)[-1],
    \] where the map $\Cone(\alpha)[-1] \rightarrow \Cone(\gamma)[-1]$ is induced from the vertical maps from the square. 
    \item[$\bullet$] In principle, symbols in Gothic font (\emph{e.g.}, $\frakX, \frakY, \frakZ$) stand for formal schemes or weak formal schemes; symbols in calligraphic font (\emph{e.g.}, $\calX, \calY, \calZ$) stand for rigid analytic spaces or dagger spaces; and symbols in script font (\emph{e.g.}, $\scrO, \scrF, \scrE$) stand for sheaves (over various geometric objects).
\end{enumerate}

\section{Preliminaries I: Geometry}\label{section: Geometry} The purpose of this section is to recall some terminologies in algebraic geometry that will play essential roles in this paper. More precisely, we recall the theory of \emph{weak formal schemes} in \S \ref{subsection: weakly formal schemes} by following \cite[\S 1]{EY} and recall the notion of \emph{dagger spaces} in the language of \emph{adic spaces} in \S \ref{subsection: dagger spaces} by following \cite[\S 2]{Vezzani-MW}. Finally, we briefly discuss \emph{log structures} in \S \ref{subsection: log structure}.  We claim no originality in this section.

\subsection{Weak formal schemes}\label{subsection: weakly formal schemes} In this subsection, we fix a noetherian ring $R$ with an ideal $I$.

\begin{Definition}\label{Definition: weak completion}
 Let $A$ be an $R$-algebra and let $\widehat{A}$ be the $I$-adic completion of $A$. \begin{enumerate}
    \item[(i)] The \textbf{$I$-adic weak completion} $A^{\normalfont \dagger}$ of $A$ is the $R$-subalgebra of $\widehat{A}$, consisting of elements for the form \[
        f= \sum_{i=0}^{\infty} P_i(a_1, ..., a_r),
    \] where $a_1, ..., a_r\in A$ and $P_i(T_1, ..., T_r)\in I^i R[T_1, ..., T_r]$, such that there exists a constant $C>0$ satisfying \[
        C(i+1)\geq \deg P_i
    \] for all $i\geq 0$.
    
    \item[(ii)] We say $A$ is \textbf{weakly complete} (resp., \textbf{weakly complete finitely generated (wcfg)}) if $A^{\normalfont \dagger} = A$ (resp., there exists a surjective $R$-algebra morphism $R[X_1, ..., X_n]^{\normalfont \dagger} \rightarrow A$).
\end{enumerate}
\end{Definition}

Given a weakly complete $R$-algebra $A$, it is regarded naturally as a topological $R$-algebra with respect to the $I$-adic topology. If $A$ is wcfg, we also regard $A$ as a topological $R$-algebra with respect to the $I$-adic topology. Hence, the surjection \[
    R[X_1, ..., X_n]^{\dagger} \rightarrow A
\] is a continuous $R$-algebra morphism. 

\begin{Definition}[$\text{\cite[Definition 1.3]{EY}}$]\label{Definition: pseudo-wcfg}
A topological $R$-algebra $A$ is said to be \textbf{pseudo-weakly complete finitely generated (pseudo-wcfg)} if there exists an ideal of definition $J\subset A$ and a finite generating system $a_1, ..., a_m$ of $J$ such that the morphism \[
    R[T_1, ..., T_m] \rightarrow A, \quad T_i\mapsto a_i
\] makes $A$ an $(I, T_1, ..., T_m)$-adically wcfg $R[T_1, ..., T_m]$-algebra. Pseudo-wcfg $R$-algebras form a category whose morphisms  are given by continuous $R$-algebra homomorphisms. 
\end{Definition}

Unwinding this definition, given a pseudo-wcfg $R$-algebra $A$, there exists $n\in \Z_{\geq 0}$ and a surjective continuous $R[T_1, ..., T_m]$-algebra morphism \[
    \xi: R[T_1, ..., T_m][X_1, ..., X_n]^{\dagger} \rightarrow A,
\] where the weak completion is taken with respect to the $(I, T_1, ..., T_m)$-adic topology. We call such a $\xi$ a \textbf{\textit{representation}} of $A$. Note that, by \cite[Corollary 1.5]{EY}, the condition of being pseudo-wcfg is independent to the choice of the ideals of definition and the choice of generating systems. 

Our next goal is to globalise the aforementioned terminology. To this end, we have to study \emph{localisations} of pseudo-wcfg $R$-algebras, resembling the theory of schemes. Let $A$ be a pseudo-wcfg $R$-algebra with a representation $R[T_1, ..., T_m][X_1, ..., X_n]^{\dagger} \rightarrow A$. For any $f\in A$, define \[
    A_f^{\dagger} := \text{ the $(I, T_1, ..., T_m)$-adic weak completion of $A_f$}.
\] According to \cite[Proposition 1.19]{EY}, $A_f^{\dagger}$ is independent to the choice of representations $\xi$. Consequently, we can define the following ringed space \begin{equation}\label{eq: affine weak formal scheme}
    \Spwf A:= (\Spec A/J, \scrO_{\Spwf A}),
\end{equation} where the structure sheaf $\scrO_{\Spwf A}$ is defined by \[
    \scrO_{\Spwf A}: \Spec (A/J)_{f} \mapsto A_f^{\dagger}
\] for any $f\in A$. Note that the underlying topological space $\Spec A/J$ is independent to the choice of ideals of definition and so $\scrO_{\Spwf A}$ is well-defined.

\begin{Definition}[$\text{\cite[Definition 1.8]{EY}}$]\label{Definition: weak formal scheme}
A \textbf{weak formal scheme} $\frakX$ over $R$ is a ringed space which admits an open covering $\{\frakU_{\lambda}\}_{\lambda\in \Lambda}$ such that each $\frakU_{\lambda}$ is isomorphic to $\Spwf A$ for some pseudo-wcfg $R$-algebra $A$. The category of weak formal schemes over $R$ is denoted by $\FSch^{\normalfont \dagger}_{R}$.
\end{Definition}

\begin{Remark}\label{Remark: functors from weak formal schemes}
\normalfont From the construction, one sees that we have the following natural functors: \begin{enumerate}
    \item[(i)] Let $\Sch_{R/I}$ be the category of schemes over $R/I$. Then, we have a functor \[
        \FSch_{R}^{\dagger} \rightarrow \Sch_{R/I}, \quad \frakX\mapsto X_0
    \] locally given by \[
        \Spwf A = (\Spec A/J, \scrO_{\Spwf A})\mapsto \Spec A/J.
    \]
    \item[(ii)] Let $\FSch_{R}$ be the category of formal schemes over $R$. Then, we have a functor \[
        \FSch_{R}^{\dagger} \rightarrow \FSch_R, \quad \frakX \mapsto \widehat{\frakX}
    \] locally given by \[
        \Spwf A = (\Spec A/J, \scrO_{\Spwf A})\mapsto \Spf A = (\Spec A/J, \scrO_{\Spf A}),
    \] where the structure sheaf $\scrO_{\Spf A}$ assigns each distinguished open subspace $\Spec(A/J)_f$ to the completion $\widehat{A}_f$ of $A_f$.
\end{enumerate} Consequently, for any $\frakX\in \FSch_{R}^{\dagger}$, we shall refer $X_0$ as the \textbf{\textit{associated scheme}} over $R/I$ and $\widehat{\frakX}$ as the \textbf{\textit{associated formal scheme}} over $R$. 
\blackqed
\end{Remark}

\subsection{Dagger spaces}\label{subsection: dagger spaces} In this subsection, let $F$ be an nonarchimedean field of mixed characteristic $(0, p)$ with valuation ring $\calO_F$. Let $\frakm_F$ be the maximal ideal of $\calO_F$ and let $\varpi_F\in \frakm_F$ be a pseudouniformiser of $F$ that divides $p$. For any $n\in \Z_{\geq 0}$, define \begin{align*}
    F[X_1, ..., X_n]^{\dagger} := & \varinjlim_{h}F\langle \varpi_{F}^{1/h}X_1, ..., \varpi_{F}^{1/h}X_{n}\rangle\\
    = & \left\{ \sum_{(i_1, ..., i_n)\in \Z_{\geq 0}^n}a_{i_1, ..., i_n}X_1^{i_1}\cdots X_n^{i_n}\in F\llbrack X_1, ..., X_n\rrbrack: \lim |a_{i_1, ..., i_n}|\delta^{\sum i_j}  =0 \text{ for some }\delta\in \R_{>1}\right\}, 
\end{align*}
where, for any $h\in \Z_{>1}$, \[
    F\langle \varpi_{F}^{1/h}X_1, ..., \varpi_{F}^{1/h}X_{n}\rangle =  \left\{ \sum_{(i_1, ..., i_n)\in \Z_{\geq 0}^n}a_{i_1, ..., i_n}X_1^{i_1}\cdots X_n^{i_n}\in F\llbrack X_1, ..., X_n\rrbrack: \lim |a_{i_1, ..., i_n}||\varpi_F|^{\frac{-1}{h}\sum i_j}  =0\right\}.
\] The algebra $F[X_1, ..., X_n]^{\dagger}$ is equipped with the topology given by the Gau{\ss} norm and the completion of $F[X_1, ..., X_n]^{\dagger}$ is the Tate algebra $F\langle X_1, ..., X_n\rangle$.

More generally, a \textbf{\textit{dagger algebra}} $A$ is a topological $F$-algebra which is isomorphic to $F[X_1, ..., X_n]^{\normalfont \dagger}/(f_1, ..., f_r)$ for some $n\in \Z_{\geq 0}$ and some $f_i\in F\langle \varpi_F^{1/N}X_1, ..., \varpi_F^{1/N}X_n\rangle$ for some sufficiently large $N$. Note that, if $\widehat{A}$ denotes the completion of $A$, then $\widehat{A} \simeq F\langle X_1, ..., X_n\rangle/(f_1, ..., f_r)$.

For any dagger algebra $A$ over $F$, we define the topological space \[
    |\Spadagger A| := |\Spa (\widehat{A}, \widehat{A}^{\circ})|,
\] where $\widehat{A}^{\circ}$ is the ring of power bounded elements of $\widehat{A}$. For any $f_1, ..., f_n, g\in A\subset \widehat{A}$, we define the rational subspace \[
    |\Spadagger A[f_1, ..., f_n/g]^{\dagger}| := \{|f_i|\leq |g|\text{ for all }i=1, ..., n\}\subset |\Spadagger A|.
\] By \cite[Proposition 2.8]{GK-dagger}, we see that rational subspaces in $|\Spadagger A|$ form a basis of the topology. Consequently, we can consider the following ringed space \[
    \Spadagger A := (|\Spadagger A|, \scrO_{\Spadagger A}), 
\] where the structure sheaf $\scrO_{\Spadagger A}$ is defined by \[
    \scrO_{\Spadagger A}: |\Spadagger A[f_1, ..., f_n/g]^{\dagger}|\mapsto A[f_1, ..., f_n/g]^{\dagger} := A[X_1, ..., X_n]^{\dagger}/(gX_i-f_i).
\] Here, \[
    A[X_1, ..., X_n]^{\dagger} := A \otimes^{\dagger} F[X_1, ..., X_n]^{\dagger}
\] with the tensor product $\otimes^{\dagger}$ in the category of dagger algebras over $F$ (see [\emph{op. cit.}, Paragraph 1.16]).

\begin{Definition}\label{Definition: dagger space}
A \textbf{dagger space} $\calX$ over $F$ is a ringed space which admits an open covering $\{\calU_{\lambda}\}_{\lambda\in \Lambda}$ such that each $\calU_{\lambda}$ is isomorphic to $\Spadagger A$ for some dagger algebra $A$ over $F$. We denote by $\Rig^{\normalfont \dagger}_F$ the category of dagger spaces over $F$.  
\end{Definition}

\begin{Remark}\label{Remark: dagger spaces and rigid analytic spaces}
\normalfont 
Let $\Rig_F$ be the category of rigid analytic spaces over $F$.\footnote{ In this article, we will always view rigid analytic spaces as adic spaces.} Then, from the construction, one sees that there is a natural functor \[
    \Rig_F^{\dagger} \rightarrow \Rig_F, \quad \calX \mapsto \widehat{\calX}
\] locally given by \[
    \Spadagger A \mapsto \Spa (\widehat{A}, \widehat{A}^{\circ}).
\] Hence, for any dagger space $\calX$ over $F$, we call $\widehat{\calX}$ the \textbf{\textit{associated rigid analytic space}}. 
\blackqed
\end{Remark}

As an analogue of the relationship between formal schemes and rigid analytic spaces, one can also associate a dagger space over $F$ to a weak formal scheme over $\calO_F$. This phenomenon was first established in \cite{LM}. We recall this construction by following \cite{EY} for the notion of weak formal schemes introduced in the previous subsection.

\begin{Proposition}[$\text{\cite[Proposition 1.21]{EY}}$]\label{Proposition: affine weak formal scheme to affinoid dagger space}
Let $A$ be a pseudo-wcfg $\calO_F$-algebra with ideal of definition $J$ and a generating system $a_1, ..., a_n\in J$. For any $m\in \Z_{\geq 0}$, define \begin{align*}
    A_m := & A\left[ \frac{a_1^{m_1}\cdots a_n^{m_n}}{p}: m_i\in \Z_{\geq 0}, \sum m_i = m\right]^{\normalfont \dagger}\\
    = & \left(A \left[X_{m_1, ..., m_n}: m_i\in \Z_{\geq 0}, \sum m_i = m \right]^{\normalfont \dagger}/(pX_{m_1, ..., m_n} - a_1^{m_1}\cdots a_n^{m_n})\right)/\text{$p$-torsion}.
\end{align*} \begin{enumerate}
    \item[(i)] Each $A_m$ is wcfg over $\calO_F$ and is independent to the choice of $a_1, ..., a_n$. 
    \item[(ii)] The wcfg $\calO_F$-algebras $A_m$ form a projective system and defines a dagger space \[
        (\Spwf A)_{\eta} = \bigcup_{m\geq 0} \Spadagger A_m\otimes_{\calO_F} F
    \] that is independent to the choice of $J$.
\end{enumerate}
\end{Proposition}

Thanks to the independence in the proposition above, we have a natural functor \[
    \FSch_{\calO_F}^{\dagger} \rightarrow \Rig_F^{\dagger}, \quad \frakX \mapsto \calX = \frakX_{\eta},
\] locally given by \[
    \Spwf A \mapsto (\Spwf A)_{\eta}. 
\] For any $\frakX\in \FSch_{\calO_F}^{\dagger}$, we then call $\calX = \frakX_{\eta}$ the \textbf{\textit{associated dagger space}} or the \textbf{\textit{dagger generic fibre}} of $\frakX$.

\subsection{Log structure}\label{subsection: log structure}
The purpose of this subsection is to briefly discuss how one can equip a weak formal scheme a \emph{log structure}. In particular, we shall make precise the definition of \emph{strictly semistable} weak formal schemes. We encourage readers who are unfamiliar with the language of log geometry to consult \cite{Illusie} for more details.

Let $R$ be a noetherian ring with an ideal $I$. Let $\frakX$ be a weak formal scheme over $R$ with respect to the $I$-adic topology. Recall that a \textbf{\textit{pre-log structure}} on $\frakX$ is a sheaf of (commutative) monoids $\scrM$ on the étale site $\frakX_{\et}$ of $\frakX$ together with a morphism of sheaves of monoids \[
    \alpha: \scrM \rightarrow \scrO_{\frakX_{\et}}.
\] It is furthermore called a \textbf{\textit{log structure}} if the induced morphism \[
    \alpha: \alpha^{-1}\scrO_{\frakX_{\et}}^\times \rightarrow \scrO_{\frakX_{\et}}^\times
\] is an isomorphism. Recall also that given a pre-log structure $\scrM$ on $\frakX$, there is an \emph{associated log structure} $\scrM^a$ on $\frakX$, constructed by the pushout diagram \[
    \begin{tikzcd}
        \alpha^{-1}\scrO_{\frakX_{\et}}^\times \arrow[r, "\alpha"]\arrow[d] & \scrO_{\frakX_{\et}}\arrow[d, dashed]\\
        \scrM\arrow[r, dashed] & \scrM^a
    \end{tikzcd}.
\]  Finally, recall that a morphism of weak formal schemes with log structures is a pair \[
    (f, f^{\natural}): (\frakX, \scrM) \rightarrow (\frakY, \scrN),
\] where $f: \frakX \rightarrow \frakY$ is a morphism of weak formal schemes and $f^{\natural}: f^{-1}\scrN \rightarrow \scrM$ is a morphism of sheaves of monoids. The morphism $(f, f^{\natural})$ is a \textbf{\textit{strict open immersion}} (resp., \textbf{\textit{closed immersion}}; resp., \textbf{\textit{exact closed immersion}}) if $f$ is an open immersion (resp., closed immersion; resp., closed immersion) and $f^{\natural}$ is an isomorphism (resp., surjection; resp., isomorphism).\footnote{ For other adjectives for morphisms of log weak formal schemes, \emph{e.g.}, smooth and étale morphisms, we refer the readers to \cite[\S 1.3]{EY}.} To simplify the notation, we shall often drop the $f^{\natural}$ in the definition.

\begin{Definition}
Let $\frakX$ be a weak formal scheme with a log structure $\alpha:\scrM \rightarrow \scrO_{\frakX_{\et}}$. Let $P$ be a monoid and we abuse the notation to denote the associated constant sheaf of monoids on $\frakX$ again by $P$. A \textbf{chart of $\frakX$ modeled on $P$} is a morphism of sheaves of monoids \[
    \theta: P \rightarrow \scrM
\] such that the log structure associated with $\alpha\circ \theta$ is isomorphic to $\scrM$.
\end{Definition}

The following example play an essential role in this paper. 

\begin{Example}\label{Example: model of semistability}
\normalfont 
Let $R = \calO_{K}$ with $I = (\varpi)$. Let $n\in \Z_{\geq 0}$ and $r\leq n$, consider \[
    \frakX = \Spwf \calO_K[ X_1, ..., X_n]^{\dagger}/(X_1\cdots X_r - \varpi).
\] Then, $\frakX$ can be equipped with a log structure given by the chart \[
    \Z_{\geq 0}^n \rightarrow \calO_K[ X_1, ..., X_n]^{\dagger}/(X_1\cdots X_r - \varpi), \quad (a_1, ..., a_n)\mapsto X_1^{a_1}\cdots X_n^{a_n}.
\] In this case, the normal crossing divisor in $\frakX$ defined by $X_{r+1} \cdots X_n$ is called the \textbf{\textit{horizontal divisor}} of $\frakX$. Finally, by letting $\calO_K^{\log = \varpi}$ be the log weak formal scheme $\Spwf \calO_K$ with the log structure given by the chart \[
    \Z_{\geq 0} \rightarrow \calO_{K}, \quad 1 \mapsto \varpi,
\] one sees easily that $\frakX$ is a log weak formal scheme over $\calO_K^{\log = \varpi}$.
\blackqed
\end{Example}

\begin{Definition}\label{Definition: semistability}
Let $\frakX$ be a log weak formal scheme over $\calO_K^{\log = \varpi}$. We say $\frakX$ is \textbf{strictly semistable} if, Zariski locally, there is a strict smooth morphism \[
    \frakX \rightarrow \Spwf \calO_{K}[X_1, ..., X_n]^{\normalfont \dagger}/(X_1\cdots X_r - \varpi),
\] of log weak formal schemes over $\calO_{K}^{\log = \varpi}$. Moreover, the normal crossing divisor locally generated by $X_{r+1}\cdots X_n$ is called the \textbf{horizontal divisor} of $\frakX$.
\end{Definition}

\section{Preliminaries II: Hyodo--Kato theory}\label{section: HK cohomology}
In this section, we recall the theory of Hyodo--Kato cohomology with coefficients by following \cite{Yamada}. To this end, we fix the following notations though out this section: \begin{enumerate}
    \item[$\bullet$] We denote by $k^{\log = 0}$ be the log scheme $\Spec k$ with the log structure given by \[
        \Z_{\geq 0} \rightarrow k, \quad 1\mapsto 0.
    \]
    \item[$\bullet$] Recall $\calO_{K_0} := W(k)$. We similarly denote by $\calO_{K_0}^{\log = 0}$ the log weak formal scheme $\Spwf \calO_{K_0}$ with log structure given by \[
        \Z_{\geq 0} \rightarrow \calO_{K_0}, \quad 1\mapsto 0.
    \] Moreover, we write $\calO_{K_0}^{\log = \emptyset}$ for the weak formal scheme $\Spwf \calO_{K_0}$ with the trivial log structure. Moreover, the Frobenius $\phi$ on $W(k)$ induces natural Frobenii on $\calO_{K_0}^{\log = 0}$ and $\calO_{K_0}^{\log = \emptyset}$, which are still denoted by $\phi$. 
    \item[$\bullet$] As before, we write $\calO_K^{\log = \varpi}$ for the log weak formal scheme $\Spwf \calO_K$ with the log structure given by \[
        \Z_{\geq 0} \rightarrow \calO_K, \quad 1 \mapsto \varpi.
    \]
    \item[$\bullet$] Let $\frakS$ be the log weak formal scheme $\Spwf \calO_{K_0}\llbrack T\rrbrack$ with the log structure given by \[
        \Z_{\geq 0} \rightarrow \calO_{K_0}\llbrack T \rrbrack, \quad 1 \mapsto T.
    \] Consequently, we have natural morphisms of fine weak formal schemes \[
        \calO_{K_0}^{\log = 0}\xrightarrow{\sigma_0} \frakS \xleftarrow{\sigma_{\varpi}} \calO_K^{\log = \varpi}
    \] given by \[
        0 \mapsfrom T \mapsto \varpi. 
    \] Moreover, we extend the Frobenius $\phi$ to $\frakS$ by setting $\phi: T\mapsto T^p$. We denote by $\calS$ the dagger space associated with $\frakS$; it is the open unit ball with a structure of a dagger space. 
    \item[$\bullet$] We fix a scheme $X_0$ over $k$, which is assumed to be strictly semistable over $k^{\log = 0}$, \emph{i.e.},  Zariski locally, we have a strict smooth morphism \[
        X_0 \rightarrow \Spec k[X_1, ..., X_n]/(X_1\cdots X_r)
    \] of log schemes over $k^{\log = 0}$. Here, the log structure on the right-hand side is given by the chart \[
        \Z_{\geq 0}^n \rightarrow k[X_1, ..., X_n]/(X_1\cdots X_r), \quad (a_1, ..., a_n)\mapsto X_1^{a_1}\cdots X_n^{a_n}.
    \]
\end{enumerate}

\subsection{Overconvergent \texorpdfstring{$F$}{F}-isocrystals and rigid cohomology with coefficients}\label{subsection: overconvergent F-isocrystals}  
Our goal in this subsection is to define and study \emph{overconvergent $F$-isocrystals}. We follow the strategy in \cite{Yamada}, introducing first the so-called \emph{log overconvergent site} of $X_0$. The idea of such a site goes back to B. Le Stum's work \cite{LS-overconvergent} without log structure. 

\begin{Definition}[$\text{\cite[Definition 2.22]{Yamada}}$]\label{Definition: overconvergent site}
Let $\frakT$ be either $\frakS$, $\calO_{K_0}^{\log = 0}$, $\calO_{K_0}^{\log = \emptyset}$ or $\calO_{K}^{\log = \varpi}$ and so we have a natural homeomorphic exact closed immersion $\iota: k^{\log = 0}\rightarrow \frakT$. The \textbf{log overconvergent site} $\OC(X_0/\frakT)$ of $X_0$ relative to $\frakT$ is defined as follows:\begin{enumerate}
    \item[$\bullet$] An object in $\OC(X_0/\frakT)$ is a commutative diagram \[
        \begin{tikzcd}
            & Z \arrow[r, "i"]\arrow[d, "h_k"]\arrow[dl, "\theta"'] & \frakZ\arrow[d, "h_{\frakT}"]\\
            X_0\arrow[r] & k^{\log =0}\arrow[r, "\iota"] & \frakT
        \end{tikzcd},
    \] where \begin{enumerate}
        \item[$\circ$] $i:Z \rightarrow \frakZ$ is a homeomorphic exact closed immersion from a fine log scheme $Z$ over $k^{\log = 0}$ into a log weak formal scheme $\frakZ$ that is flat over $\Z_p$;
        \item[$\circ$] $h_k: Z \rightarrow k^{\log = 0}$ (resp., $h_{\frakT}: \frakZ \rightarrow \frakT$) is a morphism of schemes (resp., weak formal schemes); 
        \item[$\circ$] $\theta: Z \rightarrow X_0$ is a morphism of log schemes over $k^{\log = 0}$. 
    \end{enumerate} We shall usually abbreviate such an object as a quintuple $(Z, \frakZ, i, h, \theta)$. 
    \item[$\bullet$] Morphisms in $\OC(X_0/\frakT)$ are the obvious ones.
    \item[$\bullet$] A cover in $\OC(X_0/\frakT)$ is a collection of morphisms $\{f_{\lambda}: (Z_{\lambda}, \frakZ_{\lambda}, i_{\lambda}, h_{\lambda}, \theta_{\lambda}) \rightarrow (Z, \frakZ, i, h, \theta)\}_{\lambda\in \Lambda}$ such that \begin{enumerate}
        \item[$\circ$] the induced morphism $\frakZ_{\lambda}\rightarrow \frakZ$ is strict; 
        \item[$\circ$] the induced family of morphisms on the generic dagger fibres $\{\calZ_{\lambda} \rightarrow \calZ\}_{\lambda\in \Lambda}$ is an open cover for $\calZ$; 
        \item[$\circ$] the induced morphism $Z_{\lambda} \rightarrow Z\times_{\frakZ}\frakZ_{\lambda}$ is an isomorphism for every $\lambda$.
    \end{enumerate}
\end{enumerate}
\end{Definition}

\begin{Remark}\label{Remark: OC(X/T) as a sliced site}
One can also consider the \emph{absolute} log overconvergent site $\OC(k^{\log = 0}/\frakT)$, whose objects are just commutative diagrams \[
    \begin{tikzcd}
        Z \arrow[r, "i"]\arrow[d, "h_k"] & \frakZ\arrow[d, "h_{\frakT}"]\\
        k^{\log = 0}\arrow[r, "\iota"] & \frakT 
    \end{tikzcd}
\] as in Definition \ref{Definition: overconvergent site} and whose morphisms and covers are defined similarly as above. 
Hence, we can view $X_0$ as a presheaf on $\OC(k^{\log =0}/\frakT)$, defined as \[
    X_0(Z, \frakZ, i, h) = \left\{ \text{commutative diagrams }\begin{tikzcd}
            & Z \arrow[r, "i"]\arrow[d, "h_k"]\arrow[dl, "\theta"'] & \frakZ\arrow[d, "h_{\frakT}"]\\
            X_0\arrow[r] & k^{\log =0}\arrow[r, "\iota"] & \frakT
        \end{tikzcd}\right\}.
\]
Consequently, we then can view $\OC(X_0/\frakT)$ as the sliced site of $\OC(k^{\log = 0}/\frakT)$ over the presheaf $X_0$.
\blackqed
\end{Remark}

\begin{Lemma}\label{Lemma: realisation and structure sheaf}
Let $\frakT$ be either $\frakS$, $\calO_{K_0}^{\log = 0}$, $\calO_{K_0}^{\log = \emptyset}$ or $\calO_{K}^{\log = \varpi}$. \begin{enumerate}
    \item[(i)] The category of sheaves on $\OC(X_0/\frakT)$ is equivalent to the following category: \begin{enumerate}
        \item[$\bullet$] Objects are collections of sheaves $\scrF_{\calZ} = \scrF_{(Z, \frakZ, i, h, \theta)}$ on $\calZ$ for each $(Z, \frakZ, i, h, \theta)\in \OC(X_0/\frakT)$ and morphisms \[
            \tau_{f}: f_{\eta}^{-1}\scrF_{\calZ'} \rightarrow \scrF_{\calZ}
        \] for each morphism $f: (Z, \frakZ, i, h, \theta)\rightarrow (Z', \frakZ', i', h', \theta')$ satisfying the usual cocylce condition and $f$ is an isomorphism if $\calZ$ is an open subset of $\calZ'$. Here, $f_{\eta}$ is the morphism on the dagger generic fibres induced by $f$.
        \item[$\bullet$] Morphisms are compatible morphisms between collections of sheaves. 
    \end{enumerate} For each sheaf $\scrF$ on $\OC(X_0/\frakT)$, we call $\scrF_{\calZ}$ the \textbf{realisation} of $\scrF$ on $\calZ$.
    \item[(ii)] The presheaf $\scrO_{X_0/\frakT}$ on $\OC(X_0/\frakT)$ defined by \[
        \scrO_{X_0/\frakT} : (Z, \frakZ, i, h, \theta)\mapsto \scrO_{\calZ}(\calZ)
    \] is a sheaf. We then call $\scrO_{X_0/\frakT}$ the \textbf{structure sheaf} on $\OC(X_0/\frakT)$.
\end{enumerate}
\end{Lemma}
\begin{proof}
The first assertion follows from \cite[Proposition 2.1.9]{LS-overconvergent} while the second assertion follows from [\emph{op. cit.}, Corollary 2.3.3].
\end{proof}

\begin{Definition}[$\text{\cite[Definition 2.25]{Yamada}}$]
Let $\frakT$ be either $\frakS$, $\calO_{K_0}^{\log = 0}$, $\calO_{K_0}^{\log = \emptyset}$ or $\calO_{K}^{\log = \varpi}$.
\begin{enumerate}
    \item[(i)] A \textbf{log overconvergent isocrystal} on $X_0$ over $\frakT$ is an $\scrO_{X_0/\frakT}$-module $\scrE$ such that, \begin{enumerate}
        \item[$\bullet$] for any $(Z, \frakZ, i, h, \theta)\in \OC(X_0/\frakT)$, the realisation $\scrE_{\calZ}$ is a coherent locally free $\scrO_{\calZ}$-module; and 
        \item[$\bullet$] for any morphism $f: (Z, \frakZ, i, h, \theta) \rightarrow (Z', \frakZ', i', h', \theta')$, the induced morphism \[
            \tau_{f}: f_{\eta}^*: \scrE_{\calZ} \rightarrow \scrE_{\calZ'}
        \] is an isomorphism.
    \end{enumerate} Log overconvergent isocrystals naturally form a category, which is denoted by $\Isoc^{\normalfont \dagger}(X_0/\frakT)$.
    \item[(ii)] When $\frakT \neq \calO_K^{\log = \varpi}$, a \textbf{log overconvergent $F$-isocrystal} on $X_0$ over $\frakT$ is a pair $(\scrE, \Phi)$, where $\scrE\in \Isoc^{\normalfont\dagger}(X_0/\frakT)$ and $\Phi: \phi^*\scrE \xrightarrow{\simeq} \scrE$ is an isomorphism.\footnote{ Here, we abuse the notation, write $\phi$ to be the endomorphism on $\Isoc^{\normalfont \dagger}(X_0/\frakT)$ induced from the Frobenius on $\frakT$ and the absolute Frobenius on $X_0$ (see \cite[(2.21)]{Yamada}).} Log overconvergent $F$-isocrystals naturally form a category, which is denoted by $\FIsoc^{\normalfont \dagger}(X_0/\frakT)$.
\end{enumerate}
\end{Definition}

\begin{Lemma}\label{Lemma: internal homs and tensor products for FIsoc}
Let $\frakT$ be either $\frakS$, $\calO_{K_0}^{\log = 0}$, $\calO_{K_0}^{\log = \emptyset}$ or $\calO_{K}^{\log = \varpi}$. Then, the category $\Isoc^{\normalfont\dagger}(X_0/\frakT)$ (resp., $\FIsoc^{\normalfont\dagger}(X_0/\frakT)$) admits internal $\Hom$'s and tensor products, denoted by $\sheafHom$ and $\otimes$ respectively. In particular, given $\scrE\in \Isoc^{\normalfont \dagger}(X_0/\frakT)$ (resp., $(\scrE, \Phi)\in \FIsoc^{\normalfont \dagger}(X_0/\frakT)$), the dual $\scrE^{\vee}$ (resp., $(\scrE, \Phi)^{\vee}$) is well-defiend.
\end{Lemma}
\begin{proof}
The constructions are given in \cite[Definition 2.27]{Yamada}.
\end{proof}

Let $\frakT$ be either $\frakS$, $\calO_{K_0}^{\log = 0}$, $\calO_{K_0}^{\log = \emptyset}$ or $\calO_{K}^{\log = \varpi}$ and let $\scrE \in \Isoc^{\dagger}(X_0/\frakT)$ (resp.,  $(\scrE, \Phi)\in \FIsoc^{\dagger}(X_0/\frakT)$ when $\frakT \neq \calO_K^{\log = \varpi}$). Our next goal is to introduce the notion of \emph{log rigid cohomology of $X_0$ relative to $\frakT$ with coefficients in $\scrE$ (resp. $(\scrE, \Phi)$}). To this end, we fix a collection $\{(Z_{\lambda}, \frakZ_{\lambda}, i_{\lambda}, h_{\lambda}, \theta_{\lambda})\}_{\lambda\in \Lambda}$, where $(Z_{\lambda}, \frakZ_{\lambda}, i_{\lambda}, h_{\lambda}, \theta_{\lambda})\in \OC(X_0/\frakT)$ such that   \begin{enumerate}
    \item[$\bullet$] each $Z_{\lambda}$ is of finite type over $k$,
    \item[$\bullet$] $\{\theta_{\lambda}: Z_{\lambda} \rightarrow X_0\}_{\lambda\in \Lambda}$ is a Zariski open cover by strict open immersions, and 
    \item[$\bullet$] every $h_{\lambda, \frakT}: \frakZ_{\lambda} \rightarrow \frakT$ is smooth. 
\end{enumerate} When $\frakT \neq \calO_K^{\log = \varpi}$, we further fix a Frobenius lift $\varphi_{\lambda}$ on $\frakZ_{\lambda}$ which is compatible with $\phi$. The existence of such a collection is guaranteed by \cite[Proposition 2.33]{Yamada}. 

For any finite subset $\Xi\subset \Lambda$, choose an exactification (\cite[Proposition 4.10]{Kato-LogFontaineIllusie}) \[
    Z_{\Xi} := \bigcap_{\lambda\in \Xi} Z_{\lambda} \xrightarrow{i_{\Xi}} \frakZ_{\Xi} \rightarrow \prod_{\lambda\in \Xi, \frakT}\frakZ_{\lambda}
\] of the diagonal embedding $Z_{\Xi} \rightarrow  \prod_{\lambda\in \Xi, \frakT}\frakZ_{\lambda}$. By letting $h_{\Xi, k}: Z_{\Xi} \rightarrow k^{\log =0}$ and $h_{\Xi, \frakT}: \frakZ_{\Xi} \rightarrow \frakT$ be the structure morphisms and $\theta_{\Xi}:= \bigcap_{\lambda\in \Xi}\theta_{\lambda}$, we see that \[
    (Z_{\Xi}, \frakZ_{\Xi}, i_{\Xi}, h_{\Xi}, \theta_{\Xi})\in \OC(X_0/\frakT).
\]  Furthermore, if $\frakT \neq \calO_{K}^{\log = \varpi}$, the Frobenii $\varphi_{\lambda}$'s induces a Frobenius $\varphi_{\Xi}$ on $\frakZ_{\Xi}$ by $\varphi_{\Xi}:= (\prod_{\lambda\in \Xi}\varphi_{\lambda})|_{\frakZ_{\Xi}}$.

The map $h_{\Xi, \frakT}: \frakZ_{\Xi} \rightarrow \frakT$ induces a map on the generic fibre $h_{\Xi, \calT}: \calZ_{\Xi} \rightarrow \calT$ where $\calZ_{\Xi}$ and $\calT$ are the dagger spaces associated with $\frakZ_{\Xi}$ and $\frakT$ respectively. Let $\Omega_{\calZ_{\Xi}/\calT}^{\log, \bullet}$ be the complex of log differential forms of $\calZ_{\Xi}$ over $\calT$. We write \[
    \Omega_{Z_{\Xi}, \eta}^{\log, \bullet} := \Omega_{\calZ_{\Xi}/\calT}^{\log, \bullet}|_{]Z_{\Xi}[_{\calZ_{\Xi}}}.
\] By \cite[Lemma 1.2]{GK-Drinfeld}, we know that $\Omega_{Z_{\Xi}, \eta}^{\log, \bullet}$ is independent from the choice of exactification $\frakZ_{\Xi}$. 

For $\scrE\in \Isoc^{\dagger}(X_0/\frakT)$, we have the realisation $\scrE_{\Xi}$ on $\calZ_{\Xi}$. By \cite[Corollary 2.27]{Yamada}, we know that $\scrE_{\Xi}$ is equipped with an integrable connection \[
    \nabla: \scrE_{\Xi} \rightarrow \scrE_{\Xi}\otimes \Omega_{\calZ_{\Xi}/\calT}^{\log, \bullet}.
\] Moreover, if $(\scrE, \Phi)\in \FIsoc^{\dagger}(X_0/\frakT)$, then $\Phi$ induces a commutative diagram \begin{equation}\label{eq: integrable connecition with Frobenius structure}
    \begin{tikzcd}
        \varphi_{\Xi}^*\scrE_{\Xi}\arrow[r, "\varphi_{\Xi}^*\nabla"]\arrow[d, "\Phi"'] & \varphi_{\Xi}^*\scrE_{\Xi} \otimes \varphi_{\Xi}^*\Omega_{\calZ_{\Xi}/\calT}^{\log, \bullet}\arrow[d, "\Phi\otimes \varphi_{\Xi}"]\\
        \scrE_{\Xi}\arrow[r, "\nabla"] & \scrE_{\Xi}\otimes \Omega_{\calZ_{\Xi}/\calT}^{\log, \bullet}
    \end{tikzcd}.
\end{equation} In particular, by restricting to $]Z_{\Xi}[_{\calZ_{\Xi}}$, we can consider the complex $R\Gamma(]Z_{\Xi}[_{\calZ_{\Xi}}. \scrE_{\Xi}\otimes\Omega_{Z_{\Xi}, \eta}^{\log, \bullet})$.

For any finite subsets $\Xi_1\subset \Xi_2\subset \Lambda$, one has a natural map $\delta_{\Xi_2, \Xi_1}: ]Z_{\Xi_2}[_{\calZ_{\Xi_2}} \rightarrow ]Z_{\Xi_1}[_{\calZ_{\Xi_1}}$, which induces $\delta_{\Xi_2, \Xi_1}^{-1}(\scrE_{\Xi_1}\otimes \Omega^{\bullet}_{Z_{\Xi_1}, \eta}) \rightarrow \scrE_{\Xi_2}\otimes \Omega_{Z_{\Xi_2}, \eta}^{\bullet}$. Consequently, after fixing an order on $\Lambda$, one obtains a simplicial dagger space $]Z_{\bullet}[_{\calZ_{\bullet}}$ and a complex of sheaves $\scrE_{\bullet}\otimes \Omega_{Z_{\bullet}, \eta}^{\bullet}$ on $]Z_{\bullet}[_{\calZ_{\bullet}}$. Consequently, $\Phi$ induces a Frobenius action on the complex $R\Gamma(]Z_{\bullet}[_{\calZ_{\bullet}}, \scrE_{\bullet}\otimes\Omega_{Z_{\bullet}, \eta}^{\log, \bullet})$. We still denote by $\Phi$ the induced operator. 

\begin{Definition}\label{Definition: rigid cohomology with coefficients}
Let $\frakT$ be either $\frakS$, $\calO_{K_0}^{\log = 0}$, $\calO_{K_0}^{\log = \emptyset}$ or $\calO_{K}^{\log = \varpi}$ and let $\scrE\in \Isoc^{\normalfont \dagger}(X_0/\frakT)$. The \textbf{log rigid cohomology of $X_0$ relative to $\frakT$ with coefficients in $\scrE$} is defined to be the complex 
\[
    R\Gamma_{\rig}(X_0/\frakT, \scrE) := \hocolim R\Gamma(]Z_{\bullet}[_{\calZ_{\bullet}}, \scrE_{\bullet}\otimes\Omega_{Z_{\bullet}, \eta}^{\log, \bullet}),
\] where the homotopy colimit is taken over the category of hypercovers of $X_0$ built from the simplicial dagger spaces described above (see \cite[\href{https://stacks.math.columbia.edu/tag/01H1}{Tag 01H1}]{stacks-project}).\footnote{ Recall from Remark \ref{Remark: OC(X/T) as a sliced site} that we can view $X_0$ as a presheaf of the absolute log overconvergent site $\OC(k^{\log =0}/\frakT)$. Then, we can consider hypercovers of $X_0$ in $\OC(k^{\log = 0}/\frakT)$ (see, for example, \cite[Definition 2.1]{nlab:hypercover}). }
\end{Definition}

\begin{Remark}
\normalfont 
Similar notion apply to $(\scrE, \Phi)\in \FIsoc^{\dagger}(X_0/\frakT)$ when $\frakT \neq \calO_K^{\log = \varpi}$. We leave the detailed definition to the readers. 
\blackqed
\end{Remark}

\begin{Remark}
\normalfont 
Readers who are more familiar with the traditional definition of rigid cohomology (introduced by P. Berthelot in \cite{Berthelot}) may wonder how to compare Berthelot's definition with the aforementioned definition. For this, we refer the readers to \cite[Theorem 5.1]{GK-dagger}, \cite[Remark 4.2.5]{Kedlaya-finite}, and \cite[Corollary 3.6.8]{LS-overconvergent}.
\blackqed
\end{Remark}

\vspace{2mm}

We conclude this subsection by giving the definition of \emph{log rigid cohomology with compact supports} by following \cite{EY-Poincare}. To this end, recall that $X_0$ is strictly semistable, \emph{i.e}, Zariski locally, we have a strictly smooth morphism \[
    X_0 \rightarrow \Spec k[X_1, ..., X_n]/(X_1\cdots X_r)
\] of log schemes over $k^{\log = 0}$. Denote by $D_0 \subset X_0$ the horizontal divisor, locally defined by $X_{r+1}\cdots X_n$. 

Let $\frakT$ be either $\frakS$, $\calO_{K_0}^{\log = 0}$, $\calO_{K_0}^{\log = \emptyset}$ or $\calO_{K}^{\log = \varpi}$. Given $(Z, \frakZ, i, h, \theta)\in \OC(X_0/\frakT)$, we define the sheaf $\scrO_{\frakZ}(-D_0)$ on $\frakZ$ to be the locally free $\scrO_{\frakZ}$-module, locally generated by $i_*\theta^* X_{r+1}\cdots X_n$. This sheaf then induces a locally free $\scrO_{\calZ}$-module $\scrO_{\calZ}(-D_0)$ on the dagger generic fibre $\calZ$ of $\frakZ$. Consequently, by Lemma \ref{Lemma: realisation and structure sheaf}, there is a sheaf $\scrO_{X_0/\frakT}(-D_0)$ on $\OC(X_0/\frakT)$ whose realisation on each $(Z, \frakZ, i, h, \theta)$ is $\scrO_{\calZ}(-D_0)$.\footnote{ If $\scrM_{D_0}$ is the sheaf of monoid defined by the horizontal divisor $D_0$, then the sheaf $\scrO_{X_0/\frakT}(-D_0)$ is denoted by $\scrO_{X_0/\frakT}(\scrM_{D_0})$. We chose our notation since it resembles the similar notation used in traditional algebraic geometry.}

\begin{Definition}[$\text{\cite[Definition 3.3]{EY-Poincare}}$]\label{Definition: rigid cohomology with compact support}
Assume that $X_0$ is proper. Let $\frakT$ be either $\frakS$, $\calO_{K_0}^{\log = 0}$, $\calO_{K_0}^{\log = \emptyset}$ or $\calO_{K}^{\log = \varpi}$ and let $\scrE\in \Isoc^{\normalfont \dagger}(X_0/\frakT)$. Denote by $\scrE(-D_0)$ the tensor product $\scrE \otimes_{\scrO_{X_0/\frakT}}\scrO_{X_0/\frakT}(-D_0)$. Then, the \textbf{log rigid cohomology with compact support of $X_0$ relative to $\frakT$ with coefficients in $\scrE$} is defined to the complex 
\[
    R\Gamma_{\rig, c}(X_0/\frakT, \scrE) := R\Gamma_{\rig}(X_0/\frakT, \scrE(-D_0)).
\] 
\end{Definition}

\begin{Remark}
\normalfont 
When $\frakT \neq \calO_K^{\log = \varpi}$ and $(\scrE, \Phi)\in \FIsoc^{\dagger}(X_0/\frakT)$, then the \emph{log rigid cohomology with compact support of $X_0$ relative to $\frakT$ with coefficients in $(\scrE, \Phi)$} is defined in a similar manner. We, again, leave the detailed definition to the readers. 
\blackqed
\end{Remark}

\subsection{Hyodo--Kato theory with coefficients}\label{subsection: HK theory with coeff}
Following \cite{Yamada}, to discuss Hyodo--Kato theory with coefficients, we have to restrict the coefficients. This relies on the notion of \emph{residue maps}, which we now recall from \cite[Definition 2.3.9]{Kedlaya-semistable}: 

Let $\scrE\in \Isoc^{\dagger}(X_0/\calO_{K_0}^{\log = \emptyset})$. For any $(Z, \frakZ, i, h, \theta)\in \OC(X_0/\frakS)$, we can regard $(Z, \frakZ, i, h', \theta)\in \OC(X_0/\calO_{K_0}^{\log = \emptyset})$, where $h': \frakZ \xrightarrow{h} \frakS \rightarrow \calO_{K_0}^{\log = \emptyset}$. Zariski locally on $\frakZ$, we have a strictly smooth morphism \[
    \frakZ \rightarrow \Spwf\calO_{K_0}\llbrack T\rrbrack [X_1, ..., X_n]^{\dagger}/(T-X_1\cdots X_r).
\] Let $W$ be one of the $X_i$'s and let $\calD_W\subset \calZ$ be the closed dagger subspace defined by $W$. Denote by $\Omega_{\calZ/K_0}^1$ the sheaf of differential one-forms on $\calZ$ over $K_0$, then we have \[
    \coker\left( \scrE_{\calZ}\otimes (\Omega_{\calZ/K_0}^1\oplus \oplus_{X_i\neq W}\scrO_{\calZ}\dlog X_i) \rightarrow \scrE_{\calZ}\otimes \Omega_{\calZ/K_0}^{\log, 1}\right) = \scrE_{\calZ}\otimes \scrO_{\calD_W}\dlog W.
\] Hence, we have a map \[
    \scrE_{\calZ} \xrightarrow{\nabla}\scrE_{\calZ}\otimes \Omega_{\calZ/K_0}^{\log, 1} \rightarrow \scrE_{\calZ}\otimes \scrO_{\calD_W}\dlog W,
\] which induces a map $\scrE_{\calZ}\otimes \scrO_{\calD_W} \rightarrow \scrE_{\calZ} \otimes \scrO_{\calD_W}\dlog W$. After identifying $\scrE_{\calZ} \otimes \scrO_{\calD_W}\dlog W$ with $\scrE_{\calZ} \otimes \scrO_{\calD_W}$, one obtains the \textbf{\textit{residue map}} \[
    \Res_{W}: \scrE_{\calZ}\otimes \scrO_{\calD_W} \rightarrow \scrE_{\calZ}\otimes\scrO_{\calD_W}
\] along $\calD_W$.

\begin{Definition}\label{Definition: nilpotent residue and unipotent overconvergent isocrystals}
Let $\scrE\in \Isoc^{\normalfont \dagger}(X_0/\calO_{K_0}^{\log = \emptyset})$. \begin{enumerate}
    \item[(i)] We say $\scrE$ has \textbf{nilpotent residues}, if for any $(Z, \frakZ, i, h, \theta)\in \OC(X_0/\frakS)$, the residue maps $\Res_W$ on the realisation $\scrE_{\calZ}$ are nilpotent. We denote by $\Isoc^{\normalfont\dagger}(X_0/\calO_{K_0}^{\log = \emptyset})^{\nr}$ the full subcategory of overconvergent isocrystals having nilpotent residues. 
    \item[(ii)] Suppose $\scrE\in \Isoc^{\normalfont, \dagger}(X_0/\calO_{K_0}^{\log = \emptyset})^{\nr}$. We say $(\scrE, \Phi)$ is \textbf{unipotent} if $\scrE$ is an iterated extension of $\scrO_{X_0/\calO_{K_0}^{\log = \emptyset}}$. We denote by $\FIsoc^{\normalfont\dagger}(X_0/\calO_{K_0}^{\log = \emptyset})^{\unip}$ the full subcategory of unipotent overconvergent isocrystals. 
    \item[(iii)] If $(\scrE, \Phi)\in \FIsoc^{\normalfont \dagger}(X_0/\calO_{K_0}^{\log = \emptyset})$, we say it has \textbf{nilpotent residue} (resp., is \textbf{unipotent}) if $\scrE$ has nilpotent residue (resp., is unipotent). We let $\FIsoc^{\normalfont \dagger}(X_0/\calO_{K_0}^{\log = \emptyset})^{\nr}$ (resp., $\FIsoc^{\normalfont \dagger}(X_0/\calO_{K_0}^{\log = \emptyset})^{\unip}$) be the full subcategory of overconvergent $F$-isocrystals having nilpotent residues (resp., being unipotent.)
\end{enumerate}
\end{Definition}

Note that there is a commutative diagram \[
    \begin{tikzcd}
        \calO_{K_0}^{\log = 0}\arrow[r, "0\mapsfrom T"]\arrow[rd] & \frakS\arrow[d] & \calO_K^{\log = \varpi}\arrow[l, "T\mapsto \varpi"']\arrow[ld]\\ & \calO_{K_0}^{\log = \emptyset}
    \end{tikzcd}.
\] Therefore, via base change, we have functors of overconvergent isocrystals \[
    \begin{tikzcd}
        \Isoc^{\dagger}(X_0/\calO_{K_0}^{\log = 0}) & \Isoc^{\dagger}(X_0/\frakS)\arrow[r]\arrow[l] & \Isoc^{\dagger}(X_0/\calO_K^{\log = \varpi})\\
        & \Isoc^{\dagger}(X_0/\calO_{K_0}^{\log = \emptyset})\arrow[u]\arrow[ru]\arrow[lu]
    \end{tikzcd}.
\] In particular, for any $\scrE\in \Isoc^{\dagger}(X_0/\calO_{K_0}^{\log = \emptyset})$, we can view it as an overconvergent isocrystal in $\Isoc^{\dagger}(X_0/\calO_{K_0}^{\log = 0})$, $\Isoc^{\dagger}(X_0/\frakS)$, or $\Isoc^{\dagger}(X_0/\calO_K^{\log = \varpi})$. We shall abuse the notation and still denote its images by $\scrE$. Similar for overconvergent $F$-isocrystals. \\

To define Hyodo--Kato cohomology with coefficients, we fix a collection $\{(Z_{\lambda}, \frakZ_{\lambda}, i_{\lambda}, h_{\lambda}, \theta_{\lambda})\}_{\lambda\in \Lambda}$ where $(Z_{\lambda}, \frakZ_{\lambda}, i_{\lambda}, h_{\lambda}, \theta_{\lambda})\in \OC(X_0/\frakS)$ as in \S \ref{subsection: overconvergent F-isocrystals}. Resuming the notation in \emph{loc. cit.}, for any finite subset $\Xi\subset \Lambda$, we can consider the \textbf{\textit{Kim--Hain complex}} \[
    \Omega_{Z_{\Xi}, \eta}^{\log, \bullet}[u]_{\text{naïve}} := \Omega_{\calZ_{\Xi}/\calS}^{\log, \bullet}|_{]Z_{\Xi}[_{\calZ_{\Xi}}}[u],
\] generated by $\Omega_{Z_{\Xi}, \eta}^{\log, \bullet}$ and degree-0 elements $u^{[i]}$ (for $i\in \Z_{\geq 0}$) such that \begin{enumerate}
    \item[$\bullet$] $u^{[0]} = 1$,
    \item[$\bullet$] $u^{[i]}\wedge u^{[j]} = \frac{(i+j)!}{i!j!}u^{[i+j]}$, 
    \item[$\bullet$] $du^{[i+1]} = - (d\log T)u^{[i]} $.
\end{enumerate} 
We let $\Omega_{Z_{\Xi}, \eta}^{\log, \bullet}[u]_k$ be the subcomplex of $\Omega_{Z_{\Xi}, \eta}^{\log, \bullet}[u]_{\text{naïve}}$, consisting of sections of $\Omega_{Z_{\Xi}, \eta}^{\log, \bullet}$ and $u^{[0]}$, ..., $u^{[k]}$.

\begin{Definition}\label{Definition: Hyodo--Kato cohomology}
 Let $\scrE\in \Isoc^{\normalfont \dagger}(X_0/\calO_{K_0}^{\log = \emptyset})^{\unip}$. \begin{enumerate}
    \item[(i)] The \textbf{Hyodo--Kato cohomology of $X_0$ with coefficients in $\scrE$} is defined to be the complex \[
        R\Gamma_{\HK}(X_0, \scrE) := \hocolim \hocolim_k R\Gamma(]Z_{\bullet}[_{\calZ_{\bullet}}, \scrE_{\bullet}\otimes \Omega_{Z_{\bullet}, \eta}^{\log, \bullet}[u]_k),
    \]
    \item[(ii)] Suppose $X_0$ is proper. The \textbf{Hyodo--Kato cohomology of $X_0$ with compact supports and coefficients in $(\scrE, \Phi)$} is defined to be the complex \[
        R\Gamma_{\HK, c}(X_0, \scrE)) := R\Gamma_{\HK}(X_0/\calO_{K_0}^{\log = 0}, \scrE(D_0)).
    \]
\end{enumerate}
\end{Definition}

\begin{Remark}
\normalfont
We again leave the similar definition for $(\scrE, \Phi)\in \FIsoc^{\dagger}(X_0/\calO_{K_0}^{\log = \emptyset})^{\unip}$ to the readers. 
\blackqed
\end{Remark}

Let $\scrE\in \Isoc^{\dagger}(X_0/\calO_{K_0}^{\log = \emptyset})$. On both $R\Gamma_{\HK}(X_0, \scrE)$ and $R\Gamma_{\HK, c}(X_0, \scrE)$, there is a \textbf{\textit{monodromy operator}} $N$ defined by \[
    N(u^{[i]}) = u^{[i-1]}.
\] Moreover, if $(\scrE, \Phi)\in \FIsoc^{\dagger}(X_0/\calO_{K_0}^{\log = \emptyset})$, the complex $R\Gamma_{\HK}(X_0, (\scrE, \Phi))$ and $R\Gamma_{\HK, c}(X_0, (\scrE, \Phi))$ admit Frobenius actions induced by $\Phi$ and \[
    \varphi(u^{[i]}) = p^i u^{[i]}.
\] We again denote the Frobenius action on the complexes by $\Phi$. Therefore, one sees immediately that \begin{equation}\label{eq: relation of N and Frob}
    p \Phi N =  N \Phi.
\end{equation}\\

Our next goal is to briefly discuss the so-called \emph{Hyodo--Kato morphisms}. To this end, we assume from now on that there exists a proper flat log weak formal scheme $\frakX$ over $\calO_K$ which is strictly semistable over $\calO_K^{\log = \varpi}$ such that its dagger generic fibre $\calX$ and the associated scheme of $\frakX$ over $k$ is $X_0$. Zariski locally on $\frakX$, there is a strict smooth morphism \[
    \frakX \rightarrow \Spwf \calO_K[X_1, ..., X_n]^{\dagger}/(X_1\cdots X_r -\varpi). 
\] We denote by $\frakD$ the horizontal divisor locally defined by $X_{r+1}\cdots X_n$. Then, $\frakD$ induces the horizontal divisor $D_0$ on $X_0$ and the horizontal divisor $\calD$ on $\calX$.

Let $\iota: X_0 \hookrightarrow \frakX$ be the canonical inclusion and let $h_{k}: X_0 \rightarrow k^{\log  = 0}$ and $h_{\calO_K^{\log = \varpi}}: \frakX \rightarrow \calO_{K}^{\log = \varpi}$ be the structure morphisms, then $(X_0, \frakX, \iota, h, \id)\in \OC(X_0/\calO_{K}^{\log = \varpi})$. For any $\scrE\in \Isoc^{\dagger}(X_0/\calO_{K_0}^{\log = \emptyset})$, we see immediately from construction that \[
    R\Gamma_{\rig}(X_0/\calO_{K}^{\log = \varpi}, \scrE) = R\Gamma_{\log\dR}(\calX, \scrE) := \text{the log-de Rham cohomology of $\calX$ with coefficients in $\scrE$}
\] Moreover, since $\calX$ is assumed to be proper smooth, we see that \[
    R\Gamma_{\rig, c}(X_0/\calO_{K}^{\log = \varpi}, \scrE) = R\Gamma_{\log\dR,c}(\calX, \scrE) := R\Gamma_{\log\dR}(\calX, \scrE(-\calD)).
\]
Here, $\scrE(-\calD) = \scrE \otimes_{\scrO_{\calX}}\scrO_{\calX}(-\calD)$ is indeed the subsheaf of $\scrE$ that vanishes at $\calD$. 

\begin{Theorem}\label{Theorem: Hyodo--Kato isomorphism}
For any $(\scrE, \Phi)\in \FIsoc^{\normalfont \dagger}(X_0/\calO_{K_0}^{\log = \emptyset})^{\unip}$, we have the following quasi-isomorphisms: \begin{enumerate}
    \item[(i)] The comparison quasi-isomorphisms \[
        R\Gamma_{\HK}(X_0, (\scrE, \Phi)) \rightarrow R\Gamma_{\rig}(X_0, (\scrE, \Phi)), \quad u^{[i]} \mapsto \left\{\begin{array}{ll}
            1 & \text{if }i=0  \\
            0 & \text{else}
        \end{array}\right..
    \] It is moreover compatible with the Frobenius structures on both sides. 
    \item[(ii)] The Hyodo--Kato quas-isomorphism \[
        \Psi_{\HK}: R\Gamma_{\HK}(X_0, (\scrE, \Phi)) \otimes_{K_0}K \rightarrow R\Gamma_{\log\dR}(\calX, \scrE).
    \]
    \item[(iii)] The Hyodo--Kato quasi-isomorphism \[
        \Psi_{\HK, c}: R\Gamma_{\HK, c}(X_0, (\scrE, \Phi)) \otimes_{K_0}K \rightarrow R\Gamma_{\rig, c}(X_0/\calO_{K}^{\log = \varpi}, \scrE) = R\Gamma_{\log\dR,c}(\calX, \scrE).
    \]
\end{enumerate} 
\end{Theorem}
\begin{proof}
The first assertion is \cite[Theorem 4.8]{Yamada}. The second and the third quasi-isomorphism is given by a choice of $p$-adic logarithm and \[
    \Psi_{\HK}, \Psi_{\HK, c}: u^{[i]}\mapsto \frac{(-\log \varpi)^i}{i!}.
\] For the proof, see \cite[Proposition 8.8]{Yamada} and \cite[Corollary 3.8]{EY-Poincare}. 
\end{proof}

\begin{Remark}
    In a forthcoming work of K. Yamada, the condition on unipotence will be weaken to a more general class of log overconvergent $F$-isocrystals.
\end{Remark}

Recall that the log de Rham complex $R\Gamma_{\log\dR}(\calX)$ admits a filtration induced by  \[
    \sigma_{\geq i} (0\rightarrow \scrO_{\calX} \rightarrow \Omega_{\calX/K}^{\log, 1} \rightarrow \cdots),
\] where $\sigma_{\geq i}$ is the stupid truncation functor. Suppose $\scrG$ is a vector bundle on $\calX$ with integrable log connection. Assume that $\scrG$ admits a finite descending filtration $\Fil^i \scrG$. Recall that such a filtration satisfies \textbf{\textit{Griffiths's transversality}} if and only if \[
    \nabla(\Fil^i\scrG) \subset \Fil^{i-1}\scrG \otimes \Omega_{\calX/K}^{\log, 1}.
\] In this case, we have a filtration $F^nR\Gamma_{\log\dR}(\calX, \scrG)$ on $R\Gamma_{\log\dR}(\calX, \scrG)$ \[
    \Fil^n\left(0 \rightarrow \scrG \xrightarrow{\nabla} \scrG\otimes \Omega_{\calX/K}^{\log, 1}\rightarrow \cdots\right) := \left( 0\rightarrow \Fil^n \scrG \xrightarrow{\nabla} \Fil^{n-1}\scrG \otimes \Omega_{\calX/K}^{\log, 1}\rightarrow \cdots\right).
\] Together with the Hyodo--Kato isomorphism, they inspire the following definition. 

\begin{Definition}[$\text{\cite[Definition 9.6]{Yamada}}$]\label{Definition: syntomic coefficient}
The category of \textbf{syntomic coefficients} $\Syn(X_0, \frakX, \calX)$ is defined to be the category of triples $(\scrE, \Phi, \Fil^{\bullet})$, where \begin{enumerate}
    \item[$\bullet$] $(\scrE, \Phi)\in \FIsoc^{\normalfont \dagger}(X_0/\calO_{K_0}^{\log = \emptyset})^{\unip}$, 
    \item[$\bullet$] $\Fil^{\bullet}$ is a filtration of $\scrE$ on $\calX$ (i.e., after base change to $\OC(X_0/\calO_K^{\log = \varpi})$) that satisfies Griffiths's transversality. 
\end{enumerate} Morphisms are morphisms of overconvergent $F$-isocrystals that preserve the filtrations. 
\end{Definition}

\begin{Remark}\label{Remark: dual filtration}
\normalfont 
Given $(\scrE, \Phi, \Fil^{\bullet})\in \Syn(X_0, \frakX, \calX)$, we define its dual $(\scrE^{\vee}, \Phi^{\vee}, \Fil^{\vee, \bullet})\in \Syn(X_0, \frakX, \calX)$ as follows. By \cite[Proposition 3.6]{Yamada}, we know that we have a dual $(\scrE^{\vee}, \Phi^{\vee})\in \FIsoc^{\dagger}(X_0/\calO_{K_0}^{\log = \emptyset})^{\unip}$ with $\scrE^{\vee} = \sheafHom(\scrE, \scrO_{X_0/\calO_{K_0}^{\log = \emptyset}})$. Hence, on $\calX$, we have surjective morphisms \[
    \scrE^{\vee} \rightarrow (\Fil^i \scrE)^{\vee}
\] for any $i$. Hence,  we define \[
    \Fil^{\vee, -i+1}\scrE^{\vee} := \ker(\scrE^{\vee} \rightarrow (\Fil^i \scrE)^{\vee}).
\] One can easily check that this filtration satisfies Griffiths's transversality. Consequently, $(\scrE^{\vee}, \Phi^{\vee}, \Fil^{\vee, \bullet})\in \Syn(X_0, \frakX, \calX)$.
\blackqed
\end{Remark}
\section{Finite polynomial cohomology with coefficients}\label{section: FP cohomology}
The theory of finite polynomial cohomology was first developed by Besser in \cite{Besser-integral} and generalised to general varieties in \cite{BLZ-FPCoh}.
Such a theory resolves the problem that there is no Poincar\'{e} duality on syntomic cohomology.
The aim of this section is to generalise the aforementioned works to a cohomology theory with coefficients.

Throughout this section, we resume the notations used in \S \ref{section: HK cohomology} and \S \ref{subsection: HK theory with coeff}. In particular, we have a fixed proper weak formal scheme $\frakX$ over $\calO_K^{\log = \varpi}$, which is strictly semistable, with a proper smooth dagger generic fibre $\calX$ over $K$ and strictly semistable special fibre $X_0$ over $k^{\log = 0}$. We also recall the horizontal divisors $\frakD \subset \frakX$, $D_0 \subset X_0$, and $\calD \subset \calX$.

\subsection{Definitions and some basic properties}\label{subsection: main definition and first properties}
Following \cite{Besser-integral}, consider \begin{align*}
    \Poly & := \text{ the multiplicative monoid of all polynomials $P(X) = \prod_{i=1}^n(1-\alpha_i T)\in \Q[T]$ with constant term $1$}.
\end{align*} The following definition is inspired by \cite{BLZ-FPCoh}.

\begin{Definition}\label{Definition: syntomic P-cohomology with syntomic coefficients}
Let $(\scrE, \Phi, \Fil^{\bullet})\in \Syn(X_0, \frakX, \calX)$. \begin{enumerate}
    \item[(i)] Given $P\in \Poly$ and $n\in \Z$, the \textbf{syntomic $P$-cohomology of $\calX$ with coefficients in $\scrE$ twisted by $n$} is defined to be  \[
        R\Gamma_{\syn, P}(\calX, \scrE, n) := \left[\begin{tikzcd}
            \scalemath{0.7}{R\Gamma_{\HK}(X_0, (\scrE, \Phi))\otimes_{K_0}K} \arrow[rrr, "(P(\Phi^e)\otimes\id)\oplus (\Psi_{\HK}\otimes \id)"]\arrow[d, "N\otimes\id"'] &&& \scalemath{0.7}{R\Gamma_{\HK}(X_0, (\scrE, \Phi))\otimes_{K_0}K \oplus \DR(\calX, \scrE, n)} \arrow[d, "(N\otimes \id) \oplus 0"]\\
            \scalemath{0.7}{R\Gamma_{\HK}(X_0, (\scrE, \Phi))\otimes_{K_0}K} \arrow[rrr, "P(q\Phi^e)\otimes \id"] &&& \scalemath{0.7}{R\Gamma_{\HK}(X_0, (\scrE, \Phi))\otimes_{K_0}K}
        \end{tikzcd}\right],
    \] where \[
        \DR(\calX, \scrE, n) := R\Gamma_{\log\dR}(\calX, \scrE)/F^nR\Gamma_{\log\dR}(\calX, \scrE)
    \] and recall that $q = p^e = \# k$. The $i$-th cohomology group of $R\Gamma_{\syn, P}(\calX, \scrE, n)$ is then denoted by $H_{\syn, P}^i(\calX, \scrE, n)$.
    \item[(ii)] When $P = 1- q^{-n}T$, we write \[
        R\Gamma_{\syn}(\calX, \scrE, n) := R\Gamma_{\syn, 1-q^{-n}T}(\calX, \scrE, n)
    \] and call it the \textbf{syntomic cohomology of $\calX$ with coefficients in $\scrE$ twisted by $n$}. The $i$-th cohomology group of $R\Gamma_{\syn}(\calX, \scrE, n)$ is then denoted by $H_{\syn}^i(\calX, \scrE, n)$.
\end{enumerate}
\end{Definition}

For any $(\scrE, \Phi, \Fil^{\bullet})\in \Syn(X_0, \frakX, \calX)$, the $\scrO_{\calX}$-module $\scrE(-\calD)$ on $\calX$ also admits a filtration defined as \[
    \Fil^i\scrE(-\calD) := (\Fil^i\scrE) \otimes_{\scrO_{\calX}}\scrO_{\calX}(-\calD).
\] By definition, this filtration on $\scrE(-\calD)$ also satisfies Griffiths's transversality. It then consequently defines a filtration $F^nR\Gamma_{\log\dR,c}(\calX, \scrE)$ on $R\Gamma_{c, \log\dR}(\calX, \scrE)$. This observation then leads to the following definition.

\begin{Definition}\label{Definition: compactly supported syntomic P-cohomology}
Let $(\scrE, \Phi, \Fil^{\bullet})\in \Syn(X_0, \frakX, \calX)$. \begin{enumerate}
    \item[(i)] Given $P\in \Poly$ and $n\in \Z$, the \textbf{compactly supported syntomic $P$-cohomology of $\calX$ with coefficients in $\scrE$ twisted by $n$} is defined to be \[
        R\Gamma_{\syn, P, c}(\calX, \scrE, n) := \left[\begin{tikzcd}
            \scalemath{0.7}{R\Gamma_{\HK, c}(X_0, (\scrE, \Phi))\otimes_{K_0}K} \arrow[rrr, "(P(\Phi^e)\otimes\id)\oplus (\Psi_{\HK,c}\otimes \id)"]\arrow[d, "N\otimes\id"'] &&& \scalemath{0.7}{R\Gamma_{\HK, c}(X_0, (\scrE, \Phi))\otimes_{K_0}K \oplus \DR_c(\calX, \scrE, n)} \arrow[d, "(N\otimes \id) \oplus 0"]\\
            \scalemath{0.7}{R\Gamma_{\HK, c}(X_0, (\scrE, \Phi))\otimes_{K_0}K} \arrow[rrr, "P(q\Phi^e)\otimes \id"] &&& \scalemath{0.7}{R\Gamma_{\HK, c}(X_0, (\scrE, \Phi))\otimes_{K_0}K}
        \end{tikzcd}\right],
    \] 
    where \[
        \DR_c(\calX, \scrE, n) := R\Gamma_{\log\dR,c}(\calX, \scrE)/F^nR\Gamma_{\log\dR,c}(\calX, \scrE).
    \] The $i$-th cohomology group of $R\Gamma_{\syn, P, c}(\calX, \scrE, n)$ is denoted by $H_{\syn, P, c}^i(\calX, \scrE, n)$.
    \item[(ii)] When $P = 1-q^{-n}T$, we write \[
        R\Gamma_{\syn, c}(\calX, \scrE, n) := R\Gamma_{\syn, 1-q^{-n}T, c}(\calX, \scrE, n)
    \] and call it the \textbf{syntomic cohomology of $\calX$ with coefficients in $\scrE$ twisted by $n$}. The $i$-th cohomology of $R\Gamma_{\syn, c}(\calX, \scrE, n)$ is then denoted by $H_{\syn}^i(\calX, \scrE, n)$.
\end{enumerate}
\end{Definition}

\begin{Notation}
\normalfont For complexes $R\Gamma_{\HK}(X_0, (\scrE, \Phi))$, $R\Gamma_{\log\dR}(\calX, \scrE)$, $R\Gamma_{\syn, P}(\calX, \scrE, n)$, ... etc, we drop the `$\scrE$' in the notation if $\scrE$ is nothing but the structure sheaf. Similar notations apply to their cohomology groups. Moreover, we will often abuse the notations and write $P(\Phi^e)$ and $N$ for $P(\Phi^e)\otimes\id$ and $N\otimes\id$ on vector spaces over $K$.
\blackqed
\end{Notation}

\begin{Remark}\label{Remark: Hodge filtration on cohomology}
    \normalfont 
        The complexes $\DR(\calX, \scrE, n)$ (for various $n$) defines the Hodge filtration on $H_{\log\dR}^i(\calX, \scrE)$ as follows. The natural map \[
            R\Gamma_{\log\dR}(\calX, \scrE) \rightarrow \DR(\calX, \scrE, n)
        \] induces natural maps on cohomology \[
            H_{\log\dR}^i(\calX, \scrE) \rightarrow H^i(\DR(\calX, \scrE, n)). 
        \] We claim that this map is surjective. Observe that the surjectivity follows from the degeneration of the \emph{(log) Hodge--de Rham spectral sequence} at the $E_1$-page, hence we show such a degeneration. Note that $\scrE$ is unipotent, thus it sits inside a short exact sequence $0 \rightarrow \scrE' \rightarrow \scrE \rightarrow \scrE'' \rightarrow 0$, where $\scrE'$ and $\scrE''$ are both unipotent. This short exact sequence then gives rise to a distinguished triangle \[
            R\Gamma_{\log\dR}(\calX, \scrE') \rightarrow R\Gamma_{\log\dR}(\calX, \scrE) \rightarrow R\Gamma_{\log\dR}(\calX, \scrE'').
        \] 
        One observes that if the (log) Hodge--de Rham spectral sequences for the left and right complexes degenerate at the $E_1$-pages, then the same holds for the middle term. Therefore, one reduces to show such a degeneration for lower-rank $\scrE$'s. By induction, it is then enough to show such a degeneration for the trivial coefficient. However, this is \cite[Theorem 3.2]{EV-vanishing}.
        To conclude, we have the Hodge filtration defined by \[
            F^n H_{\log\dR}^i(\calX, \scrE) := \ker\left(H_{\log\dR}^i(\calX, \scrE) \rightarrow H^i(\DR(\calX, \scrE, n))\right).
        \]  Similar discussion also applies to the compact support cohomology. 
    \blackqed
\end{Remark}

\begin{Proposition}\label{Proposition: change of polynomial}
Given $(\scrE, \Phi, \Fil^{\bullet})\in \Syn(X_0, \frakX, \calX)$ and two polynomials $P, Q\in \Poly$, we have a natural morphism \[
    R\Gamma_{\syn, P}(\calX, \scrE, n) \rightarrow R\Gamma_{\syn, PQ}(\calX, \scrE, n).
\] In particular, if $P$ is divided by $1-q^{-n}T$, then there is a natural morphism \[
    R\Gamma_{\syn}(\calX, \scrE, n) \rightarrow R\Gamma_{\syn, P}(\calX, \scrE, n).
\] Similar statements hold for the compactly supported version.
\end{Proposition}
\begin{proof}
Observe that we have a commutative diagram \[
    \begin{tikzcd}[column sep = tiny]
            & \scalemath{0.7}{R\Gamma_{\HK}(X_0, (\scrE, \Phi))\otimes_{K_0}K} \arrow[ld, equal]\arrow[rrr, "(PQ(\Phi^e)\otimes\id) \oplus (\Psi_{\HK}\otimes \id)"]\arrow[dd] &&& \scalemath{0.7}{R\Gamma_{\HK}(X_0, (\scrE, \Phi))\otimes_{K_0}K \oplus \DR(\calX, \scrE, n)}\arrow[dd] \\
            \scalemath{0.7}{R\Gamma_{\HK}(X_0, (\scrE, \Phi))\otimes_{K_0}K} \arrow[rrr, crossing over,"(P(\Phi^e)\otimes\id)\oplus (\Psi_{\HK}\otimes \id)"]\arrow[dd, "N\otimes\id"'] &&& \scalemath{0.7}{R\Gamma_{\HK}(X_0, (\scrE, \Phi))\otimes_{K_0}K \oplus \DR(\calX, \scrE, n)}  \arrow[ru, "Q(\Phi^e)\oplus \id"']\\
            & \scalemath{0.7}{R\Gamma_{\HK}(X_0, (\scrE, \Phi))\otimes_{K_0}K} \arrow[ld, equal]\arrow[rrr] &&& \scalemath{0.7}{R\Gamma_{\HK}(X_0, (\scrE, \Phi))\otimes_{K_0}K \oplus \DR(\calX, \scrE, n)} \\
            \scalemath{0.7}{R\Gamma_{\HK}(X_0, (\scrE, \Phi))\otimes_{K_0}K} \arrow[rrr, "P(q\Phi^e)\otimes \id"] &&& \scalemath{0.7}{R\Gamma_{\HK}(X_0, (\scrE, \Phi))\otimes_{K_0}K}\arrow[ur, "Q(q\Phi^e)\oplus \id"']\arrow[from=uu, crossing over]
        \end{tikzcd},
\]
where the front face is the diagram defining $R\Gamma_{\syn, P}(\calX, \scrE, n)$ while the back face is the diagram defining $R\Gamma_{\syn, PQ}(\calX, \scrE, n)$.
The result then follows immediately.
\end{proof}

To understand the syntomic $P$-cohomology better, we begin with some simplification of notations. We let \begin{align*}
    C^{\bullet} & := \Cone \left(R\Gamma_{\HK}(X_0, (\scrE, \Phi))\otimes_{K_0}K \rightarrow  R\Gamma_{\HK}(X_0, (\scrE, \Phi))\otimes_{K_0}K \oplus \DR(\calX, \scrE, n)\right)[-1]\\
    D^{\bullet} & := \Cone\left(R\Gamma_{\HK}(X_0, (\scrE, \Phi))\otimes_{K_0}K \rightarrow R\Gamma_{\HK}(X_0, (\scrE, \Phi))\otimes_{K_0}K\right)[-1]
\end{align*} be the mapping fibres of the horizontal rows in the definition of $R\Gamma_{\syn, P}(\calX, \scrE, n)$. Then, by definition, we have \begin{align*}
    R\Gamma_{\syn, P}(\calX, \scrE, n) & = \Cone\left(C^{\bullet} \xrightarrow{\alpha} D^{\bullet}\right)[-1],
\end{align*} where $\alpha$ is the map induced by the vertical maps in the definitions of $R\Gamma_{\syn, P}(\calX, \scrE, n)$ and $R\Gamma_{\syn, P, c}(\calX, \scrE, n)$ respectively. Therefore, we have a long exact sequence \[
    \cdots \rightarrow H^{i-1}(C^{\bullet}) \xrightarrow{\alpha} H^{i-1}(D^{\bullet}) \xrightarrow{\beta} H_{\syn, P}^{i}(\calX, \scrE, n) \xrightarrow{\gamma} H^{i}(C^{\bullet}) \xrightarrow{\alpha} H^{i}(D^{\bullet}) \rightarrow\cdots.
\] 

\begin{Proposition}\label{Proposition: diagram to understand syntomic P-cohomology}
Let $(\scrE, \Phi, \Fil^{\bullet})\in \Syn(X_0, \frakX, \calX)$ and $P\in \Poly$. \begin{enumerate}
    \item[(i)] A class in $H_{\syn, P}^{i}(\calX, \scrE, n)$ is represented by a quintuple $(x,y,z,u,v)$ with \[
        \begin{array}{ll}
            x\in R\Gamma^i_{\HK}(X_0, (\scrE, \Phi))\otimes_{K_0}K, & u\in R\Gamma_{\HK}^{i-1}(X_0, (\scrE, \Phi))\otimes_{K_0}K,  \\
            y\in R\Gamma_{\HK}^{i-1}(X_0, (\scrE, \Phi))\otimes_{K_0}K, &  v\in R\Gamma_{\HK}^{i-2}(X_0, (\scrE, \Phi))\otimes_{K_0}K,\\
            z\in \DR^{i-1}(\calX, \scrE, n),
        \end{array}
    \] such that \[
        \begin{array}{ll}
            d_{\HK}\otimes\id(x) =0, & N\otimes\id(x) - d_{\HK}\otimes\id(u) = 0, \\
            P(\Phi^e)\otimes\id(x) + d_{\HK}\otimes\id(y) = 0, & N\otimes\id(y) + P(q\Phi^e)\otimes\id(u) + d_{\HK}\otimes\id(v) = 0.\\
            \Psi_{\HK}\otimes\id(x) + d_{\dR}(z)\in F^nR\Gamma_{\log\dR}^{i}(\calX, \scrE),
        \end{array}
    \] Here, $d_{\HK}$ and $d_{\dR}$ are differentials of $R\Gamma_{\HK}(X_0, (\scrE, \Phi))$ and $R\Gamma_{\log\dR}(\calX, \scrE)$ respectively.
    \item[(ii)] We have a diagram 
    \[
    \begin{tikzcd}
    & \substack{H_{\HK}^{i-1}(X_0, (\scrE, \Phi))^{P(\Phi^e) = 0, N=0} \otimes_{K_0}K  \\ \cap F^n H_{\log\dR}^{i-1}(\calX, \scrE)} \arrow[d] & &0 \arrow[d]\\
	&\scalemath{0.8}{\dfrac{H_{\HK}^{i-2}(X_0, (\scrE, \Phi))\otimes_{K_0} K}{\image N + \image P(q\Phi^e)}} \arrow[d] & &B^i \arrow[d] \\
	0 \arrow[r]  & \coker \alpha \arrow[d] \arrow[r, "\beta"] & \scalemath{0.8}{H_{\syn, P}^i(\calX, \scrE, n)} \arrow[r,"\gamma"] & \ker \alpha \arrow[r] \arrow[d] &0 \\
	& \scalemath{0.8}{\dfrac{ H_{\HK}^{i-1}(X_0, (\scrE, \Phi))^{P(q\Phi^e) = 0}\otimes_{K_0} K }{ N \left  (\substack{ H_{\HK}^{i-1}(X_0, (\scrE, \Phi))^{P(\Phi^e) = 0} \otimes_{K_0} K  \\ \cap F^n H_{\log\dR}^{i-1}(\calX, \scrE) } \right )}}  \arrow[d] & &\scalemath{1}{\substack{H_{\HK}^{i}(X_0, (\scrE, \Phi))^{P(\Phi^e)=0, N=0} \otimes_{K_0}K \\ \cap F^n H_{\log\dR}^i(\calX, \scrE)}} \arrow[d] \\
	&0 & &\scalemath{0.8}{\dfrac{ H_{\HK}^{i-1}(X_0, (\scrE, \Phi)) \otimes_{K_0} K }{\image N + \image P(q\Phi^e)}}
    \end{tikzcd}
    \] 
    where the middle row and two vertical columns are all exact. Here, \[
        B^i := \frac{ \left\{(x, y)\in H_{\HK}^{i-1}(X_0, (\scrE, \Phi))\otimes_{K_0} K \oplus \frac{H_{\log\dR}^{i-1}(\calX, \scrE)}{F^n H_{\log\dR}^{i-1}(\calX, \scrE)}: Nx \in \image P(q\Phi^e) \right\} }{ \left\{ (P(\Phi^e)x, \Psi_{\HK}(x)): x\in H_{\HK}^{i-1}(X_0, (\scrE, \Phi))\right\} }.
    \]
    \item[(iii)] A similar statements hold for the compactly supported version. 
\end{enumerate}
\end{Proposition}
\begin{proof}
Unwinding the construction of mapping fibres, one sees that $R\Gamma_{\syn, P}^i(\calX, \scrE, n)$ is equal to the direct sum \[
    \scalemath{0.85}{R\Gamma^i_{\HK}(X_0, (\scrE, \Phi))\otimes_{K_0}K \oplus R\Gamma^{i-1}_{\HK}(X_0, (\scrE, \Phi))\otimes_{K_0}K \oplus \DR^{i-1}(\calX, \scrE, n) \oplus R\Gamma^{i-1}_{\HK}(X_0, (\scrE, \Phi))\otimes_{K_0}K \oplus R\Gamma^{i-2}_{\HK}(X_0, (\scrE, \Phi))\otimes_{K_0}K}
\] with differentials given by \[
    d_{\syn, P} = \begin{pmatrix} d_{\HK}\\ -P(\Phi^e) & -d_{\HK}\\ -\Psi_{\HK} & & -d_{\dR}\\ N & & & -d_{\HK}\\ & N & & P(q\Phi^e) & d_{\HK}\end{pmatrix}.
\] When $\scrE$ is the trivial overconvergent $F$-isocrystal, this description is exact the one discussed in \cite[\S 2.4]{BLZ-FPCoh}.

For (ii), by definition, we have a commutative diagram \[
    \begin{tikzcd}[column sep = small]
        \cdots \arrow[r] & \scalemath{0.7}{H_{\HK}^{i-1}(X_0, (\scrE, \Phi))\otimes_{K_0}K } \arrow[r]\arrow[d, "N"] & \substack{\scalemath{0.7}{H_{\HK}^{i-1}(X_0, (\scrE, \Phi))\otimes_{K_0}K }\\ \scalemath{0.7}{ \oplus H^{i-1}(\DR(\calX, \scrE, n)) } }\arrow[r]\arrow[d, "N \oplus 0"] & \scalemath{0.7}{H^i(C^{\bullet}) } \arrow[d, "\alpha"] \arrow[r] & \scalemath{0.7}{H_{\HK}^{i}(X_0, (\scrE, \Phi))\otimes_{K_0}K } \arrow[r]\arrow[d, "N"] & \substack{ \scalemath{0.7}{H_{\HK}^{i}(X_0, (\scrE, \Phi))\otimes_{K_0}K}\\ \scalemath{0.7}{ \oplus H^{i}(\DR(\calX, \scrE, n)) }}\arrow[r]\arrow[d, "N \oplus 0"] & \cdots\\
        \cdots \arrow[r] & \scalemath{0.7}{H_{\HK}^{i-1}(X_0, (\scrE, \Phi))\otimes_{K_0} K } \arrow[r] & \scalemath{0.7}{H_{\HK}^{i-1}(X_0, (\scrE, \Phi))\otimes_{K_0} K } \arrow[r] & \scalemath{0.7}{H^i(D^{\bullet}) } \arrow[r] & \scalemath{0.7}{H_{\HK}^{i}(X_0, (\scrE, \Phi))\otimes_{K_0} K } \arrow[r] & \scalemath{0.7}{H_{\HK}^{i}(X_0, (\scrE, \Phi))\otimes_{K_0} K } \arrow[r] & \cdots
    \end{tikzcd},
    \] where the horizontal rows are exact sequences. One then deduces a commutative diagram \[
    \begin{tikzcd}
        \scalemath{0.8}{ 0 } \arrow[r] & \scalemath{0.8}{ \dfrac{ H_{\HK}^{i-1}(X_0, (\scrE, \Phi))\otimes_{K_0}K \oplus \frac{H_{\log\dR}^{i-1}(\calX, \scrE)}{ F^n H_{\log\dR}^{i-1}(\calX, \scrE)} }{ \{P(\Phi^e)x, \Psi_{\HK}(x): x\in H_{\HK}^{i-1}(X_0, (\scrE, \Phi)) \} } } \arrow[r] \arrow[d, "N \oplus 0"] & \scalemath{0.8}{ H^i(C^{\bullet}) } \arrow[r]\arrow[d, "\alpha"] & \scalemath{0.8}{ H_{\HK}^i(X_0, (\scrE, \Phi))^{P(\Phi^e)=0 }\otimes_{K_0}K \cap F^n H_{\log\dR}^{i}(\calX, \scrE) }\arrow[r]\arrow[d, "N"] & \scalemath{0.8}{ 0 }\\
        \scalemath{0.8}{ 0 } \arrow[r] & \scalemath{0.8}{ \dfrac{H_{\HK}^{i-1}(X_0, (\scrE, \Phi))\otimes_{K_0}K}{\image P(q\Phi^e)} }  \arrow[r] & \scalemath{0.8}{ H^i(D^{\bullet}) } \arrow[r] & \scalemath{0.8}{ H_{\HK}^{i}(X_0, (\scrE, \Phi))^{P(q\Phi^e) =0 }\otimes_{K_0}K } \arrow[r] & \scalemath{0.8}{ 0 }
    \end{tikzcd}.
\] 
The desired diagram can then be concluded directly by snake lemma. 
\end{proof}

\begin{Corollary}\label{Corollary: 3 step filtration}
Let $(\scrE, \Phi, \Fil^{\bullet})\in \Syn(X_0, \frakX, \calX)$ and $P\in \Poly$. For any $n\in \Z$ and any $i\in \Z_{\geq 0}$, there is a $3$-step descending filtration $F_P^{\bullet} = F_{\syn, P}^{\bullet}$ on $H_{\syn, P}^i(\calX, \scrE, n)$ such that \begin{align*}
    F_P^0/F_P^1 & = \ker\left( H_{\HK}^i(X_0, (\scrE, \Phi))^{P(\Phi^e)=0, N=0}\cap F^n H_{\log\dR}^i(\calX, \scrE) \rightarrow \frac{H_{\HK}^{i-1}(X_0, (\scrE, \Phi))\otimes_{K_0}K}{\image N + \image P(q\Phi^e)} \right)\\
     F_P^2 / F_P^3 & = \image\left( \frac{H_{\HK}^{i-2}(X_0, (\scrE, \Phi))\otimes_{K_0}K}{\image N + \image P(q\Phi^e)} \rightarrow \coker\alpha \xrightarrow{\beta} H_{\syn, P}^i(\calX, \scrE, n)\right)
\end{align*} and $F_P^1/F_P^2$ sits inside the short exact sequence \[
    0 \rightarrow \frac{H_{\HK}^{i-1}(X_0, (\scrE, \Phi))^{P(q\Phi^e)=0}\otimes_{K_0}K}{N\left( H_{\HK}^{i-1}(X_0, (\scrE, \Phi))^{P(\Phi^e)=0}\otimes_{K_0}K \cap F^n H_{\log\dR}^{i-1}(\calX, \scrE) \right)} \rightarrow F_P^1/F_P^2 \rightarrow B^i \rightarrow 0.
\] Similar for the compactly supported version, whose 3-step filtrations are denoted by $F_{P, c}^{\bullet} = F_{\syn, P, c}^{\bullet}$.
\end{Corollary}
\begin{proof}
The filtrations are defined as follows: \begin{align*}
    F_P^1 &:= \ker\left( H_{\syn, P}^{i}(\calX, \scrE, n) \xrightarrow{\gamma} \ker\alpha \rightarrow H_{\HK}^{i}(X_0, (\scrE, \Phi))^{P(\Phi^e) = 0, N=0}\otimes_{K_0}K \cap F^n H_{\log\dR}^{i}(\calX, \scrE)\right) = \gamma^{-1}(B^i),\\
    F_P^2 & := \image\left( \frac{H_{\HK}^{i-2}(X_0, (\scrE, \Phi))\otimes_{K_0}K}{\image N + \image P(q\Phi^e)} \rightarrow \coker\alpha \xrightarrow{\beta} H_{\syn, P}^i(\calX, \scrE, n)\right), \\
    F_P^3 & := 0.
\end{align*} 
The description of the graded pieces then follows from straightforward computation. 

We remark that it is an easy exercise in homological algebra showing that this filtration agrees with the natural $3$-step filtration on a mapping square. 
\end{proof}

\begin{Remark}\label{Remark: 3-step filtration in BLZ}
    \normalfont 
        When $\scrE$ is the trivial overconvergent $F$-isocrystal, it is not hard to see that our filtration agrees with the 3-step filtration given by \cite[Definition 2.3.3]{BLZ-FPCoh}.
    \blackqed
\end{Remark}

\begin{Corollary}\label{Corollary: special syntomic P-cohomology groups}
\begin{enumerate}
    \item[(i)] For any $(\scrE, \Phi, \Fil^{\bullet})\in \Syn(X_0, \frakX, \calX)$, $P\in \Poly$ and any small enough $n$, we have \begin{align*}
        H_{\syn, P}^0(\calX, \scrE, n) & \cong H_{\HK}^0(X_0, (\scrE, \Phi))^{P(\Phi^e) = 0, N=0}\otimes_{K_0}K\quad \text{ and }\\
        H_{\syn, P, c}^0(\calX, \scrE, n) & \cong H_{\HK,c}^0(X_0, (\scrE, \Phi))^{P(\Phi^e) = 0, N=0}\otimes_{K_0}K.
    \end{align*}
    \item[(ii)] Let $d= \dim \calX$. Suppose the polynomial $P$ satisfies the following conditions:\begin{itemize}
        \item $P(q\Phi^e)$ acts invertibly on $H_{\HK, c}^{2d-1}(X_0)$; 
        \item $P(\Phi^e)$ acts invertibly on $H_{\HK, c}^{2d}(X_0) \simeq K_0(-d)$. 
    \end{itemize} Then, \[
        H_{\syn, P, c}^{2d+1}(\calX, d+1) \cong K.
    \] 
\end{enumerate}
\end{Corollary}
\begin{proof}
Apply the diagram in Proposition \ref{Proposition: diagram to understand syntomic P-cohomology} to $i=0$, one easily obtains the first assertion. 

For the second statement, we again apply \emph{loc. cit.} for $i=2d+1$ and observe the following: \begin{itemize}
    \item $H_{\HK, c}^{2d}(X_0)\cong K_0(-d)$ with monodromy $N=0$;
    \item $B^{2d+1} = \frac{\left\{(x, y)\in K_0(-d)\otimes_{K_0}K \oplus \frac{H_{\log\dR, c}^{2d}(\calX)}{F^{d+1}H_{\log\dR, c}^{2d}(\calX)} \right\}}{ \left\{ (P(\Phi^e)x, \Psi_{\HK}(x)): x\in K_0(-d) \right\}} \cong K_0(-d)\otimes_{K_0}K$ under the map $(x, y)\mapsto x-P(\Phi^e)y$, due to the conditions on $P$ and $F^{d+1}H_{\log\dR, c}^{2d}(\calX) = 0$; 
    \item $H_{\HK, c}^{2d+1}(X_0)^{P(\Phi^e)=0, N=0}\otimes_{K_0} K \cap F^{d+1} H_{\log\dR, c}^{2d+1}(\calX) = 0$ due to degree reason; 
    \item $\frac{H_{\HK, c}^{2d-1}(X_0)\otimes_{K_0}K}{\image N + \image P(q\Phi^e)} = 0$ due to the conditions on $P$; 
    \item $\frac{H_{\HK, c}^{2d}(X_0)^{P(q\Phi^e)=0}\otimes_{K_0}K}{N\left( H_{\HK, c}^{2d}(X_0)^{P(\Phi^e)=0}\otimes_{K_0}K \cap F^{d+1}H_{\log\dR, c}^{2d}(\calX) \right)} = 0$ due to the conditions on $P$. 
\end{itemize}
Consequently, the diagram in Proposition \ref{Proposition: diagram to understand syntomic P-cohomology} reduces to \[
    \begin{tikzcd}
        && B^{2d+1}\cong K_0(-d)\otimes_{K_0}K\arrow[d, "\cong"]\\
        0 \arrow[r] & H_{\syn, P, c}^{2d+1}(\calX, d+1) \arrow[r] & \ker \alpha \arrow[r] & 0
    \end{tikzcd}
\] and the result follows. 
\end{proof}

\vspace{3mm}

Now, suppose further that $\frakX$ is smooth over $\calO_K$ with trivial log structure (and so $X_0$ and $\calX$ are smooth over $k$ and $K$, respectively, with trivial log structure). In this situation, $H_{\HK}^i = H_{\HK, c}^i$ agrees with the (non-log) rigid cohomology (with monodromy $N=0$, see \cite[Proposition 8.9]{Yamada}); $H_{\log\dR}^i = H_{\log, \dR, c}^i$ agrees with the (non-log) de Rham cohomology. To simplify the notation, we will keep using $H_{\HK}^i$ for the rigid cohomology but $H_{\dR}^i$ for the de Rham cohomology. 

Since $N=0$, the vertical maps in the mapping square in Definition \ref{Definition: syntomic P-cohomology with syntomic coefficients} are zero. Hence, instead of using the same definition, we make the following modifications in this situation. 

\begin{Definition}\label{Definition: syntomic P-cohomology with syntomic coefficients; trivial log structure}
    Let $(\scrE, \Phi, \Fil^{\bullet})\in \Syn(X_0, \frakX, \calX)$.
    \begin{enumerate}
        \item[(i)] Given $P\in \Poly$ and $n\in \Z$, the \textbf{syntomic $P$-cohomology of $\calX$ with coefficients in $\scrE$ twisted by $n$} is defined to be \[
            \scalemath{0.9}{R\Gamma_{\syn, P}(\calX, \scrE, n) := \Cone\left( R\Gamma_{\HK}(X_0, (\scrE, \Phi))\otimes_{K_0}K \xrightarrow{(P(\Phi^e)\otimes \id)\oplus (\Psi_{\HK}\otimes \id)} R\Gamma_{\HK}(X_0, (\scrE, \Phi))\otimes_{K_0} \oplus \DR(\calX, \scrE, n) \right)[-1]}.
        \] 
        \item[(ii)] When $P=1-q^nT$, we write \[
            R\Gamma_{\syn}(\calX, \scrE, n) := R\Gamma_{\syn, 1-q^{-n}T}(\calX, \scrE, n)
        \] and call it the \textbf{syntomic cohomology of $\calX$ with coefficients in $\scrE$ twisted by $n$}. 
    \end{enumerate}
\end{Definition}

\begin{Remark}\label{Remark: syntomic P-cohomology with trivial log structure}
    \normalfont 
        If one uses Definition \ref{Definition: syntomic P-cohomology with syntomic coefficients} in this situation, the resulting syntomic $P$-cohomology will be a direct sum $H^i(C^{\bullet})\oplus H^{i-1}(D^{\bullet})$, which is not the same as the cohomology given by Definition \ref{Definition: syntomic P-cohomology with syntomic coefficients; trivial log structure}. This reflects the fact that  admissible filtered $\phi$-modules are not closed under extension in the category of admissible filtered $(\phi, N)$-modules. 
        
        As a convention, when $\frakX$ is smooth over $\calO_K$ with trivial log structure, the syntomic $P$-cohomology for $\calX$ will always be referred as the one in Definition \ref{Definition: syntomic P-cohomology with syntomic coefficients; trivial log structure}. 
    \blackqed
\end{Remark}

\begin{Corollary}\label{Corollary: exact sequence when X is of good reduction}
Suppose $\frakX$ is smooth over $\calO_K$ with trivial log structure. Fix $i\in \Z_{\geq 0}$ and let $P\in \Poly$.
Then, we have a short exact sequence \[
    0 \rightarrow \frac{H_{\dR}^{i-1}(\calX, \scrE)}{P(\Phi^e)\left( F^n H_{\dR}^{i-1}(\calX, \scrE)\right)} \xrightarrow{i_{\fp}} H_{\syn, P}^i(\calX, \scrE, n) \xrightarrow{\pr_{\fp}} H_{\HK}^{i}(X_0, (\scrE, \Phi))^{P(\Phi^e)=0} \otimes_{K_0}K\cap F^nH_{\dR}^i(\calX, \scrE) \rightarrow 0.
\] In particular, if $\Fil^{\bullet} = (\scrE = \Fil^r\scrE \supset \cdots \supset \Fil^{\ell}\scrE \supset \Fil^{\ell+1}\scrE = 0)$ for some $r, \ell\in \Z$ with $\ell \geq r$ and $d= \dim \calX$, then \[
    H_{\syn, P}^0(\calX, \scrE, 0) \cong H_{\HK}^0(X_0, (\scrE, \Phi))^{P(\Phi^e)=0}\otimes_{K_0}K \quad \text{ and }\quad H_{\dR}^{2d}(\calX, \scrE) \cong H_{\syn, P}^{2d+1}(\calX, \scrE, \ell+d+1).
\]
\end{Corollary}
\begin{proof}
By definition, we see that there is a short exact sequence \[
    \scalemath{0.9}{0 \rightarrow \frac{H_{\HK}^{i-1}(X_0, (\scrE, \Phi))\otimes_{K_0}K \oplus \frac{H_{\dR}^{i-1}(\calX, \scrE)}{ F^n H_{\dR}^{i-1}(\calX, \scrE)}}{\{(P(\Phi^e)x, \Psi_{\HK}(x)): x\in H_{\HK}^{i-1}(X_0, (\scrE, \Phi))\}} \rightarrow H_{\syn, P}^i(\calX, \scrE, n) \rightarrow H_{\HK}^{i}(X_0, (\scrE, \Phi))^{P(\Phi^e)=0} \otimes_{K_0}K\cap F^nH_{\dR}^i(\calX, \scrE) \rightarrow 0}.
\]
The desired result then follows from the diagram in Proposition \ref{Proposition: diagram to understand syntomic P-cohomology} and the isomorphism \[
    \frac{H_{\HK}^{i-1}(X_0, (\scrE, \Phi))\otimes_{K_0}K \oplus \frac{H_{\dR}^{i-1}(\calX, \scrE)}{ F^n H_{\dR}^{i-1}(\calX, \scrE)}}{\{(P(\Phi^e)x, \Psi_{\HK}(x)): x\in H_{\HK}^{i-1}(X_0, (\scrE, \Phi))\}} \rightarrow \frac{H_{\dR}^{i-1}(\calX, \scrE)}{P(\Phi^e)(F^n H_{\dR}^{i-1}(\calX, \scrE))}, \quad (x, y)\mapsto x-P(\Phi^e)y.
\] Note that when $\scrE$ is the trivial overconvergent $F$-isocrystal, the short exact sequence is exactly the one in \cite[Proposition 2.5]{Besser-integral}.
\end{proof}

\subsection{Cup products and pushforward maps}\label{subsection: Cup products}
\begin{Proposition}\label{Proposition: cup products}
Let $(\scrE, \Phi_{\scrE}, \Fil^{\bullet}_{\scrE})$, $(\scrG, \Phi_{\scrG}, \Fil^{\bullet}_{\scrG})\in \Syn(X_0, \frakX, \calX)$, $P, Q\in \Poly$ and $n, m\in \Z$, we have a natural map \[
    R\Gamma_{\syn, P}(\calX, \scrE, n) \otimes R\Gamma_{\syn, Q, c}(\calX, \scrG, m) \rightarrow R\Gamma_{\syn, P*Q, c}(\calX, \scrE\otimes\scrG, n+m).
\] Here, for $P(X) = \prod_i (1-\alpha_i X)$, $Q(X) = \prod_j (1-\beta_j X)$, we set $P\ast Q(X) = \prod_{i, j}(1-\alpha_i\beta_j X)$.
\end{Proposition}
\begin{proof}
Note that we have natural maps induced by tensor products \begin{equation}\label{eq: cup products for dR and HK coh}
    \begin{array}{l}
        R\Gamma_{\log\dR}(\calX, \scrE) \otimes R\Gamma_{\log\dR, c}(\calX, \scrG) \rightarrow R\Gamma_{\log\dR, c}(\calX, \scrE\otimes\scrG)  \\
        R\Gamma_{\HK}(X_0, (\scrE, \Phi_{\scrE}))\otimes R\Gamma_{\HK, c}(X_0, (\scrG, \Phi_{\scrG})) \rightarrow R\Gamma_{\HK, c}(X_0, (\scrE\otimes\scrG, \Phi_{\scrE}\otimes\Phi_{\scrG})) 
    \end{array}.
\end{equation} Note also that \begin{align*}
    (\Fil^i_{\scrE} \scrE\otimes\Omega_{\calX/K}^{\log, s})\times (\Fil_{\scrG}^{j}\scrG\otimes \Omega_{\calX/K}^{\log, t}) & \rightarrow (\Fil^i_{\scrE}\scrE \otimes \Fil^j_{\scrG}\scrG)\otimes \Omega_{\calX/K}^{\log, s+t} \hookrightarrow (\Fil^{i+j}\scrE\otimes \scrG )\otimes\Omega_{\calX/K}^{\log, s+t}, \\
    ((e\otimes \omega), (f\otimes \theta)) & \mapsto (e\otimes f)\otimes \omega\wedge \theta
\end{align*} induces a map \[
    F^n R\Gamma_{\log\dR}(\calX, \scrE) \otimes F^m R\Gamma_{\log\dR, c}(\calX, \scrG) \rightarrow F^{n+m}R\Gamma_{\log\dR, c}(\calX, \scrE\otimes \scrG)
\] and so a map \[
    \DR(\calX, \scrE, n) \otimes \DR_c(\calX, \scrG, m) \rightarrow \DR_c(\calX, \scrE\otimes\scrG, n+m).
\]

On the other hand, for any $P, Q\in \Poly$, we have a commutative diagram  \[
    \begin{tikzcd}
        R\Gamma_{\HK}(X_0, (\scrE, \Phi_{\scrE})) \otimes R\Gamma_{\HK, c}(X_0, (\scrG, \Phi_{\scrG})) \arrow[r]\arrow[d, shift right = 20mm, "P(\Phi_{\scrE}^e)"']\arrow[d, shift left = 15mm, "Q(\Phi_{\scrG}^{e})"] & R\Gamma_{\HK, c}(X_0, (\scrE\otimes \scrG, \Phi_{\scrE}\otimes \Phi_{\scrG}))\arrow[d, "P\ast Q((\Phi_{\scrE}\otimes\Phi_{\scrG})^e)"]\\
        R\Gamma_{\HK}(X_0, (\scrE, \Phi_{\scrE})) \otimes R\Gamma_{\HK, c}(X_0, (\scrG, \Phi_{\scrG})) \arrow[r] & R\Gamma_{\HK, c}(X_0, (\scrE\otimes\scrG, \Phi_{\scrE}\otimes \Phi_{\scrG}))
    \end{tikzcd}.
\] This diagram is compatible with the monodromy operator $N$.

These two observations give rise to the following commutative diagram \begin{equation}\label{eq: big diagram defining the pairing}
    \begin{tikzcd}[column sep = tiny]
        & \scalemath{0.65}{ R\Gamma_{\HK,c}(X_0, (\scrE \otimes \scrG, \Phi_{\scrE}\otimes \Phi_{\scrG}))\otimes_{K_0}K } \arrow[rr]\arrow[dd] && \substack{\scalemath{0.65}{ R\Gamma_{\HK,c}(X_0, (\scrE \otimes \scrG, \Phi_{\scrE}\otimes \Phi_{\scrG}))\otimes_{K_0}K}\\ \scalemath{0.65}{\oplus \DR_c(\calX, \scrE\otimes \scrG, m+n) }}\arrow[dd] \\
        \substack{ \scalemath{0.65}{ R\Gamma_{\HK}(X_0, (\scrE, \Phi_{\scrE}))\otimes_{K_0}K }\\ \scalemath{0.65}{\otimes R\Gamma_{\HK, c}(X_0, (\scrG, \Phi_{\scrG}))\otimes_{K_0}K}} \arrow[ru]\arrow[rr, crossing over, "\left(\substack{P(\Phi_{\scrE}^e)\otimes \id\\ \otimes Q(\Phi_{\scrG}^e)\otimes \id}\right)\oplus \left(\substack{\Psi_{\HK}\otimes \id \\ \otimes \Phi_{\HK, c}\otimes \id}\right)"]\arrow[dd, "\substack{(N\otimes \id) \\ \otimes(N \otimes \id)}"'] && \left(\substack{ \scalemath{0.65}{ R\Gamma_{\HK}(X_0, (\scrE, \Phi_{\scrE}))\otimes_{K_0}K }\\ \scalemath{0.65}{\otimes R\Gamma_{\HK, c}(X_0, (\scrG, \Phi_{\scrG}))\otimes_{K_0}K}} \right)\scalemath{0.65}{\oplus} \left(\substack{ \scalemath{0.65}{ \DR(\calX, \scrE, n)} \\ \scalemath{0.65}{ \otimes \DR_c(\calX, \scrG, m) }} \right)\arrow[ru] \\
        & \scalemath{0.65}{ R\Gamma_{\HK, c}(X_0, (\scrE\otimes\scrG, \Phi_{\scrE}\otimes \Phi_{\scrG}))\otimes_{K_0}K } \arrow[rr] && \scalemath{0.65}{ R\Gamma_{\HK, c}(X_0, (\scrE\otimes\scrG, \Phi_{\scrE}\otimes \Phi_{\scrG}))\otimes_{K_0}K } \\
        \substack{ \scalemath{0.65}{ R\Gamma_{\HK}(X_0, (\scrE, \Phi_{\scrE}))\otimes_{K_0}K }\\ \scalemath{0.65}{\otimes R\Gamma_{\HK, c}(X_0, (\scrG, \Phi_{\scrG}))\otimes_{K_0}K}}\arrow[rr, "(P(q\Phi_{\scrE}^e)\otimes \id) \otimes (Q(q\Phi_{\scrG}^e)\otimes \id)"']\arrow[ru] && \substack{ \scalemath{0.65}{ R\Gamma_{\HK}(X_0, (\scrE, \Phi_{\scrE}))\otimes_{K_0}K }\\ \scalemath{0.65}{\otimes R\Gamma_{\HK, c}(X_0, (\scrG, \Phi_{\scrG}))\otimes_{K_0}K}}\arrow[from=uu, crossing over]\arrow[ru]
    \end{tikzcd},
\end{equation}
where \begin{itemize}
    \item the back face is the diagram defining $R\Gamma_{\syn, P*Q, c}(\calX, \scrE \otimes \scrG, n+m)$,
    \item morphisms from the front fact to the back face are the ones in the observations above.
\end{itemize}
In particular, the observations above yield a morphism \[
    \left[ \begin{array}{c}
        \text{front face }\\ \text{ of \eqref{eq: big diagram defining the pairing}}
    \end{array}\right] \rightarrow R\Gamma_{\syn, P*Q, c}(\calX, \scrE \otimes \scrG, n+m).
\]
Finally, observe that there is a natural projection \[
    R\Gamma_{\syn, P}(\calX, \scrE, n)\otimes R\Gamma_{\syn, Q, c}(\calX, \scrG, m) \rightarrow \left[ \begin{array}{c}
        \text{front face }\\ \text{ of \eqref{eq: big diagram defining the pairing}}
    \end{array}\right].
\]
The desired pairing is then given by the composition. 
\end{proof}

For any $(\scrE, \Phi, \Fil^{\bullet})\in \Syn(X_0, \frakX, \calX)$, the maps in \eqref{eq: cup products for dR and HK coh} induce pairings \begin{equation}\label{eq: pairing in Ertl--Yamada}
    \begin{array}{l}
        R\Gamma_{\log\dR}(\calX, \scrE) \otimes R\Gamma_{\log\dR, c}(\calX, \scrE^{\vee}) \rightarrow R\Gamma_{\log\dR, c}(\calX)  \\
        R\Gamma_{\HK}(X_0, (\scrE, \Phi)) \otimes R\Gamma_{\HK, c}(X_0, (\scrE^{\vee}, \Phi^{\vee})) \rightarrow R\Gamma_{\HK, c}(X_0) 
    \end{array}
\end{equation} which yield quasi-isomorphisms (see \cite[\S 6]{EY-Poincare}) \begin{equation}\label{eq: Poincare duality by Ertl--Yamada}
    \begin{array}{ll}
        R\Gamma_{\log\dR}(\calX, \scrE) \simeq R\Hom(R\Gamma_{\log\dR,c}(\calX, \scrE^{\vee}), K)[-2\dim \calX] & \text{ in }\sfD(\Mod_K)\\
        R\Gamma_{\HK}(X_0, (\scrE, \Phi)) \simeq R\Hom(R\Gamma_{\HK, c}(X_0, (\scrE^{\vee}, \Phi^{\vee})), K_0)[-2\dim \calX](-\dim \calX) & \text{ in }\sfD(\Mod_{K_0}(\varphi, N))
    \end{array},
\end{equation} where $\sfD(\Mod_K)$ is the derived category of $K$-vector spaces, $\sfD(\Mod_{K_0}(\varphi, N))$ is the derived category of $(\varphi, N)$-modules over $K_0$\footnote{ By a $(\varphi, N)$-modules over $K_0$ we mean a $K_0$-vector space equipped with a $\varphi$-semilinear action $\phi$ and a $K_0$-linear endomorphism $N$ such that $N\phi = p\phi N$.}, and $(-\dim \calX)$ is the Tate twist defined by multiplying $\varphi$ by $p^{\dim \calX}$.

\begin{Corollary}\label{Corollary: pairing for finite polynomial cohomology}
For any $P, Q\in \Poly$, there is a natural pairing \begin{equation}\label{eq: desired pairing for finite polynomial cohomology}
    R\Gamma_{\syn, P}(\calX, \scrE, n)\otimes R\Gamma_{\syn, Q, c}(\calX, \scrE^{\vee}, m) \rightarrow R\Gamma_{\syn, P*Q, c}(\calX, n+m),
\end{equation} which is compatible with the pairings in \eqref{eq: pairing in Ertl--Yamada} via the natural morphisms \[
    \begin{array}{l}
         R\Gamma_{\syn, P}(\calX, \scrE, n)\otimes R\Gamma_{\syn, Q, c}(\calX, \scrE^{\vee}, m) \rightarrow R\Gamma_{\HK}(X_0, (\scrE, \Phi))\otimes R\Gamma_{\HK, c}(X_0, (\scrE^{\vee}, \Phi^{\vee}))\\
         R\Gamma_{\dR}(\calX, \scrE)[-1]\otimes R\Gamma_{\dR}(\calX, \scrE^{\vee})[-1] \rightarrow R\Gamma_{\syn, P}(\calX, \scrE, n)\otimes R\Gamma_{\syn, Q, c}(\calX, \scrE^{\vee}, m) 
    \end{array}.
\] 
\end{Corollary}
\begin{proof}
This is an immediate consequence of Proposition \ref{Proposition: cup products}.
\end{proof}

\begin{Remark}\label{Remark: pairing for general semistable reduction}
\normalfont 
For $P, Q\in \Poly$ are chosen so that $P*Q$ satisfies the conditions in Corollary \ref{Corollary: special syntomic P-cohomology groups} (ii), we then have a pairing \[
    H_{\syn, P}^i(\calX, \scrE, n) \times H_{\syn, Q, c}^{2d-i+1}(\calX, \scrE^{\vee}, -n+d+1) \rightarrow H_{\syn, P*Q, c}^{2d+1}(\calX, d+1)\cong K,
\] where the last isomorphism is due to  Corollary \ref{Corollary: special syntomic P-cohomology groups}.  However, we do not know whether this pairing is perfect (compare with the case in Corollary \ref{Corollary: perfect pairing for good reduction case}). In fact, it seems to the authors that it is too optimistic to expect this pairing to be perfect. 
\blackqed
\end{Remark}

Our next task is to establish a proper pushforward map for syntomic $P$-cohomology groups. In \cite{Besser-integral}, such a map is constructed by using the perfect pairing on finite polynomial cohomology groups. In our situation, we are not able to prove the pairing introduced above is a perfect pairing. However, one may still establish a proper pushforward map in our case by using the perfect pairings on Hyodo--Kato cohomology groups and the perfect pairing on the de Rham cohomology groups. We begin with the following lemma. 

\begin{Lemma}\label{Lemma: Poincare duality with filtration}
    Let $\calX$ be a proper smooth dagger space over $K$ with log structure given by a normal crossing divisor. We write $d = \dim \calX$. Let $\scrE$ be a unipotent vector bundle on $\calX$ with integrable connection. Moreover, assume $\scrE$ admits a filtration $\Fil^{\bullet}$ that satisfies Griffiths's transversality. Then, for any degree $i$, we have a natural isomorphism \[
        F^nH_{\log, \dR}^i(\calX, \scrE^{\vee}) \cong \left( \frac{H_{\log, \dR, c}^{2d-i}(\calX, \scrE^{\vee})}{F^{d-n+1} H_{\log\dR, c}^{2d-i}(\calX, \scrE^{\vee})} \right)^{\vee},
    \] where the Hodge filtration on $H_{\log, \dR, c}^{2d-i}(\calX, \scrE^{\vee})$ is induced from the dual filtration $\Fil^{\vee, \bullet}$ defined in Remark \ref{Remark: dual filtration}.
\end{Lemma}
\begin{proof}
    Recall the pairing \[
        R\Gamma_{\log\dR}(\calX, \scrE) \otimes R\Gamma_{\log\dR, c}(\calX, \scrE^{\vee}) \rightarrow R\Gamma_{\log\dR, c}(\calX)
    \]
    is induced by the pairing \[
        \scrE\otimes\Omega_{\calX/K}^{\log, s} \times \scrE^{\vee}\otimes\Omega_{\calX/K}^{\log, t} \rightarrow \Omega_{\calX/K}^{\log, s+t}, \quad (f\otimes \omega, g\otimes\theta)\mapsto \langle f, g \rangle \omega\wedge \theta,
    \] where $\langle -, - \rangle$ is the natural pairing on $\scrE \times \scrE^{\vee}$. 

    Suppose $f\otimes\omega$ is a section of $\Fil^n\scrE \otimes\Omega_{\calX/K}^{\log, s}$ and $g\otimes \theta$ is a section of $\scrE^{\vee}\otimes\Omega_{\calX/K}^{\log, s}$, then \begin{align*}
        & \langle f, g \rangle\omega\wedge \theta = 0\\
        \Leftrightarrow \quad & g\text{ is a section of $\Fil^{\vee, -i+1}\scrE^{\vee}$ or $\omega\wedge \theta = 0$}\\
        \Leftrightarrow \quad & g\text{ is a section of $\Fil^{\vee, j}\scrE^{\vee}$ for some $j\geq -i+1$ or $\omega\wedge \theta = 0$}.
    \end{align*}
    Consequently, write $F^nR\Gamma_{\log\dR}(\calX, \scrE) \hookrightarrow R\Gamma_{\log\dR}(\calX, \scrE)$ for the natural inclusion, its dual projection $R\Gamma_{\log\dR, c}(\calX, \scrE^{\vee}) \twoheadrightarrow R\Hom(F^nR\Gamma_{\log\dR}(\calX, \scrE), K)[-2d]$ has kernel  \[
        \ker \left(R\Gamma_{\log\dR, c}(\calX, \scrE^{\vee}) \twoheadrightarrow R\Hom(F^nR\Gamma_{\log\dR}(\calX, \scrE), K)[-2d]\right) = F^{d-n+1}R\Gamma_{\log\dR, c}(\calX, \scrE^{\vee}).
    \]
    In particular, \[
        R\Hom(F^nR\Gamma_{\log\dR}(\calX, \scrE), K)[-2d] \simeq \DR_c(\calX, \scrE^{\vee}, d-n+1).
    \]
    The assertion then follows from taking cohomology (Remark \ref{Remark: Hodge filtration on cohomology}). 
\end{proof}

\begin{Proposition}\label{Proposition: proper pushforward}
Suppose $\frakZ$ is a flat strictly semistable weak formal scheme over $\calO_{K}^{\log = \varpi}$ together with an exact closed immersion $\iota: \frakZ \hookrightarrow \frakX$ of codimension $i$. We write $Z_0$ for its special fibre and $\calZ$ for its dagger generic fibre. Assume also that $\calZ$ is smooth over $K$. Let $(\scrE, \Phi, \Fil^{\bullet})\in \Syn(X_0, \frakX, \calX)$ such that  $\Fil^{\bullet} = (\scrE = \Fil^r\scrE \supset \cdots \supset \Fil^{\ell}\scrE \supset \Fil^{\ell+1}\scrE = 0)$ for some $r, \ell\in \Z$ with $\ell\geq r$.  Then, for any $P\in \Poly$ and any $j\in \Z_{\geq 0}$, we have a pushforward map \[
    \iota_* : H_{\syn, P}^{j}(\calZ, \iota^*\scrE, n) \rightarrow H_{\syn, P}^{j+2i}(\calX, \scrE, n+i).
\]
\end{Proposition}
\begin{proof}
    The proof follows from the following three steps.\\

    \noindent \textbf{Step 1.} We claim that there is a pushforward map \[
        \iota_*: R\Gamma_{\HK}(Z_0, (\iota^*\scrE, \Phi))[2i](i) \rightarrow R\Gamma_{\HK}(X_0, (\scrE, \Phi)). 
    \]  
    Consider the pullback map \[
        \iota^* :R\Gamma_{\HK, c}(X_0, (\scrE^{\vee}, \Phi^{\vee})) \rightarrow R\Gamma_{\HK, c}(Z_0, (\iota^*\scrE^{\vee}, \Phi^{\vee})).
    \] By Poincaré duality (\eqref{eq: Poincare duality by Ertl--Yamada}), this is the same as the pullback map \[
        \iota^*: R\Hom(R\Gamma_{\HK}(X_0, (\scrE, \Phi)), K_0)[-2d](-d) \rightarrow R\Hom(R\Gamma_{\HK}(Z_0, (\iota^*\scrE, \Phi)), K_0)[-2d+2i](-d+i).
    \] 
    By applying $R\Hom(-, K_0)[2d](d)$, we then get \[
        \iota_*:R\Gamma_{\HK}(Z_0, (\iota^*\scrE, \Phi))[2i](i) \rightarrow R\Gamma_{\HK}(X_0, (\scrE, \Phi))     
    \]
    as desired. \\
    
    \noindent \textbf{Step 2.} We claim that there is a pushforward map \[
        \iota_*:\DR(\calZ, \iota^*\scrE, n)[2i] \rightarrow \DR(\calX, \scrE, n+i).
    \]
    This pushforward map is obtained similarly. Indeed, consider the pullback map \[
        \iota^* : F^mR\Gamma_{\log\dR, c}(\calX, \scrE^{\vee}) \rightarrow F^mR\Gamma_{\log\dR, c}(\calZ, \iota^*\scrE^{\vee})
    \] 
    for any $m\in \Z$. By the proof of Lemma \ref{Lemma: Poincare duality with filtration}, we know that this is the same as \[
        \iota^*: R\Hom(\DR(\calX, \scrE, d-m+1), K)[-2d] \rightarrow R\Hom(\DR(\calZ, \iota^*\scrE, d-i-m+1), K)[-2d+2i].
    \] 
    By applying $R\Hom(-, K)[2d]$, one obtains \[
        \DR(\calZ, \iota^*\scrE, d-i-m+1)[2i] \rightarrow \DR(\calX, \scrE, d-m+1).   
    \]
    Since $m$ is arbitrary, the desired pushforward map follows. \\
    
    \noindent \textbf{Step 3.} Combining the previous steps and the definition of $R\Gamma_{\syn, P}$, one obtains a map \[
        R\Gamma_{\syn, P}(\calZ, \iota^*\scrE, n)[2i] \rightarrow R\Gamma_{\syn, P}(\calX, \scrE, n+i).
    \]
    The desired pushforward map then follows from taking cohomology groups. 
\end{proof}

\subsection{Finite polynomial cohomology}\label{subsection: finite polynomial cohomology}

Let $X$ be a proper smooth scheme over $\calO_K$ with trivial log structure, which is of relative dimension $d$. We write $X_0$ (resp., $\frakX$; resp., $\calX$) for its special fibre (resp., $\varpi$-adic weak completion; resp., dagger generic fibre). For each $i\in \Z_{\geq 0}$, we define \[
    Z^i(X) := \left\{Z: \text{ smooth irreducible closed subscheme of $X$ of codimension $i$}\right\}.
\] For each $Z\in Z^i(X)$, we again use the similar notation and denote by $Z_0$ (resp., $\frakZ$; resp., $\calZ$) its special fibre (resp., $\varpi$-adic weak completion; resp., dagger generic fibre). For each such $Z$, we write $\iota_{\frakZ}: \frakZ \hookrightarrow \frakX$ for the induced closed immersion from $\frakZ$ to $\frakX$.

\begin{Lemma}\label{Lemma: syntomic coefficients and purity}
Let $(\scrE, \Phi, \Fil^{\bullet})\in \Syn(X_0, \frakX, \calX)$ and we suppose furthermore that $(\scrE, \Phi)$ is an unipotent object in $\FIsoc^{\dagger}(X_0/\calO_{K_0})$, i.e., $\scrE$ is an iterated extension of the trivial overconvergent $F$-isocrystal $\scrO_{X_0/\calO_{K_0}}$. Let $i,j\in \Z_{\geq 0}$ and let $Z\in Z^i(X)$. Then, the Frobenius action on $H_{\HK}^j(Z_0, (\scrE, \Phi)) := H_{\HK}^j(Z_0, (\iota_{\frakZ}^* \scrE, \iota_{\frakZ}^*\Phi))$ has eigenvalues of Weil weight $j$. 
\end{Lemma}
\begin{proof}
Due to the assumption on $\scrE$, it suffices to show the statement for the trivial coefficient. However, this is a result in \cite{CLS}.
\end{proof}

\begin{Remark}\label{Remark: purity for rigid cohomology with unipotent F-isocrystal}
    \normalfont 
        Lemma \ref{Lemma: syntomic coefficients and purity} can easily be generalised, with little modification on the Weil weight, to those $(\scrE, \Phi)$ which are iterated extensions of $\scrO_{X_0/\calO_{K_0}}(r)$ (for some fixed $r\in \Z$), where $\scrO_{X_0/\calO_{K_0}}(r)$ is the Tate twist of $\scrO_{X_0/\calO_{K_0}}$ defined by multiplying $\varphi$ by $p^{-r}$. Such a purity statement breaks for general $(\scrE, \Phi)$.
    \blackqed
\end{Remark}  

Let \begin{align*}
    \Poly_j := & \text{ the submonoid of $\Poly$, consisting of polynomials of weight $j$,}\\
    & \text{ \emph{i.e.}, those having roots $\alpha$ with $|\alpha|_{\C} = q^{j/2}$ for every embedding of $\alpha$ in $\C$}.
\end{align*} For any $Z\in Z^i(X)$, any $n\in \Z$ and any $j\in \Z$, we write \begin{equation}\label{eq: finite-polynomial cohomology}
    H_{\fp}^j(\calZ, \scrE, n) := \varinjlim_{P\in \Poly_j} H_{\syn, P}^j(\calZ, \iota_{\frakZ}^*\scrE, n).
\end{equation} Then, we have the following immediate corollary. 

\begin{Corollary}\label{Corollary: short exact sequence of fp coh}
Let $(\scrE, \Phi, \Fil)\in \Syn(X_0, \frakX, \calX)$ such that $(\scrE, \Phi)$ is an unipotent object in $\FIsoc^{\normalfont \dagger}(X_0/\calO_{K_0})$. For any $Z\in Z^i(X)$, any $n\in \Z$ and any $j\in \Z$, we have a short exact sequence \[
    0\rightarrow \frac{H_{\dR}^{j-1}(\calZ, \scrE)}{F^n H_{\dR}^{j-1}(\calZ, \scrE)} \xrightarrow{i_{\fp}} H_{\fp}^j(\calZ, \scrE, n) \xrightarrow{\pr_{\fp}} F^n H_{\dR}^j(\calZ, \scrE) \rightarrow 0.
\] Here, $H_{\dR}^j(\calZ, \scrE) := H_{\dR}^j(\calZ, \iota_{\frakZ}^* \scrE)$.
\end{Corollary}
\begin{proof}
By Corollary \ref{Corollary: exact sequence when X is of good reduction}, we have a short exact sequence \[
    0 \rightarrow \frac{H_{\dR}^{j-1}(\calZ, \scrE)}{P(\Phi^e)\left( F^n H_{\dR}^{j-1}(\calZ, \scrE)\right)} \xrightarrow{i_{\fp}} H_{\syn, P}^i(\calZ, \scrE, n) \xrightarrow{\pr_{\fp}} H_{\HK}^{j}(Z_0, (\scrE, \Phi))^{P(\Phi^e)=0} \otimes_{K_0}K\cap F^nH_{\dR}^j(\calZ, \scrE) \rightarrow 0
\] for $P\in \Poly_j$. Note that $P(\Phi^e)$ gives an isomorphism \[
    P(\Phi^e): \frac{H_{\dR}^{j-1}(\calZ, \scrE)}{ F^n H_{\dR}^{j-1}(\calZ, \scrE)} \rightarrow \frac{H_{\dR}^{j-1}(\calZ, \scrE)}{P(\Phi^e)\left( F^n H_{\dR}^{j-1}(\calZ, \scrE)\right)}.
\] Moreover,  Lemma \ref{Lemma: syntomic coefficients and purity} implies that \[
    H_{\HK}^{j}(Z_0, (\scrE, \Phi))^{P(\Phi^e)=0} \otimes_{K_0}K\cap F^nH_{\dR}^j(\calZ, \scrE) = F^nH_{\dR}^j(\calZ, \scrE)
\] when $P$ runs through the cofinal set of polynomials divided by the characteristic polynomial of $\Phi^e$. One then easily conclude the result.   
\end{proof}

\begin{Corollary}\label{Corollary: perfect pairing for good reduction case}
Let $(\scrE, \Phi, \Fil^{\bullet})\in \Syn(X_0, \frakX, \calX)$ such that $(\scrE, \Phi)$ is an unipotent object in $\FIsoc^{\normalfont \dagger}(X_0/\calO_{K_0})$. 
Then, we have a perfect pairing \[
    H_{\fp}^i(\calX, \scrE, n)\times H_{\fp}^{2d-i+1}(\calX, \scrE^{\vee}, d-n+1) \rightarrow K. 
\]
\end{Corollary}
\begin{proof}
    By Corollary \ref{Corollary: short exact sequence of fp coh}, we have two short exact sequences \[
        0 \rightarrow \frac{H_{\dR}^{i-1}(\calX, \scrE)}{F^n H_{\dR}^{i-1}(\calX, \scrE)} \rightarrow H_{\fp}^i(\calX, \scrE, n) \rightarrow F^n H_{\dR}^i(\calX, \scrE) \rightarrow 0
    \] and 
    \[
        0 \rightarrow \frac{H_{\dR}^{2d-i}(\calX, \scrE^{\vee})}{F^{d-n+1} H_{\dR}^{2d-i}(\calX, \scrE^{\vee})} \rightarrow H_{\fp}^{2d-i+1}(\calX, \scrE, n) \rightarrow F^{d-n+1} H_{\dR}^{2d-i+1}(\calX, \scrE) \rightarrow 0.
    \]
    Taking dual to the latter short exact sequence gives rise to the short exact sequence \[
        0 \rightarrow \left(F^{d-n+1} H_{\dR}^{2d-i+1}(\calX, \scrE)\right)^{\vee} \rightarrow (H_{\fp}^{2d-i+1}(\calX, \scrE, n))^{\vee} \rightarrow \left(\frac{H_{\dR}^{2d-i}(\calX, \scrE^{\vee})}{F^{d-n+1} H_{\dR}^{2d-i}(\calX, \scrE^{\vee})}\right)^{\vee} \rightarrow 0.
    \] However, by Lemma \ref{Lemma: Poincare duality with filtration}, we have a commutative diagram \[
    \begin{tikzcd}[column sep = tiny]
        0 \arrow[r] & \dfrac{H_{\dR}^{i-1}(\calX, \scrE)}{F^n H_{\dR}^{i-1}(\calX, \scrE)} \arrow[r]\arrow[d, "\cong"] & H_{\fp}^i(\calX, \scrE, n)\arrow[r] & F^nH_{\dR}^i(\calX, \scrE) \arrow[r]\arrow[d, "\cong"] & 0\\
        0\arrow[r] & \left(F^{d-n+1} H_{\dR}^{2d-i+1}(\calX, \scrE^{\vee})\right)^{\vee}\arrow[r] & H_{\fp}^{2d-i+1}(\calX, \scrE^{\vee}, d-n+1)^{\vee}\arrow[r] & \left(\dfrac{H_{\dR}^{2d-i}(\calX, \scrE^{\vee})}{ F^{d-n+1} H_{\dR}^{2d-i}(\calX, \scrE^{\vee}) }\right)^{\vee}\arrow[r] & 0 
    \end{tikzcd},
\] which induces an isomorphism on the middle terms. The assertion then follows. 
\end{proof}
\section{The Abel--Jacobi map and \texorpdfstring{$p$}{p}-adic integration}\label{section: AJ map and integration}

The main goal of this section is to provide a construction of the Abel--Jacobi map.
Following Besser's strategy in \cite{Besser-integral}, we first show that finite polynomial cohomology with coefficients can be viewed as a generalisation of Coleman's integration theory for modules with connections in \cite[Section~10]{C-pShimura}.
Then we construct the Abel--Jacobi map in \S \ref{subsection: AJ map}.
As mentioned in the introduction, the idea is that the Abel--Jacobi map should admit an interpretation as certain kind of integration.
Since we wish to utilise the Coleman integration theory, we will restrict ourselves to varieties with good reduction in the aforementioned discussions. 
In the final subsection \S \ref{subsection: towards a semistable theory}, we shall briefly discuss the case for varieties with semistable reduction.

\subsection{A connection with Coleman's \texorpdfstring{$p$}{p}-adic integration}\label{subsection: Coleman's p-adic integration}
Suppose in this subsection that $X$ is a proper smooth curve over $\calO_K$ with trivial log structure. We also assume that $K = K_0$ for the convenience of the exposition. Suppose moreover that $(\scrE, \Phi, \Fil^{\bullet})\in \Syn(X_0, \frakX, \calX)$ is of trivial filtration, \emph{i.e.}, \[
    \Fil^i \scrE = \left\{\begin{array}{ll}
        \scrE, & i\leq 0 \\
        0, & i>0
    \end{array}\right.
\] 

For any affine open $Y \subset X$, we again similarly write $Y_0$ (resp., $\frakY$; resp., $\calY$) for its special fibre (resp., $\varpi$-adic weak completion; resp., dagger generic fibre). For any polynomial $P\in \Poly$, we similarly consider \[
    R\Gamma_{\syn, P}(\calY, \scrE, n) := \Cone\left(\scalemath{0.8}{R\Gamma_{\HK}(Y_0/K, (\scrE, \Phi))} \xrightarrow{P(\Phi^e)\oplus \Psi_{\HK}} \scalemath{0.8}{R\Gamma_{\HK}(Y_0/K, (\scrE, \Phi)) \oplus R\Gamma_{\dR}(\calY, \scrE)/F^nR\Gamma_{\dR}(\calY, \scrE)}\right)[-1].
\] Note that, since $X$ is smooth over $\calO_K$, we do not have the monodromy operator; also, because of the assumption $K = K_0$, the Hyodo--Kato map $\Psi_{\HK}$ is, in fact, the identity map. Immediately from the construction, we see that there is a restriction map \[
    R\Gamma_{\syn, P}(\calX, \scrE, n) \rightarrow R\Gamma_{\syn, P}(\calY, \scrE, n), \quad x \mapsto x|_{\calY}.
\] Our main results of this subsection reads as follows. 

\begin{Lemma}\label{Lemma: syntomic P-cohomology for affine piece}
Let $(\scrE, \Phi, \Fil^{\bullet})$ be as above. Then, for any affine open $Y \subset X$, we have an isomorphism \[
    H_{\syn, P}^1(\calY, \scrE, 1) \cong \left\{(f, \omega)\in \scrE(\calY) \oplus \left ( \scrE(\calY) \otimes_{\scrO_{\calY}(\calY)} \Omega_{\calY/K}^1 \right ): \nabla f = P(\Phi^e)\omega\right\}.
\]
\end{Lemma}
\begin{proof}
By the assumption on the filtration of $\scrE$ and the construction, the complex defining $R\Gamma_{\syn, P}(\calY, \scrE, 1)$ is given by \[
    \scrE(\calY) \xrightarrow{x \mapsto (-\nabla x, P(\Phi^e)x, x)} \left(\scrE(\calY) \otimes_{\scrO_{\calY}(\calY)}\Omega_{\calY/K}^1\right) \oplus \scrE(\calY) \oplus \scrE(\calY) \xrightarrow{ (x, y, z)\mapsto P(\Phi^e)x + \nabla y} \scrE(\calY) \otimes_{\scrO_{\calY}(\calY)}\Omega_{\calY/K}^1.
\] Therefore, \[
    H_{\syn, P}^1(\calY, \scrE, 1) = \frac{\left\{(x, y, z)\in \left(\scrE(\calY) \otimes_{\scrO_{\calY}(\calY)}\Omega_{\calY/K}^1\right) \oplus \scrE(\calY) \oplus \scrE(\calY): P(\Phi^e)x+\nabla y = 0 \right\}}{\left\{ (-\nabla x, P(\Phi^e)x, x): x\in \scrE(\calY)\right\}}.
\] However, one observes that there is an isomorphism \[
    H_{\syn, P}^1(\calY, \scrE, 1) \rightarrow \left\{(f, \omega)\in \scrE(\calY) \oplus \scrE(\calY) \otimes_{\scrO_{\calY}(\calY)} \Omega_{\calY/K}^1: \nabla f = \omega\right\}, \quad (x, y, z)\mapsto (-x - \nabla z, y - P(\Phi^e)z),
\] the assertion then follows. 
\end{proof}

To link our theory with Coleman integration, we have to introduce more terminologies. Let $\widehat{\calX}$ be the $\varpi$-adic completion of $\calX$, which is viewed as a rigid analytic space. Recall that there is a specialisation map \[
    \mathrm{sp}: |\widehat{\calX}| \rightarrow |X_0|, 
\] locally given by \[
    |\Spa(A, A^{\circ})| \rightarrow |\Spec A^{\circ}/\varpi|, \quad |\cdot|_x \mapsto \frakp_{x} := \{a\in A^{\circ}/\varpi: |\widetilde{a}|_x<1\text{ for some }\widetilde{a}\in A^{\circ}\text{ s.t. } \widetilde{a} \equiv a \mod \varpi\}.
\]Then, for any affine open $U_0 \subset X_0$, we define the tube \[
    ]U_0[_{\widehat{\calX}} := \text{ the interior of } \mathrm{sp}^{-1}(U_0).
\]
Note that any affine open $U_0\subset X_0$ is of the form  $U_0 = X_0 \smallsetminus \{ x_1, \ldots, x_t\}$ where $x_i$ are points in $X_0$. Thus,
$$ ]U_0[_{\widehat{\calX}} =\widehat{\calX} \smallsetminus \left(\bigcup_{j=1}^t D(\widetilde{x}_j, 1)\right) $$
where $\widetilde{x}_j \in \widehat{\calX}$ is a lift of $x_j$ and $D(\widetilde{x}_j, 1)$ is the disc of radius $1$ centred at $\widetilde{x}_j$, \emph{i.e.}, $D(\widetilde{x}_j, 1) = ]x_j[_{\widehat{\calX}}$.
An open subset $\calW \subset \widehat{\calX}$ of the form
$$\widehat{\calX} \smallsetminus \left( \bigcup_{j=1}^t D(\widetilde{x}_j, r_j)\right) $$
with $0 < r_j <1$ is called a \textit{\textbf{strict neighbourhood}} of $]U_0[_{\widehat{\calX}}$ (in $\widehat{\calX}$).

Moreover, if we start with $\scrE\in \Isoc^{\normalfont \dagger}(X_0/\calO_K)$, then its $\varpi$-adic completion $\widehat{\scrE}$ defines a coherent sheaf on $\widehat{\calX}$ with integrable connection \[
    \widehat{\scrE} \xrightarrow{\nabla} \widehat{\scrE} \otimes_{\scrO_{\widehat{\calX}}} \Omega_{\widehat{\calX}/K}^1.
\] One sees immediately that there is a natural map \[
    H_{\dR}^i(\calX, \scrE) \rightarrow H_{\dR}^i(\widehat{\calX}, \widehat{\scrE})
\] given by the $\varpi$-adic completion. However, since both spaces are finite-dimensional $K$-vector spaces, this morphism is, in fact, an isomorphism. 

\begin{Corollary}\label{Corollary: relation with Coleman integration}
Let $(\scrE, \Phi, \Fil^{\bullet})\in \Syn(X_0, \frakX, \calX)$ be as above and let $P\in \Poly$ be a polynomial such that $P(\Phi^e)$ annihilates $H_{\HK}^1(\calX, (\scrE, \Phi))$ and is an isomorphism on $H_{\HK}^0(\calX, (\scrE, \Phi))$.
So we have a projection \[
    pr_{\fp}: H_{\syn, P}^1(\calX, \scrE, 1) \twoheadrightarrow F^1H_{\dR}^1(\calX, \scrE) = \Gamma(\calX, \scrE\otimes_{\scrO_{\calX}}\Omega_{\calX/K}^1).
\]
Then, for any $\omega\in \Gamma(\calX, \scrE\otimes_{\scrO_{\calX}}\Omega_{\calX/K}^1)$, a lift $\widetilde{\omega}\in H_{\syn, P}^1(\calX, \scrE, 1)$ of $\omega$ can be viewed as an Coleman integral of $\omega$ in the following sense:

Let $Y_0 \subset X_0$ be an affine open subscheme.
Then the pullback of $\widetilde{\omega}$ in $H_{\syn, P}^1(]Y_0[_{\widehat{\calX}}, \scrE, 1)$ corresponds uniquely to a locally analytic section $F_\omega$ (see \cite[\S 10]{C-pShimura})
of ${\scrE}$ over some strict neighbourhood $\calW$ of $]Y_0[_{\widehat{\calX}}$ such that $\nabla F_{\omega} = \omega|_{\calW}$.
Moreover, if $\widetilde{\omega}'$ is another lift of $\omega$ that corresponds to the integral $F_{\omega}'$ on $\calW$, then 
\[
    F_{\omega} - F_{\omega}' \in {\scrE}(\calW)^{\nabla = 0}.
\]
\end{Corollary}
\begin{proof}
By Lemma \ref{Lemma: syntomic P-cohomology for affine piece}, we see that \[
    \widetilde{\omega}|_{\calY} = (f, \omega) \text{ with }\nabla f = P(\Phi^e)\omega.
\] Then, by applying \cite[Theorem 10.1]{C-pShimura}, we have the desired $F_{\omega}$.
\end{proof}

\begin{Remark}
\normalfont
Careful readers will find that there is a condition of $\scrE$ having \emph{regular singular annuli} in \cite[Theorem 10.1]{C-pShimura}. We remark that, since we start with $(\scrE, \Phi, \Fil^{\bullet})\in \Syn(X_0, \frakX, \calX)$, $\scrE$ is, in particular, unipotent. This then implies that Coleman's condition of $\scrE$ being regular singular is fulfilled. 
\blackqed
\end{Remark}






\subsection{The Abel--Jacobi map}\label{subsection: AJ map}

Let $X$ be a proper smooth scheme over $\calO_K$ with trivial log structure, which is of relative dimension $d$. We fix $(\scrE, \Phi, \Fil^{\bullet})\in \Syn(X_0, \frakX, \calX)$ such that \begin{itemize}
    \item $\Fil^{\bullet} = (\scrE = \Fil^0 \scrE \supset \cdots \Fil^{\ell}\scrE \supset \Fil^{\ell+1}\scrE = 0)$;\footnote{ The assumption on the filtration is unnecessary and can always be achieved by shifting the filtration on $\scrE$. We make this assumption to simplify the notations in the latter computations. }
    \item $(\scrE, \Phi)$ is unipotent object in $\FIsoc^{\dagger}(X_0/\calO_{K_0})$.\footnote{ This condition can be relaxed to those $(\scrE, \Phi)$ that is a Tate twist of a unipotent object of the category of overconvergent $F$-isocrystals. Consequently, the definition of the finite polynomial cohomology will have to be slightly modified. See Remark \ref{Remark: purity for rigid cohomology with unipotent F-isocrystal}.}
\end{itemize}  Consequently, for any $i\in \Z_{\geq 0}$, $Z\in Z^i(X)$, we have \[
    H_{\fp}^0(\calZ, \scrE, 0) \cong H_{\dR}^0(\calZ, \scrE).
\] 
We shall then always identify $H_{\fp}^0(\calZ, \scrE, 0)$ with $H_{\dR}^0(\calZ, \scrE)$ via this isomorphism.

Inspired by the classical definition of cycle class groups, we have the following definition. 

\begin{Definition}\label{Definition: de Rham cycle class group w.r.t E}
Let $(\scrE, \Phi, \Fil)$ be as above. Let $i \in \Z_{\geq 0}$. Then, the \textbf{(de Rham) cycle class group of codimension $i$ with respect to $\scrE$} is defined to be \[
    A^i(X, \scrE) := \bigoplus_{Z\in Z^i(X)} H_{\dR}^0(\calZ, \scrE).
\]
An element in $A^i(X, \scrE)$ will be written in the form $(\theta_Z)_Z$ or $\sum_Z \theta_Z \cdot Z$.
\end{Definition}

Note that, by applying Proposition \ref{Proposition: proper pushforward}, we have a pushforward map \[
    \iota_{\frakZ,*} : H_{\fp}^0(\calZ, \scrE) \rightarrow H_{\fp}^{2i}(\calX, \scrE, i).
\] It then induces a morphism \[
    \eta_{\fp}: A^i(X, \scrE) \rightarrow H_{\fp}^{2i}(\calX, \scrE, i), \quad (\theta_Z)_{Z}\mapsto \sum_{Z\in Z^i(X)} \iota_{\frakZ, *}\theta_Z.
\] Composing the map with the projection in the short exact sequence \[
    0 \rightarrow \frac{H_{\dR}^{2i-1}(\calX, \scrE)}{F^i H_{\dR}^{2i-1}(\calX, \scrE)} \xrightarrow{i_{\fp}} H_{\fp}^{2i}(\calX, \scrE, i) \xrightarrow{\pr_{\fp}} F^i H_{\dR}^{2i}(\calX, \scrE) \rightarrow 0,
\] we obtain a morphism \[
    \pr_{\fp}\circ \eta_{\fp}: A^i(X, \scrE) \rightarrow F^i H_{\dR}^{2i}(\calX, \scrE), \quad (\theta_Z)_Z \mapsto \sum_{Z\in Z^i(X)} \pr_{\fp}\left(\iota_{\frakZ, *}\theta_Z\right).
\]

\begin{Definition}\label{Definition: null-homologous classes}
Let $(\scrE, \Phi, \Fil^{\bullet})$ and $i$ be as above. A class $(\theta_Z)_Z\in A^i(X, \scrE)$ is said to be \textbf{null-homologous} (or \textbf{de Rham homologous to zero}) if $(\theta_Z)_Z\in \ker \pr_{\fp}\circ\eta_{\fp}$. Consequently, we define \[
    A^i(X, \scrE)_0 := \left\{(\theta_Z)_Z \in A^i(X, \scrE): (\theta_Z)_Z \text{ is null-homologous} \right\} = \ker\pr_{\fp}\circ \eta_{\fp}.
\]
\end{Definition}

\begin{Theorem}\label{Theorem: Abel--Jacobi map}
Let $(\scrE, \Phi, \Fil^{\bullet})$ and $i$ be as above. Then, there exists a natural morphism, the \textbf{finite polynomial Abel--Jacobi map} for $(\scrE, \Phi, \Fil^{\bullet})$, \[
    \AJ_{\fp} = \AJ_{\fp, \scrE}: A^i(X, \scrE)_0 \rightarrow \left(F^{d-i+1} H_{\dR}^{2d-2i+1}(\calX, \scrE^{\vee})\right)^{\vee}
\] such that for any $(\theta_Z)_Z\in A^i(X, \scrE)_0$ and any $\omega\in F^{d-i+1} H_{\dR}^{2d-2i+1}(\calX, \scrE^{\vee})$, \[
    \AJ_{\fp}((\theta_Z)_Z)(\omega) = \left(\int_Z \omega\right)(\theta_Z).
\] Here, \[
    \left(\int_Z \omega\right)(\theta_Z) := \langle \iota_{\frakZ}^*\widetilde{\omega}, \theta_Z \rangle_{\fp},
\] where $\widetilde{\omega}$ is a lift of $\omega$ in $H_{\fp}^{2d-2i+1}(\calX, \scrE^{\vee}, d-i+1)$ and $\langle \cdot, \cdot\rangle_{\fp}$ is the pairing induced by the perfect pairing in Corollary \ref{Corollary: perfect pairing for good reduction case}.
\end{Theorem}
\begin{proof}
Observe that \[
    \eta_{\fp}\left(A^i(X, \scrE)\right) \subset \ker \pr_{\fp} = \frac{H_{\dR}^{2i-1}(\calX, \scrE)}{ F^i H_{\dR}^{2i-1}(\calX, \scrE)}.
\] However, the Poincaré duality for de Rham cohomology implies that \[
    \frac{H_{\dR}^{2i-1}(\calX, \scrE)}{ F^i H_{\dR}^{2i-1}(\calX, \scrE)} \cong \left(F^{d-i+1} H_{\dR}^{2d-2i+1}(\calX, \scrE^{\vee})\right)^{\vee}. 
\] Therefore, we arrive at the desired map \[
    \AJ_{\fp}: A^i(X, \scrE)_0 \rightarrow \frac{H_{\dR}^{2i-1}(\calX, \scrE)}{ F^i H_{\dR}^{2i-1}(\calX, \scrE)} \cong \left(F^{d-i+1} H_{\dR}^{2d-2i+1}(\calX, \scrE^{\vee})\right)^{\vee}. 
\]

To show the formula, it is enough to look at one $Z\in Z^i(X)$ and any $\theta_Z\in H_{\dR}^0(\calZ, \scrE)$. For any $\omega\in F^{d-i+1} H_{\dR}^{2d-2i+1}(\calX, \scrE^{\vee})$, we choose a lift $\widetilde{\omega}\in H_{\fp}^{2d-2i+1}(\calX, \scrE^{\vee}, d-i+1)$ via the short exact sequence \[
    0 \rightarrow \frac{H_{\dR}^{2d-2i}(\calX, \scrE^{\vee})}{ F^{d-i+1} H_{\dR}^{2d-2i}(\calX, \scrE^{\vee})} \rightarrow H_{\fp}^{2d-2i+1}(\calX, \scrE^{\vee}, d-i+1) \rightarrow F^nH_{\dR}^{2d-2i+1}(\calX, \scrE^{\vee}) \rightarrow 0.
\] Then, we have \[
    \AJ_{\fp}(\theta_Z \cdot Z)(\omega) = \langle \iota_{\frakZ}^* \widetilde{\omega}, \theta_Z \rangle_{\fp, Z} = \langle \widetilde{\omega}, \iota_{\frakZ, *} \theta_Z\rangle_{\fp, X},
\] where the first equation follows from the compatibility in Corollary \ref{Corollary: pairing for finite polynomial cohomology}, the second equation follows from the definition of $\AJ_{\fp}$ and the final equation follows from the construction of the pushforward map. 
\end{proof}

\begin{Remark}\label{Remark: compare with the classical AJ}
\normalfont 
Comparing the Abel--Jacobi map with coefficients with the classical one (recalled in \S \ref{subsection: review of Besser's work}), one might find $A^i(\calX, \scrE)$ a bit ad hoc. However, one can check that the classical Abel--Jacobi map, in fact, factors as \[
    \AJ: A^i(X)_0 \rightarrow A^i(X, \scrO_{X_0/\calO_{K_0}})_0 \rightarrow \left( F^{d-i+1} H_{\dR}^{2d-2i+1}(\calX)\right)^{\vee}.
\] One then sees that our construction actually agrees with the classical one when the coefficient is trivial. 
\blackqed
\end{Remark}

\begin{Remark}
\normalfont 
It is clear from Corollary \ref{Corollary: relation with Coleman integration} that the lift $\widetilde{\omega}$ should be viewed as a (generalised) `Coleman integral'. This justifies the notation $\int_Z \omega$ in Theorem \ref{Theorem: Abel--Jacobi map}.
\blackqed
\end{Remark}

\subsection{Towards a semistable theory}\label{subsection: towards a semistable theory}
In this subsection, we would like to discuss a direction how one could generalise the finite polynomial Abel--Jacobi map to varieties with semistable reduction. Therefore, we fix a proper scheme $X$ over $\calO_K$, which is of relative dimension $d$. We assume its special fibre $X_0$ is strictly semistable and its $\varpi$-adic weak formal completion $\frakX$ has a proper smooth dagger generic fibre $\calX$.

For each $i\in \Z_{\geq 0}$, we define \[
    Z^i(X) := \left\{Z\hookrightarrow X \begin{array}{c}
        \text{ exact closed immersion}  \\
        \text{ of codimension $i$} 
    \end{array}: \begin{array}{l}
        \text{ $Z$ is strictly semistable over }\calO_K^{\log = \varpi}  \\
        \text{ the dagger generic fibre $\calZ$ of $Z$ is smooth over }K 
    \end{array} \right\}.
\] For each $Z\in Z^i(X)$, we similarly denote its special fibre (resp., $\varpi$-adic weak formal completion) by $Z_0$ (resp., $\frakZ$). Also, we write $\iota_{\frakZ} : \frakZ \hookrightarrow \frakX$ for the induced exact closed immersion. 

Let $(\scrE, \Phi, \Fil^{\bullet})\in \Syn(X_0, \frakX, \calX)$. We again assume $\Fil^{\bullet} = (\scrE = \Fil^0 \scrE \supset \cdots \supset \Fil^{\ell}\scrE \supset \Fil^{\ell + 1}\scrE = 0)$ to simplify the notation in computation. Hence, by Corollary \ref{Corollary: special syntomic P-cohomology groups}, we have \[
    H_{\syn, P}^0(\calZ, \scrE, 0) \cong H_{\HK}^0(Z_0, (\iota_{\frakZ}^* \scrE, \iota_{\frakZ}^* \Phi))^{P(\Phi^e) = 0, N=0}\otimes_{K_0}K
\] for any $Z\in Z^i(X)$ and any $P\in \Poly$. 

From now on, we fix a $P\in \Poly$. We similarly consider the \textbf{\textit{(Hyodo--Kato) cycle class group of codimension $i$ with respect to $P$ and $\scrE$}}, defined to be \[
    A^i(X, P, \scrE) := \bigoplus_{Z\in Z^i(X)} H_{\HK}^0(Z_0, (\iota_{\frakZ}^* \scrE, \iota_{\frakZ}^* \Phi))^{P(\Phi^e) = 0, N=0}\otimes_{K_0}K.
\] Then, via the pushforward map in Proposition \ref{Proposition: proper pushforward}, one obtains a map \[
    \eta_{\fp}: A^i(X, P, \scrE) \rightarrow H_{\syn, P}^{2i}(\calX, \scrE, i).
\] Recall from Proposition \ref{Proposition: diagram to understand syntomic P-cohomology} that we have a natural map \[
    \varrho_{\fp}: H_{\syn, P}^{2i}(\calX, \scrE, i) \rightarrow H_{\HK}^{2i}(X_0, (\scrE, \Phi))^{P(\Phi^e) = 0, N=0}\otimes_{K_0}K \cap F^i H_{\log\dR}^i(\calX, \scrE).
\] Composing $\eta_{\fp}$ with $\varrho_{\fp}$, we define the \textbf{\textit{null-homologous}} cycles to be elements in \[
    A^i(X, P, \scrE)_0 := \ker \varrho_{\fp}\circ \eta_{\fp}.
\] Consequently, by applying the 3-step filtration in Corollary \ref{Corollary: 3 step filtration}, we can define the \textbf{\textit{syntomic $P$-Abel--Jacobi map}} \[
    \AJ_{\syn, P}: A^i(X, P, \scrE)_0 \xrightarrow{\eta_{\fp}} F^1_P.
\]

To understand $\AJ_{\syn, P}$, suppose we can choose $Q\in \Poly$ such that $P*Q$ satisfies the conditions in Corollary \ref{Corollary: special syntomic P-cohomology groups} (ii). Then, the pairing in Remark \ref{Remark: pairing for general semistable reduction} yields a map \[
    H_{\syn, P}^{2i}(\calX, \scrE, i) \rightarrow \left(H_{\syn, Q, c}^{2d-2i+1}(\calX, \scrE^{\vee}, -i+d+1)\right)^{\vee} 
\] for any $m\in \Z$. Therefore, for any $\omega \in H_{\syn, Q, c}^{2d-2i+1}(\calX, \scrE^{\vee}, -i+d+1)$ and for any $(\theta_Z)_Z\in A^i(X, P, \scrE)_0$, we have \begin{equation}\label{eq: AJ map for semistable reduction}
    \AJ_{\syn, P}((\theta_Z)_Z)(\omega) = \sum_{Z\in Z^i(X)} \langle \iota_{\frakZ, *}\theta_Z, \omega\rangle = \sum_{Z\in Z^i(X)} \langle \theta_Z, \iota_{\frakZ}^* \omega \rangle,
\end{equation} where the pairing $\langle \cdot, \cdot \rangle$ is the pairing in Remark \ref{Remark: pairing for general semistable reduction}.

We finally remark that the formalism above is far from satisfactory for the following reasons: \begin{enumerate}
    \item[(i)]  Since we do not have a perfect pairing in this situation, equation \eqref{eq: AJ map for semistable reduction} does not uniquely determine the class $\AJ_{\syn, P}((\theta_Z)_Z)$.
    \item[(ii)] In practice, the cohomology classes that we are interested in usually come from the de Rham cohomology. However, we can only work out a formula for the syntomic $P$-cohomology groups. Of course, if one could choose $Q$ such that $Q(q^{-1}\Phi^{\vee, e})$ further acts isomorphically on $H_{\HK,c}^{2d-2i}(\calX, \scrE^{\vee}, -i+d+1)$, then we have a surjection \[
        H_{\syn, Q, c}^{2d-2i+1}(\calX, \scrE^{\vee}, -i+d+1) \twoheadrightarrow H_{\HK, c}^{2d-2i+1}(X_0, (\scrE^{\vee}, \Phi^{\vee}))^{Q(\Phi^{\vee, e}) = 0, N=0} \otimes_{K_0} K \cap F^{-i+d+1} H_{\log\dR, c}^{2d-2i+1}(\calX, \scrE^{\vee}).
    \] Hence, for any $\omega\in H_{\HK, c}^{2d-2i+1}(X_0, (\scrE^{\vee}, \Phi^{\vee}))^{Q(\Phi^{\vee, e}) = 0, N=0} \otimes_{K_0} K \cap F^{-i+d+1} H_{\log\dR,c}^{2d-2i+1}(\calX, \scrE^{\vee})$, one can choose a lift $\widetilde{\omega}\in H_{\syn, Q, c}^{2d-2i+1}(\calX, \scrE^{\vee}, -i+d+1)$. However, \eqref{eq: AJ map for semistable reduction} still suggests that the resulting calculation will still depend on the choice $\widetilde{\omega}$.
    \item[(iii)] Finally, it would be more pleasant to have a theory that does not depend on the choice of $P$ and $Q$. However, unlike in the good reduction case, we do not have a good control on the eigenvalues of the Frobenius actions on the cohomology groups. To the authors's knowledge, these eigenvalues could be pathological for general coefficients. Even for certain \emph{nice enough} coefficients, we still know little on how one can control the Frobenius actions (see, for example, \cite[Proposition 6.6]{Yamada}).
\end{enumerate}

\section{Applications}\label{section: Applications}
In this section, we aim to carry out an arithmetic application of the previous theory. More precisely, we would like to establish a relation between the finite polynomial Abel--Jacobi map and the arithmetic of compact Shimura curves over $\Q$. Our result is inspired by the works of Darmon--Rotger in \cite{DR} and Bertolini--Darmon--Prasanna in \cite{BDP}.

To establish such an arithmetic application, one immediately encounters the problem that, in the previous sections, we consider our coefficients in $\Syn(X_0, \frakX, \calX)$, which are required to be unipotent, while we do not know whether the sheaves that one often encounters in the theory of automorphic forms are unipotent. However, taking a closer look at our theory, one finds that the condition of unipotence can be loosened once the following conditions are satisfied: 

\begin{Conditions}
Suppose $X_0$, $\frakX$ and $\calX$ are as in \S \ref{section: HK cohomology}. Let $\Syn^{\sharp}(X_0, \frakX, \calX)$ be the category consisting of triples $(\scrE, \Fil^{\bullet}, \Phi)$, where $\scrE\in \Isoc^{\normalfont \dagger}(X_0/\calO_{K_0}^{\log = \emptyset})$ , $\Fil^{\bullet}$ be a filtration of $\scrE$ on $\calX$ that satisfies Griffiths's transversality, and $\Phi$ is a Frobenius action on $R\Gamma_{\HK}(X_0, \scrE)$.\footnote{ Here, we relax the condition a little: we do not require a Frobenius action on the sheaf level but only a Frobenius action on the complex. If we start with $(\scrE, \Phi)\in \FIsoc^{\normalfont \dagger}(X_0/\calX_{K_0}^{\log = \emptyset})$, then the condition is automatic. } 
\begin{enumerate}
    \item[$\bullet$] {\normalfont (HK-dR).} There exists an morphism \[
        \Psi_{\HK}: R\Gamma_{\HK}(X_0, \scrE) \rightarrow R\Gamma_{\dR}(\calX, \scrE)
    \] such that, after base change to $K$, $\Psi_{\HK}\otimes \id$ is a quasi-isomorphism.
    
    \item[$\bullet$] {\normalfont (Purity).} Suppose moreover that there is a smooth scheme $X$ over $\calO_K$ whose $\varpi$-adic weak completion is $\frakX$. Then, there exists $s\in \Z$ such that for every closed irreducible smooth subscheme $Z\subset X$, the characteristic polynomial of $\Phi$ on $H_{\HK}^i(Z_0, \scrE)$ is of pure Weil weight $s+i\in \Z$. 
    (If $\scrE$ is in fact unipotent, one only needs to check this condition on $X$.)
\end{enumerate} 
\end{Conditions} 

Remark that the hypothesis (HK-dR) allows us to consider the diagram in Definition \ref{Definition: syntomic P-cohomology with syntomic coefficients}. Note also that (HK-dR) holds when $\frakX$ is smooth and $K = K_0$ is an unramified finite extension of $\Q_p$ (by \cite[Proposition 8.9]{Yamada}). On the other hand, the hypothesis (Purity) allows us to consider \[
    H_{\fp}^j(\calZ, \scrE, n) := \varinjlim_{P\in \Poly_{s+j}} H_{\syn, P}^j(\calZ, \scrE, n)
\] as in \eqref{eq: finite-polynomial cohomology}. 
In this situation, the formalism in \S \ref{subsection: AJ map} immediately goes through.

\subsection{Heuristic on non-unipotent coefficients}\label{subsection: Heuristic on non-unipotent coefficients}
As we have remarked above, in the following two applications (\S \ref{subsection: diagonal cycle} and \S \ref{subsection: cycles; isogenies}), we do not know whether our sheaves are unipotent. For these sheaves, the filtrations in the targets of the associated Abel--Jacobi maps need to be modified.
We here propose a suitable notation and explain the numbering of the filtrations. 

For simplicity, assume $X$ is a proper variety over $\Q_p$. As explained in \cite{Nekovar-pAJmap}, the $p$-adic Abel--Jacobi map concerns the extension class of $\Q_p$ by a certain $G_{\Q_p}$-representation comes from (a piece of) \'{e}tale cohomology of $X$.
If $X$ is of good reduction, then its \'{e}tale cohomology is crystalline.
The $p$-adic Hodge theory provides a comparison between \'{e}tale cohomology and de Rham cohomology, the later lies in the category of filtered Frobenius modules (see for example \cite{Berger-IntropHdg} and the references therein).

To motivate our discussion for non-trivial coefficients, we begin with a brief summary for the trivial coefficient case. 
Recall, if $Z \subset X$ is a smooth irreducible subvariety of codimension $n$, we have an isomorphism $H_{\dR}^0(Z) \cong \Q_p$ and the Gysin sequence\[ 
    0 \rightarrow H^{2n-1}_{\dR}(X) \rightarrow H^{2n-1}_{\dR}(X \smallsetminus Z) \rightarrow H^{2n}_{\dR, Z}(X)(-n) \cong H^0_{\dR}(Z)(-n) \xrightarrow{\mathrm{cl}_X}  H^{2n}_{\dR}(X) \rightarrow \cdots .
\]
Suppose now $Z = \sum a_j Z_j \in A^n(X)_0$, the condition of being cohomologically trivial implies that the sum of images under $\mathrm{cl}_X$, namely $\sum a_j \mathrm{cl}_X (1_{Z_j})$, is zero.
Hence one gets an extension
$$ 0 \rightarrow H^{2n-1}_{\dR}(X) \rightarrow D \rightarrow \Q_p(-n) \rightarrow 0.$$

Let's now turn to non-trivial coefficients $\scrH$. Inspired by the trivial coefficient case, the question reduces to finding a trivial piece in $H^0(Z, \scrH)$.
This is not easy in general. However, when the sheaf $\scrH$ comes from a universal object over $X$, one can modify the method above. 
We will briefly explain a simple case below.

Suppose we have a commutative diagram of smooth connected $\Q_p$-schemes with good reductions \[\begin{tikzcd}
    E \arrow[d, "\pi"] &E_Z \arrow[l]  \arrow[d] \arrow[r, hookleftarrow] &D \\
    X \arrow[r, hookleftarrow] &Z
\end{tikzcd}
\]
where
\begin{enumerate}
    \item[$\bullet$] $\dim E = d_E$, $\dim X= 1$, and $\dim Z = 0$ .
    \item[$\bullet$] $\pi: E \rightarrow X$ is proper, smooth, and the square is given by pull-back.
    \item[$\bullet$] $D$ is of codimension $n$ in $E$, hence of codimension $n -1$ in $E_Z$.
    Moreover, we require that $D \in A^n(E)_0$. But $D \in A^{n-1}(E_Z)$ is not cohomologically trivial.\footnote{ Here, we abuse the notation, denoted by $D$ for both the cohomologically trivial cycle and its support. }
\end{enumerate}

Consider \[
    \scrH := \epsilon R^s\pi_* \Omega_{E/X}^{\bullet}
\] for some $s \in \mathbb{Z}$ and some idempotent $\epsilon \in \Q_p[\Aut(E/X)]$. We shall also assume $D$ is stable under $\epsilon$, \emph{i.e.}, $\epsilon D = D$.  Hence, one sees that $\scrH$ has pure weight $s$ and so $H_{\dR}^1(X, \scrH)$ is of pure weight $s+1$. Moreover, note that we have the relative Leray spectral sequence (see \cite[(3.3.0)]{Katz-nilpotent}) \[
    E_2^{ij} := H_{\dR}^{i}(X, R^j\pi_*\Omega_{E/X}^{\bullet}) \Rightarrow H_{\dR}^{i+j}(E).
\]
This allows us to view $H_{\dR}^1(X, \scrH)$ as a piece of $H_{\dR}^{s+1}(E)$ (or more precisely, $\epsilon H_{\dR}^{s+1}(E)$).

Now, for each irreducible component $D'$ of the codimension $n-1$ cycle $D$ in $E_Z$ gives a pushforward map 
$$H^0_{\dR}(D')(-n+1) \cong \Q_p(-n+1)\rightarrow H^{2n-2}_{\dR}(E_Z)$$
of filtered Frobenius modules, which provides a piece of $\mathbb{Q}_p$. 
Moreover, by setting $E' = E \smallsetminus E_Z$, we have the Gysin sequence \[
    \cdots \rightarrow H^{2n-3}_{\dR}(E_Z)(-1) \rightarrow H^{2n-1}_{\dR}(E)  \rightarrow H^{2n-1}_{\dR}(E') \rightarrow H^{2n-2}_{\dR}(E_Z)(-1) \rightarrow H^{2n}_{\dR}(E) \rightarrow \cdots .
\]

To relate $H^1_{\dR}(X, \scrH)$ with $H^{2n-1}_{\dR}(E)$, one sees it is inevitable that we need the condition $s = 2n-2$.
In particular, $s$ can only be even. Note that, by definition, we have $H^0_{\dR}(Z, \scrH) \hookrightarrow \epsilon H_{\dR}^{2n-2}(E_Z)$. 
We shall assume that we can choose $\epsilon$ so that \begin{enumerate}
    \item[$\bullet$] it  annihilates $H^j_{\dR}(E)$ for $j \neq 2n-1$; and 
    \item[$\bullet$] it gives rise an isomorphism $\epsilon H_{\dR}^{2n-1}(E) \cong H_{\dR}^1(X, \scrH)$ via the Leray spectral sequence.
\end{enumerate}
Consequently, one obtains a diagram \[
    \begin{tikzcd}[row sep = small]
        0 \arrow[r] & \epsilon H_{\dR}^{2n-1}(E) \arrow[r] &  \epsilon H_{\dR}^{2n-1}(E') \arrow[r] &  \epsilon H_{\dR}^{2n-2}(E_Z)(-1) \arrow[r] & \cdots\\
        &&& H^0_{\dR}(Z, \scrH)(-n) \arrow[u, hook]\\
        && \Q_p(-n) \arrow[r, hook]&  H_{\dR}^0(D)(-n)\arrow[u, hook]
    \end{tikzcd},
\]
where the inclusion $\Q_p (-n) \hookrightarrow H_{\dR}^0(D)(-n)$ is the unique line in $H_{\dR}^0(D)(-n)$ determined by the coefficients of the decomposition $D = \sum_j c_j D'_j$ into irreducible components.
One then obtains an extension by pullback
 \[
    0 \rightarrow  H_{\dR}^1(X, \scrH) \rightarrow F \rightarrow \mathbb{Q}_p(-n) \rightarrow 0.
\]
Via the isomorphism (\cite[Proposition 3.5]{BDP})
$$\Ext^1_{\mathrm{ffm}} (\mathbb{Q}_p(-n), H_{\dR}^1(X, \scrH)) \cong   H_{\dR}^1(X, \scrH)/ \Fil^n,$$
one may further interpret the extension $F$ as an Abel--Jacobi map similar to what we did in \S \ref{subsection: AJ map}.
Here, the $\Ext$-group is computed in the category of filtred Frobenious modules over $\Q_p$. 

Notice that the filtration number $n = \frac{s}{2} + \codim(Z,X)$ should be the `correct' one.
Such formula holds in a more general picture.
That is, 
for a sheaf $\scrH$ of pure weight $s =2r$, and a element $\Delta \in A^i(X, \scrH)_0$.
The correct Abel--Jacobi map will send it to 
$H^{2i-1}_{\dR}(X, \scrH) / \Fil^{\frac{s}{2} +i}$.
One can then identify it as $\Fil^j H^m_{\dR}(X, \scrH^\vee)$ under the correct duality.

\begin{Remark}
In some application, the extension class relies heavily on the fact that the sheaf $\scrH$ comes from a piece of the cohomology of a certain `universal' geometric object and one can construct an idempotent $\epsilon$ that kills unwanted degree of cohomology. 
This is the strategy appearing in \cite{BDP} and \cite{DR}. However, our approach suggests that one may work directly with the sheaves themselves instead of going back and forth between the universal geometric object and the sheaves. 

In general, one does not know whether there should be any trivial piece (\emph{i.e.}, $\Q_p$-line) in $H^0_{\dR}(Z, \scrH)$. 
But in case there is, we suggest this is the correct modification of the filtration.
\blackqed
\end{Remark}

\begin{Remark}
    One notices that there is the even-ness condition on the sheaf, \emph{i.e.}, the weight $s$ of $\scrH$ must be even.
    This phenomenon also  appears in the applications \S \ref{subsection: diagonal cycle} and \S \ref{subsection: cycles; isogenies} below (see also \cite{BDP} and \cite{DR}). 
    \blackqed
\end{Remark}

\subsection{Compact Shimura curves over \texorpdfstring{$\Q$}{Q}}\label{subsection: compact Shimura curves over Q}
In this subsection, we briefly recall the theory of compact Shimura curves over $\Q$ as a preparation of the next subsection. Our main references are \cite{DT, Buzzard-ShimuraCurve, Kassaei-phd} and readers are encouraged to consult \emph{op. cit.} for more detail discussions.

Let $D$ be an indefinite non-split quaternion algebra over $\Q$. We assume that $p$ does not divide the discriminant $\disc(D)$ of $D$. In particular, $D\otimes_{\Q}\Q_p \cong M_2(\Q_p)$. We fix a maximal order $\calO_D \subset D$ and an isomorphism $\calO_D\otimes_{\Z}\Z_p \simeq M_2(\Z_p)$. We further denote by $G$ the algebraic group over $\Z$, whose $R$-points are given by \[
    G(R) := (\calO_D \otimes_{\Z} R)^\times
\] for any ring $R$. We fix once and for all a neat compact open subgroup $\Gamma^p \subset G(\widehat{\Z}^p)$ such that $\Gamma^p = \prod_{\ell \neq p } \Gamma_{\ell}$ with $\Gamma_{\ell}\subset G(\Z_{\ell})$ being compact open and $\det \Gamma^p = \widehat{\Z}^{p, \times}$. We set $\Gamma := \Gamma^p G(\Z_p) = \Gamma^p \GL_2(\Z_p)$.

By a \textbf{\textit{false elliptic curve}} over a $\Z[1/\disc(D)]$-scheme $S$, we mean a pair $(A, i)$, where $A$ is an abelian surface over $S$ and $i: \calO_D \hookrightarrow \End_S(A)$. Note that on $(A, i)$, there is a unique principal polarisation (\cite[\S 1]{Buzzard-ShimuraCurve}). By the discussion of \S 2 of \emph{op. cit.}, we know that the functor \[
    \Sch_{\Z_p} \rightarrow \Sets, \quad S \mapsto  \left\{(A, i, \alpha): \begin{array}{l}
        (A, i)\text{ is a false elliptic curve over }S  \\
        \alpha \text{ is a $\Gamma$-level structure}
    \end{array}\right\}/\simeq 
\] is representable by a smooth projective scheme over $\Z_p$ of relative dimension $1$. We denote this curve by $X$. As usual, we write $X_0$ for the special fibre of $X$ over $\F_p$, $\frakX$ for the $p$-adic weak completion of $X$, and $\calX$ the dagger space associated with $\frakX$. To be consistant with respect to our discussions above, we remark that $X_0$, $\frakX$, and $\calX$ are equipped with the trivial log structure. 

Let $\pi: A^{\univ} \rightarrow X$ be the universal false elliptic curve over $X$. We consider three sheaves \[
    \widetilde{\underline{\omega}} := \pi_* \Omega_{A^{\univ}/X}^1, \quad \widetilde{\scrH} := R^1\pi_* \Omega_{A^{\univ}/X}^{\bullet}, \quad \text{ and }\quad \widetilde{\underline{\omega}}^{-1} := R^1\pi_* \scrO_{A^{\univ}}
\] with a canonical exact sequence \begin{equation}\label{eq: canonical exact sequence for H^1_dR}
    0 \rightarrow \widetilde{\underline{\omega}} \rightarrow \widetilde{\scrH} \rightarrow \widetilde{\underline{\omega}}^{-1}\rightarrow 0
\end{equation} given by the Hodge--de Rham spectral sequence. We choose once and for all an idempotent $\bfepsilon \in M_2(\Z_p) = \calO_D$ as in \cite[\S 4]{DT} (\emph{e.g.}, $\bfepsilon = \left(\substack{1 \,\,\, 0 \\ 0 \,\,\, 0}\right)$) and define \[
    \underline{\omega} := \bfepsilon \widetilde{\underline{\omega}}, \quad \scrH := \bfepsilon \widetilde{\scrH}, \quad \text{ and }\quad \underline{\omega}^{-1} := \bfepsilon \widetilde{\underline{\omega}}^{-1}.
\]

\begin{Lemma}\label{Lemma: basic properties of scrH}
The sheaves $\underline{\omega}$, $\scrH$ and $\underline{\omega}^{-1}$ enjoy the following properties. \begin{enumerate}
    \item[(i)] We have $\underline{\omega}^{-1} \cong \underline{\omega}^{\vee}$ and $\underline{\omega} \cong \underline{\omega}^{-1}\otimes_{\scrO_X}\Omega_{X/\Z_p}^1$.
    \item[(ii)] There exists an exact sequence \[
        0 \rightarrow \underline{\omega} \rightarrow \scrH \rightarrow \underline{\omega}^{-1} \rightarrow 0.
    \] In particular, $\scrH$ is of locally free of rank $2$ over $X$ with a filtration \[
        \Fil^i \scrH = \left\{ \begin{array}{ll}
            \scrH, & i\leq 0 \\
            \underline{\omega}, & i=1\\
            0, & i>1
        \end{array}\right. .
    \]
    \item[(iii)] The sheaf $\scrH$ is self-dual.
    \item[(iv)] There is an integrable connection \[
        \nabla: \scrH  \rightarrow \scrH\otimes_{\scrO_X} \Omega_{X/\Z_p}^1
    \] such that the filtration in (ii) satisfies Griffiths's transversality. 
\end{enumerate}
\end{Lemma}
\begin{proof}
The first assertion is nothing but \cite[Lemma 7]{DT}. For the second assertion, one applies $\bfepsilon$ to the exact sequence \eqref{eq: canonical exact sequence for H^1_dR}. For (iii), note that $\widetilde{\scrH}$ is self-dual, whose duality comes from the principal polarisation on $A^{\univ}$. Together with the fact that $\scrH$ is dual to $(1-\bfepsilon)\widetilde{\scrH}$ and these two subsheaves are isomorphic to each other (see, for example, \cite[\S 1.4]{Brasca-phd} or \cite[\S 1]{Brasca}), we conclude (iii). Finally, (iv) follows from that $\widetilde{\scrH}$ admits an integrable connection, \emph{i.e.}, the Gau\ss--Manin connection, that satisfies Griffiths's transversality. 
\end{proof}

For any $k\in \Z_{>0}$, we define \[
    \underline{\omega}^k:= \underline{\omega}^{\otimes k}\quad \text{ and }\quad \scrH^k := \Sym^k \scrH.
\] By Lemma \ref{Lemma: basic properties of scrH}, we see that \begin{enumerate}
    \item[$\bullet$] $\scrH^k$ admits a filtration of length $k$, \emph{i.e.}, \[
        \scrH^k = \Fil^0\scrH^k \supset \Fil^1 \scrH^{k} \supset \cdots \supset \Fil^k \scrH^k \supset \Fil^{k+1}\scrH^k = 0
    \] with \[
        \Fil^i\scrH^k /\Fil^{i+1}\scrH^k \cong \underline{\omega}^i \otimes_{\scrO_{X}} (\underline{\omega}^{-1})^{\otimes k-i}.
    \] 
    \item[$\bullet$] $\scrH^k$ admits an integrable connection which satisfies Griffiths's transversality (with respect to the filtration above). 
\end{enumerate} Similar computation as in \cite[Theorem 2.7]{Scholl-MFanddeRham} shows that the Hodge filtration on $H_{\dR}^1(X_{\Q_p}, \scrH^k)$ is given by  \[
    F^n H_{\dR}^1(X_{\Q_p}, \scrH^k) = \left\{\begin{array}{ll}
        H_{\dR}^1(X_{\Q_p}, \scrH^k), & n\leq 0 \\
        H^0(X_{\Q_p}, \underline{\omega}^k\otimes_{\scrO_{X_{\Q_p}}} \Omega_{X_{\Q_p}}^1) = H^0(X_{\Q_p}, \underline{\omega}^{k+2}), & n=1, ..., k+1\\
        0, & n>k+1
    \end{array}\right.,
\] where the equation $H^0(X_{\Q_p}, \underline{\omega}^k\otimes_{\scrO_{X_{\Q_p}}} \Omega_{X_{\Q_p}}^1) = H^0(X_{\Q_p}, \underline{\omega}^{k+2})$ follows from the Kodaira--Spencer isomorphism (see \cite[Corollary 3.2]{Kassaei-phd}).

\begin{Lemma}\label{PLemma: scrH gives a nice F-isocrystal}
Let $k\in \Z_{>0}$. The sheaf $\scrH^k$ induces an object in $\Syn^{\#}(X_0, \frakX, \calX)$, denoted by $(\scrH^k,  \Fil^{\bullet}, \Phi)$, such that (HK-dR) is satisfied.
\end{Lemma}
\begin{proof}
Note first that $\scrH^k$ induces a locally free $\scrO_{\calX}$-module on $\calX$, which we still denote by $\scrH^k$. We now explain how $\scrH^k$ defines a sheaf on $\OC(X_0/\Z_p)$. 

Observe first that $X_0$ and $\frakX$ defines an object in $\OC(X_0/\Z_p)$ by \[
    \begin{tikzcd}
        & X_0 \arrow[r]\arrow[ld, equal]\arrow[d] & \frakX\arrow[d]\\
        X_0 \arrow[r] & \Spec \F_p \arrow[r] & \Spwf \Z_p
    \end{tikzcd},
\]
where the vertical maps are the structure morphisms. We imply write $(X_0, \frakX)$ for this object and let $\OC(X_0/\Z_p)_{/(X_0, \frakX)}$ be the sliced site over $(X_0, \frakX)$. That is, objects in $\OC(X_0/\Z_p)_{/(X_0, \frakX)}$ are of the form \[
    \begin{tikzcd}
        Z \arrow[d, "f_{\F_p}"]\arrow[dd, bend right = 20, "h_{\F_p}"']\arrow[r, "i"] & \frakZ\arrow[d, "f_{\Z_p}"']\arrow[dd, bend left = 20, "h_{\Z_p}"]\\
        X_0 \arrow[r]\arrow[d] & \frakX\arrow[d]\\
        \Spec \F_p \arrow[r] & \Spwf\Z_p
    \end{tikzcd}.
\]
The map $f_{\Z_p}$ induces a natural map on the dagger generic fibre $f_{\eta}: \calZ \rightarrow \calX$.  Consequently, we can consider the sheaf on $\OC(X_0/\Z_p)_{/(X_0, \frakX)}$ \[
    (Z, \frakZ, i, , \theta) \mapsto f_{\eta}^*\scrH^k(\calZ),
\] which we still denote it by $\scrH^k$. 

Observe next that the natural inclusion $\OC(X_0/\Z_p)_{/(X_0, \frakX)} \rightarrow \OC(X_0/\Z_p)$ is a continuous functor and  it consequently defines a morphism of ringed sites \[
    \psi: \OC(X_0/\Z_p) \rightarrow \OC(X_0/\Z_p)_{/(X_0, \frakX)}. 
\] We then obtain a sheaf $\psi^* \scrH^k$ on $\OC(X_0/\Z_p)$. We shall again abuse the notation and denote it by $\scrH^k$.  One checks easily that $\scrH^k$ is an overconvergent isocrystal. By the discussion above, it is equipped with a natural filtration that satisfies Griffiths's transversality. 

Now we explain the Frobenius action on $R\Gamma_{\HK}(X_0, \scrH^k)$. Note that we can use a similar construction above to make $\widetilde{\scrH}^{\otimes k}$ into an overconvergent isocrystal and $\scrH^k$ is a quotient of $\widetilde{\scrH}^{\otimes k}$. However, by \cite[Proposition 5.2, Propsoition 5.3, Remark 5.6]{EY-Poincare}, the overconvergent cohomology of $X_0$ with coefficients in $\widetilde{\scrH}^{\otimes k}$ coincides with the crystalline cohomology of $X_0$ with coefficients in the crystalline realisation $\widetilde{\scrH}^{\otimes k}_{\mathrm{cris}}$ defined by $\widetilde{\scrH}^{\otimes k}$ (\cite[Proposition 5.3]{EY-Poincare}). On the other hand, $\widetilde{\scrH}$ is isomorphic to the relative crystalline cohomology of $A^{\univ}$ over $X$ (see \cite[Theorem 3.10]{Ogus-FIsocII}), thus $\widetilde{\scrH}^{\otimes k}_{\mathrm{cris}}$ admits an induced Frobenius action. By applying \cite[Lemma 5.14]{CI-semistable}, we conclude that the crystalline realisation $\scrH^k_{\mathrm{cris}}$ of $\scrH^k$ is a crystalline $F$-isocrystal. The desired Frobenius action on $R\Gamma_{\HK}(X_0, \scrH^k)$ then follows from the comparison of cohomology \[
    R\Gamma_{\HK}(X_0, \scrH^k) \simeq R\Gamma_{\mathrm{cris}}(X_0, \scrH^k_{\mathrm{cris}})
\] mentioned above. 

Finally, since $(X_0, \frakX)\in \OC(X_0/\Z_p)$ and so, as discussed in \S \ref{subsection: HK theory with coeff}, the condition (HK-dR) is satisfied. 
\end{proof}

\begin{Remark}\label{remark: abusing notations}
\normalfont 
In fact, one sees that the cohomology groups \[
    H_{\HK}^i(X_0, (\scrH^k, \Phi)), \quad H_{\dR}^i(\calX, \scrH^k), \quad H_{\dR}^i(X_{\Q_p}, \scrH^k)
\] are all isomorphic to each other. 
Therefore, in what follows, we shall abuse the notation and denote all of them simply by $H_{\dR}^i(X, \scrH^k)$. 
We also make no distinction between $X$, $X_{\Q_p}$, and $\calX$ by simply writing $X$ when it causes no confusion.
\blackqed
\end{Remark}

\begin{Lemma}\label{Lemma: Weil weight}
The eigenvalue of the Frobenius $\Phi$ acting on $H_{\dR}^i(X, \scrH^{k})$ is of Weil weight $i+k$. In particular, (Purity) is satisfied.
\end{Lemma}
\begin{proof}
When $k=0$, this is the case without coefficient and one just apply the result of \cite{CLS}. So, let's assume $k>0$.
Since $k>0$, $\scrH^k$ has no nontrivial global horizontal section (\cite[Lemma VII.4]{Brooks-phd}), one only need to consider $H_{\dR}^1(X, \scrH^{k})$. 

Note first that it is enough to show the claim for $\Sym^k \widetilde{\scrH}$ since $\scrH^k $ is a direct summand of $\Sym^k \widetilde{\scrH}$.
Let \[
    (A^{\univ})^k := \underbrace{A^{\univ}\times_X \cdots \times_X A^{\univ}}_{k \text{ times}}.
\]
According to \cite[\S 7.1]{Brooks-phd} (see also \cite[\S 2.1]{BDP}), there is an idempotent $\epsilon_k \in \Q[\Aut((A^{\univ})^k/X)]$ (constructed by averaging the actions of permutations and involutions) that yields\begin{itemize}
    \item $H_{\dR}^1(X, \Sym^k \widetilde{\scrH} ) \cong \epsilon_k H_{\dR}^{k+1}((A^{\univ})^k)$ canonically and 
    \item $\epsilon_k H_{\dR}^{j}((A^{\univ})^k) = 0$ for all $j \neq k+1$.
\end{itemize}
This then implies that the Frobenius action on $H_{\dR}^1(X, \Sym^k \widetilde{\scrH})$ is of pure Weil weight $k+1$. 

Finally, we have to know the Weil weight of $H_{\dR}^0(Z, \Sym^k \widetilde{\scrH})$ for all closed subschemes $Z\subset X$. In our case, those $Z$'s are nothing but closed points. Given such a $Z$, let $A_Z$ be the false elliptic curve associated with $Z$ given by the moduli interpretation of $X$. Note that $A_Z = A^{\univ}|_Z$. In particular, we have \[
    H^0_{\dR}(Z, \Sym^k \widetilde{\scrH}) \cong \Sym^k H_{\dR}^1(A_{Z}),
\] on which the Frobenius action has pure Weil weight $k$. 
\end{proof}

\subsection{Diagonal cycles and a formula of Darmon--Rotger}\label{subsection: diagonal cycle}


Our goal in this subsection is to achieve the formula in \cite[Theorem 3.8]{DR} in the case of compact Shimura curve over $\Q$ without the use of the Kuga--Sato variety. 

Let $(k, \ell, m)\in \Z^3$ be such that \begin{enumerate}
    \item[$\bullet$] $k+ \ell +m\in 2\Z$, 
    \item[$\bullet$] $2< k\leq \ell \leq m$ and $m < k+\ell$, 
    \item[$\bullet$] $\ell + m - k = 2t + 2$ for some $t\in \Z_{\geq 0}$. 
\end{enumerate}
We write \[
    r_1:= k-2, \quad r_2 := \ell-2, \quad r_3:=m-2, \quad r:=\frac{r_1+r_2+r_3}{2}
\] 
Note that we have $r_2 +r_3 \geq r_1 \geq r_3-r_2$ and $r-r_1 =t$.

We fix three elements \begin{align*}
    \eta & \in H_{\dR}^1(X, \scrH^{r_1})\\
    \omega_2 & \in H^0(X, \underline{\omega}^{\ell}) = F^{r_2+1} H_{\dR}^1(X, \scrH^{r_2})\\
    \omega_3 & \in H^0(X, \underline{\omega}^{m}) = F^{r_3+1}H_{\dR}^1(X, \scrH^{r_3})
\end{align*} and view \[ 
    \eta \otimes \omega_1 \otimes \omega_2 \in F^{r_2+r_3+2} H_{\dR}^3(X^3, \scrH^{r_1} \boxtimes \scrH^{r_2} \boxtimes  \scrH^{r_3}) 
\] by applying Künneth formula.
We recall the notation $\scrH^{r_1} \boxtimes \scrH^{r_2} \boxtimes  \scrH^{r_3} := \pi_1^* \scrH^{r_1} \otimes \pi_2^* \scrH^{r_2} \otimes \pi_3^* \scrH^{r_3}$, 
where $\pi_i : X^3 \rightarrow X$ is the projection of the $i$-th component.

On the other hand, by applying the short exact sequence in Corollary \ref{Corollary: short exact sequence of fp coh} and note that the projections $pr_{\fp}$'s are isomorphisms in the three cases above, we obtain the three corresponding elements in the finite polynomial cohomology\begin{align*}
    \widetilde{\eta} & \in H_{\fp}^1(X, \scrH^{r_1}, 0), \\
    \widetilde{\omega}_2 & \in H_{\fp}^1(X, \scrH^{r_2}, r_2+1),\\
    \widetilde{\omega}_3 & \in H_{\fp}^1(X, \scrH^{r_3}, r_3+1). 
\end{align*}
One further applies the cup product on finite polynomial cohomology and get \[
    \widetilde{\eta} \cup \widetilde{\omega}_2 \cup \widetilde{\omega}_3 \in H_{\fp}^3(X, \scrH^{r_1} \otimes \scrH^{r_2}\otimes \scrH^{r_3}, r_2+r_3+2).
\]

\begin{Lemma}\label{Lemma: decomposition of the tensor product}
As sheaves over $X$, we have a Frobenius-equivariant decomposition \[
    \scrH^{r_2} \otimes \scrH^{r_3} \cong \oplus_{j=0}^{r_2} \scrH^{r_3+r_2-2j}(-j).
\] In particular, $\scrH^{r_1}(-t)$ is a direct summand of $\scrH^{r_2} \otimes \scrH^{r_3}$ and we have a projection \[
    \pr_{r_1}: \scrH^{r_2} \otimes \scrH^{r_3} \rightarrow \scrH^{r_1}(-t).
\]
\end{Lemma}
\begin{proof}
The assertion follows from the fact that given a two-dimensional vector space $V$ over $\Q_p$, we have a decomposition \[
    \Sym^{r_2} V \otimes_{\Q_p} \Sym^{r_3} V \cong \bigoplus_{j=0}^{r_2} \Sym^{r_3+r_2-2j} V.
\] See, for example, \cite[Exercise 11.11]{Fulton--Harris}. We note that we have the twists in the statement so that the decomposition is Frobenius-equivariant. 
\end{proof}

\begin{Corollary}\label{Corollary: decomposition of triple tensor product}
Over $X$, we have a Frobenius-equivariant decomposition \[
    \scrH^{r_1} \otimes \scrH^{r_2} \otimes \scrH^{r_3} \cong \scrH^{r_1} \otimes \left(\bigoplus_{j=0}^{r_2} \scrH^{r_3+r_2-2j}(-j)\right).
\]
\end{Corollary}

Now, consider the triple product $X^3$. Fix a base point $o\in X$ and consider the following embeddings of $X$ into $X^3$ \[
    \begin{array}{rl}
        \iota_{123}: & X \xrightarrow{\Delta} X\times X \times X\\
        \iota_{12}: & X \xrightarrow{\Delta} X \times X \rightarrow X \times X \times \{o\} \hookrightarrow X\times X\times X\\
        \iota_{13}: & X \xrightarrow{\Delta} X \times  X \rightarrow X \times \{o\} \times X \hookrightarrow X\times X\times X\\
        \iota_{23}: & X \xrightarrow{\Delta} X \times X \rightarrow \{o\} \times X \times X \hookrightarrow X\times X\times X\\
        \iota_1: & X \rightarrow X \times \{o\} \times \{o\} \hookrightarrow X\times X \times X\\
        \iota_2: & X \rightarrow \{o\} \times X \times \{o\} \hookrightarrow X\times X \times X\\
        \iota_3: & X \rightarrow \{o\} \times \{o\} \times X \hookrightarrow X\times X \times X
    \end{array}.
\]  We also consider the embeddings \begin{align*}
    \iota_{23}' & : X \xrightarrow{\Delta} X\times X\\
    \iota_2' & : X \rightarrow X \times \{o\} \hookrightarrow X\times X\\
    \iota_3' & : X \rightarrow \{o\} \times X \hookrightarrow X \times X
\end{align*} which can be viewed as cutting off the first component of the maps $\iota_{123}$, $\iota_{12}$ and $\iota_{13}$.

Now we let $X_{123} := \iota_{123}(X) \in Z^2(X^3)$.
For any $\tau \in H_{\fp}^3(X_{123}, \scrH^{r_1} \boxtimes \scrH^{r_2}\boxtimes \scrH^{r_3}, r_2+r_3+2)$, we write $\tau = \sum_{i=1}^n \tau_1^{(i)}\otimes \tau_2^{(i)} \otimes \tau_3^{(i)}$, where $\tau_j^{(i)}$ corresponds to the $\scrH^{r_j}$-component. 
Applying the perfect pairing \[
    \langle \cdot, \cdot\rangle_{\fp}: H_{\fp}^3(X_{123}, \scrH^{r_1}\boxtimes \scrH^{r_2}\boxtimes \scrH^{r_3}, r_2+r_3+2) \times H_{\fp}^0(X_{123}, (\scrH^{r_1}\boxtimes \scrH^{r_2}\boxtimes \scrH^{r_3})^{\vee}(-r), 0) \rightarrow \Q_p, 
\] we know that there exists a unique element \[
    \one_{X_{123}} \in H_{\fp}^0(X_{123}, (\scrH^{r_1}\boxtimes \scrH^{r_2}\boxtimes \scrH^{r_3})^{\vee}(-r), 0) = H_{\dR}^0(X_{123}, (\scrH^{r_1}\boxtimes \scrH^{r_2}\boxtimes \scrH^{r_3})^{\vee})(-r)
\] such that \[
    \langle \tau, \one_{X_{123}} \rangle_{\fp} = \sum_{i=1}^n \langle \tau_1^{(i)}, \pr_{r_1}(\tau_2^{(i)}\otimes \tau_3^{(i)})\rangle_{\fp}.
\]

\begin{Definition}\label{Definition: diagonal cycle with coefficients}
The \textbf{diagonal cycle with coefficients in }$(\scrH^{r_1}\boxtimes \scrH^{r_2}\boxtimes \scrH^{r_3})^{\vee}$ is defined to be \[
    \Delta_{2,2,2}^{k, \ell, m} := \one_{X_{123}} \cdot X_{123} \in A^2(X^3, (\scrH^{r_1}\boxtimes \scrH^{r_2}\boxtimes \scrH^{r_3})^{\vee}(-r)).
\]
\end{Definition}

\begin{Lemma}
The diagonal cycle $\Delta_{2,2,2}^{k, \ell, m}$ is null-homologous. 
\end{Lemma}
\begin{proof}
Note that the sheaf $\scrH^{r_j}$ has no nontrivial global horizontal sections unless $r_j=0$ (\cite[Lemma VII.4]{Brooks-phd}). Thus, by applying Künneth decomposition, we see that $H_{\dR}^2(X^3, \scrH^{r_1}\boxtimes \scrH^{r_2}\boxtimes \scrH^{r_3}) = 0 $. This then implies the desired result. 
\end{proof}

\begin{Theorem}\label{Theorem: formula for diagonal cycles}
For $\eta, \omega_2, \omega_3$ as above, we have \[
    \AJ_{\fp}(\Delta_{2,2,2}^{k, \ell, m})(\eta\otimes \omega_2\otimes\omega_3) = \langle \widetilde{\eta}, \psi_{23}^*(\omega_2 \otimes \omega_3)^{\sim} \rangle_{\fp},
\] where \begin{enumerate}
    \item[$\bullet$] $\widetilde{\eta}$ is a lift of $\eta$ via the projection $H_{\fp}^{1}(X, \scrH^{r_1}, 0) \rightarrow H_{\dR}^1(X, \scrH^{r_1})$, 
    \item[$\bullet$] $(\omega_2\otimes \omega_3)^{\sim}$ is a lift of $\omega_2\otimes \omega_3$ via the projection $H_{\fp}^{2}(X^2, \scrH^{r_2}\boxtimes \scrH^{r_3}, r_2+r_3+2) \rightarrow F^{r_2+r_3+2}H_{\dR}^2(X^2, \scrH^{r_2}\boxtimes \scrH^{r_3})$
    (alternatively, it can be written as $\pi_2^* \widetilde{\omega}_2 \cup \pi_3^* {\widetilde{\omega}_3}$ where $\widetilde{\omega}_j \in H_{\fp}^{1}(X, \scrH^{r_j}, 1)$ is a lift of $\omega_j$),
    \item[$\bullet$] $\psi_{23}^*: H_{\fp}^{2}(X^2, \scrH^{r_2} \boxtimes \scrH^{r_3}, r_2+r_3+2) \rightarrow H_{\fp}^{1}(X, \scrH^{r_1}(-t), r_2+r_3+2)$ is the map given by $\pr_{r_1} \circ \  \iota_{23}'^*$.
\end{enumerate} Here, we use that the self-duality of $\scrH^{r_1}$ gives a pairing $\scrH^{r_1}\otimes \scrH^{r_1}(-t) \rightarrow \scrO_{X}(-r)$.
\end{Theorem}
\begin{proof}
By the definition of $\Delta_{2,2,2}^{k, \ell, m}$, we have \begin{align*}
    \AJ_p(\Delta_{2,2,2}^{k, \ell, m})(\eta \otimes \omega_2 \otimes \omega_3) &=  \langle \iota_{123}^* (\eta\otimes \omega_2\otimes\omega_3)^\sim , \mathbbm{1}_{X_{123}} \rangle_{\fp} \\
     &=  \langle \iota_{123}^* (\pi_1^* \tilde{\eta} \cup \pi_2^* \tilde{\omega_2} \cup \pi_3^* \tilde{\omega_3} ), \mathbbm{1}_{X_{123}} \rangle_{\fp} \\
     &= \langle \tilde{\eta}, \pr_{r_1} \circ  \iota'^*_{23}(\omega_2 \otimes \omega_3)^{\sim}\rangle_{\fp} \\
     &= \langle \widetilde{\eta}, \psi_{23}^*(\omega_2 \otimes \omega_3)^{\sim} \rangle_{\fp}
\end{align*} 
\end{proof}

\begin{Remark}\label{Remark: Diagonal cycles}
\normalfont 
Readers might notice that our choice of the diagonal cycle $\Delta_{2,2,2}^{k, \ell, m}$ only involves $X_{123}$, which is simpler than the diagonal cycle
$$\Delta_{2,2,2} := - \sum_{\emptyset \neq I \subset \{1, 2, 3\} } (-1)^{\# I} \cdot \iota_I(X)$$
considered in \cite{BDP} for the case $k= \ell = m= 2$.
It could be thought as that the terms other than $X_{123}$ are needed only to make the cycle null-homologous. 
As shown in \cite{Besser-regulatorformula}, those terms do not contribute to the Abel--Jacobi map in the end.
So, it makes sense to consider only $X_{123}$, the image of the diagonal embedding.
\blackqed
\end{Remark}

Note that, because of Lemma \ref{Lemma: Weil weight}, we can choose a polynomial $P(T)\in \Q[T]$ such that \begin{enumerate}
    \item[$\bullet$] $P(\Phi)(\omega_2 \otimes \omega_3) = 0 \in H_{\dR}^2(X^2, \scrH^{r_2} \boxtimes \scrH^{r_3})$ and 
    \item[$\bullet$] $P(\Phi)$ acts invertibly on $H_{\dR}^1(X^2, \scrH^{r_2}\boxtimes \scrH^{r_3})$.
\end{enumerate} As a consequence, there exists a section $\rho(P, \omega_2, \omega_3)\in H^0(X^2, (\scrH^{r_2} \boxtimes \scrH^{r_3})\otimes\Omega_{X^2/\Q_p}^1)$ such that \[
    \nabla \rho(P, \omega_2, \omega_3) = P(\Phi)(\omega_2\otimes \omega_3)
\] and it is well-defined up to a horizontal section. 
Moreover, as $P(\Phi)$ acts isomorphically on $H_{\dR}^1(X, \scrH^{r_1}(-t))$, we may set \[
    \xi(\omega_2, \omega_3) := P(\Phi)^{-1}\psi_{23}^*\rho(P, \omega_2, \omega_3)\in H_{\dR}^1(X, \scrH^{r_1}(-t)).
\]
Similar as in \cite[Proposition 3.7]{DR}, this element is independent of the choice of the polynomial $P$. 

\begin{Corollary}\label{Corollary: relation with DR}
Notation as above. We have \[
    \AJ_{\fp}(\Delta_{2,2,2}^{k, \ell, m})(\eta\otimes \omega_2\otimes\omega_3) = \langle \eta, \xi(\omega_2, \omega_3)\rangle_{\dR}.
\]
In particular, 
$$ \AJ_{\fp}(\Delta_{2,2,2}^{k, \ell, m})(\eta\otimes \omega_2\otimes\omega_3) = \AJ_{p}(\Delta_{k, \ell, m})(\eta\otimes \omega_2\otimes\omega_3)$$
where $\AJ_p$, the ($p$-adic) Abel-Jacobi map, and $\Delta_{k, \ell, m}$, the generalised Gross--Kudla--Schoen cycle, are both defined similarly as in \cite{DR} via Kuga--Sato varieties.
\end{Corollary}
\begin{proof}
The proof is similar to the one of \cite[Theorem~3.8]{DR} (more precisely, the passage after Lemma 3.11 of \emph{op. cit.}).
We just translate it into the language of finite polynomial cohomology with coefficients.

First, observe that the element $(\omega_2 \otimes \omega_3)^\sim$ can be represented by 
$$(\rho(P, \omega_2, \omega_3), \omega_2 \otimes \omega_3)  \in H^2_{\syn, P}(X^2, \scrH^{r_2} \boxtimes \scrH^{r_3}, 2)$$
Now consider the following commutative diagram coming from the exact sequences in Corollary \ref{Corollary: exact sequence when X is of good reduction}
$$
\begin{tikzcd}[column sep = small]
    0 \arrow[r]  &H^1_{\dR}(X^ 2, \scrH^{r_2} \boxtimes \scrH^{r_3}) \arrow[r, "i_{\fp}"] \arrow[d, "\psi^*_{23}"]& H^2_{\syn, P}(X^2, \scrH^{r_2} \boxtimes \scrH^{r_3}, r_2+r_3+2) \arrow[r, "\pr_{\fp}"] \arrow[d, "\psi^*_{23}"] & F^{r_2+r_3+2} H^2_{\dR}(X^2, \scrH^{r_2} \boxtimes \scrH^{r_3}) \arrow[r] \arrow[d, "\psi^*_{23}"] &0 \\
    0 \arrow[r]  &H^1_{\dR}(X, \scrH^{r_1}(-t)) \arrow[r, "i_{\fp}"] & H^2_{\syn, P}(X, \scrH^{r_1}(-t)), r_2+r_3+2) \arrow[r, "\pr_{\fp}"] & F^{r_2+r_3+2} H^2_{\dR}(X, \scrH^{r_1}(-t)) \arrow[r] &0
\end{tikzcd}$$
As $F^{r_2+r_3+2} H^2_{\dR}(X, \scrH^{r_1}(-t)) =0$, the element $\psi^*_{23}$ lies inside the image of $i_{\fp}$.
Moreover, the map $i_{\fp}$ is given by
$\xi \mapsto (P(\Phi)\xi, 0)$.
Hence we see that $\psi^*_{23}(\omega_2 \otimes \omega_3)^\sim = i_{\fp}( \xi(\omega_2, \omega_3))$.

Now the result follows from the compatibility of the Poincar\'{e} pairings and the maps $i_{\fp}$ and $\pr_{\fp}$ (Corollary \ref{Corollary: pairing for finite polynomial cohomology}), which reads 
\begin{align*}
    \langle \widetilde{\eta}, \psi_{23}^*(\omega_2 \otimes \omega_3)^{\sim} \rangle_{\fp} &= \langle \pr^{-1}_{\fp} \pr_{\fp} (\widetilde{\eta}), i_{\fp} (\xi(\omega_2, \omega_3)) \rangle_{\fp} \\
    &= \langle \pr^{-1}_{\fp}(\eta), i_{\fp} (\xi(\omega_2, \omega_3)) \rangle_{\fp} \\
    &= \langle \eta, \xi(\omega_2, \omega_3) \rangle_{\dR}.
\end{align*}
\end{proof}

\begin{Remark}\label{Remark: case of elliptic modular curves}
\normalfont 
Although the computations above are in the setting of compact Shimura curves over $\Q$, our strategy should directly apply to the setting of elliptic modular curves (as in \cite{DR}) after taking care of the log structure on the compactified modular curves defined by the cusps.
\blackqed
\end{Remark}

\begin{Remark}\label{Remark: L-value}
\normalfont 
Of course, by choosing $\eta$, $\omega_2$ and $\omega_3$ carefully and applying the strategy in \cite{DR}, one can deduce that the formula in Theorem \ref{Theorem: formula for diagonal cycles} computes special values of triple product $L$-functions associated with certain eigenforms. As the purpose of this section is to indicate how the theory of finite polynomial cohomology with coefficients can be applied in practice, we do not intend to link our formula with special $L$-values in this paper. 
\blackqed
\end{Remark}

\subsection{Cycles attached to isogenies and formulae of Bertolini--Darmon--Prasanna}\label{subsection: cycles; isogenies}

In this subsection, we establish a second application of finite polynomial cohomology with coefficients in the spirit of \cite{BDP}. 
To this end, we fix a finite extension $K$ of $\Q_p$ and base change $X$ to $K$. We abuse the notation and still denote it by $X$. 
We also fix two $K$-rational points $x = (A, i, \psi)$ and $y = (A', i', \psi')$ in $X(K)$ such that there is an isogeny $\varphi: (A, i, \psi) \rightarrow (A', i', \psi')$. 
Finally, we fix an integer $r\in \Z_{>0}$ and consider the sheaf \[
    \scrH^{r,r} := \scrH^r \otimes \Sym^r \bfepsilon H_{\dR}^1(A).
\]

Our result in this subsection is the following theorem:
\begin{Theorem}\label{Theorem: cycle attached to isogenies}
There is a unique cycle $\Delta_\varphi = \theta \cdot y \in A^1(X, \scrH^{r,r}(r))_0$
such that for any $\omega \in H^0(X, \underline{\omega}^{r+2})$ and any $\alpha\in \Sym^r \bfepsilon H_{\dR}^1(A)$, we have \[
    \AJ_{\fp}(\Delta_{\varphi})(\omega \otimes \alpha) = \langle F_{\omega}(y) \otimes \alpha , \theta \rangle =  \langle \varphi^* (F_{\omega}(y)), \alpha \rangle.
\] 
Here \begin{enumerate}
    \item[(i)] $\theta$ is an element in $\scrH^{r,r}(r)(y)$ that will be specified later;
    \item[(ii)] $F_{\omega}$ is the Coleman integral of the form $\omega$;
    \item[(iii)] the middle pairing is on $\scrH^{r,r}(y) = \Sym^r \bfepsilon  H_{\dR}^1(A')\otimes \Sym^r \bfepsilon  H_{\dR}^1(A)$; and 
    \item[(iv)] the last pairing is on $\Sym^r \bfepsilon H_{\dR}^1(A)$.
\end{enumerate} In particular, we obtain the same formulae as in \cite[Proposition 3.18 \& Proposition 3.21]{BDP}.
\end{Theorem}

\begin{Remark}
\normalfont 
\begin{enumerate}
    \item[(i)] Similarly as before, we have $H_{\dR}^2(X, \scrH^{r,r}) = 0$, which implies that $A^1(X, \scrH^{r,r}(r)) = A^1(X, \scrH^{r,r}(r))_0$. 
    \item[(ii)] To be more precise, the pairing in Theorem \ref{Theorem: cycle attached to isogenies} (iii) is between $\scrH^{r,r}(y)$ and its twist $\scrH^{r,r}(y) (r)$.
\end{enumerate}
\blackqed
\end{Remark}

Let $\omega$ and $\alpha$ be as in the theorem. Recall that we have a short exact sequence \[
    0 \rightarrow \frac{H_{\dR}^0(X, \scrH^{r, r})}{F^{r+1} H_{\dR}^0(X, \scrH^{r,r})} \xrightarrow{i_{\fp}} H_{\fp}^1(X, \scrH^{r,r}, r+1) \xrightarrow{\pr_{\fp}} F^{r+1} H_{\dR}^1(X, \scrH^{r,r}) \rightarrow 0.
\] Choose a lift $(\omega \otimes \alpha)^{\sim}\in H_{\fp}^1(X, \scrH^{r, r}, r+1)$ of $\omega \otimes \alpha\in F^{r+1} H_{\dR}^1(X, \scrH^{r,r})$. 
Then, for any $\theta \in H_{\fp}^0(y, \scrH^{r,r}(r), 0) = H_{\dR}^0(y, \scrH^{r,r}(r))$, we have \[
    \AJ_{\fp}(\theta)(\omega \otimes \alpha) = \langle (\omega \otimes \alpha)^{\sim}, \iota_{y, *}\theta \rangle_{\fp, X} = \langle \iota_y^*(\omega \otimes \alpha)^{\sim}, \theta\rangle_{\fp, y},
\] where $\iota_y: y \hookrightarrow X$ is the natural closed embedding. 
Note that $\iota_y^*(\omega \otimes \alpha)^{\sim} = \iota_y^* \widetilde{\omega}\otimes \alpha$, where $\widetilde{\omega}\in H_{\fp}^1(X, \scrH^r, r+1)$ is a lift of $\omega$.

On the other hand, the short exact sequence \[
    0 \rightarrow \frac{H_{\dR}^0(y, \scrH^r)}{F^{r+1} H_{\dR}^0(y, \scrH^r)} \xrightarrow{i_{\fp}} H_{\fp}^1(y, \scrH^r, r+1) \xrightarrow{\pr_{\fp}} F^{r+1} H_{\dR}^1(y, \scrH^r) \rightarrow 0
\] and the fact that $H_{\dR}^1(y, \scrH^r) = 0$ imply there exists $s_{\omega}\in H_{\dR}^0(y, \scrH^r)$ such that $i_{\fp}(s_{\omega}) = \iota_y^* \widetilde{\omega}$. 
Hence, we have \begin{equation}\label{eq: preliminary computation of cycles attached to isogenies}
    \begin{array}{rl}
        \AJ_{\fp}(\theta)(\omega \otimes \alpha) & = \langle \iota_y^*(\omega \otimes \alpha)^{\sim}, \theta \rangle_{\fp, y}\\
        & = \langle \iota_y^* \widetilde{\omega} \otimes \alpha, \theta \rangle_{\fp, y}\\
        & = \langle i_{\fp}(s_{\omega})\otimes \alpha, \theta \rangle_{\fp, y}\\
        & = \langle s_{\omega}\otimes \alpha, \theta \rangle_{\dR, y}. 
    \end{array}
\end{equation}

\begin{Lemma}\label{Lemma: pullback of omega viewed as Coleman integration}
In $\scrH^r(y)$, we have \[
    s_{\omega} = F_{\omega}(y),
\] where $F_{\omega}$ is the Coleman integral corresponding to $\omega$. Note that $F_{\omega}$ is unique since $\scrH^r$ has no global horizontal section on $X$. 
\end{Lemma}
\begin{proof}

We first choose a polynomial $P\in \Poly$ such that $H_{\syn, P}^1(X, \scrH^r, r+1) = H_{\fp}^1(X, \scrH^r, r+1)$. Then, we have the commutative diagram \[
    \begin{tikzcd}
        0 \arrow[r] & \dfrac{H_{\dR}^0(X, \scrH^r)}{F^{r+1} H_{\dR}^0(X, \scrH^r)} \arrow[r, "i_{\fp}"]\arrow[d, "\iota_y^*"] & H_{\syn, P}^1(X, \scrH^r, r+1)\arrow[r, "\pr_{\fp}"]\arrow[d, "\iota_y^*"] & F^{r+1} H_{\dR}^1(X, \scrH^r) \arrow[r]\arrow[d, "\iota_y^*"] & 0\\
        0 \arrow[r] & \dfrac{H_{\dR}^0(y, \scrH^r)}{F^{r+1} H_{\dR}^0(y, \scrH^r)} \arrow[r, "i_{\fp}"] & H_{\syn, P}^1(y, \scrH^r, r+1) \arrow[r, "\pr_{\fp}"] & F^{r+1} H_{\dR}^1(y, \scrH^r) \arrow[r] & 0
    \end{tikzcd}.
\] Note that $\frac{H_{\dR}^0(X, \scrH^r)}{F^{r+1} H_{\dR}^0(X, \scrH^r)}$ and $F^{r+1} H_{\dR}^1(y, \scrH^r)$ both vanish. 

Let $\calW$ be a strict neighbourhood of the rigid analytic space associated with an affinoid inside $X$ such that $\calW$ contains the residue disc of $y$ and all its Frobenius translations. By Corollary \ref{Corollary: relation with Coleman integration}, the restriction of $\widetilde{\omega}$ to $\calW$ is (uniquely) represented by an element of the form $(G', \omega)$ where $G' \in H^0 (\calW, \scrH^r)$ is such that $P(\Phi)\omega = \nabla G'$.
Hence we have $P(\Phi)F_\omega = G'$.

The equation $s_{\omega} = F_{\omega}(y)$ is clear due to the construction of Coleman integration (\cite{C-drl, C-pShimura}) and that \[
    i_{\fp}: H_{\dR}^0(y, \scrH^r) = \scrH^r(y) \rightarrow H_{\syn, P}^1(y, \scrH^r, r+1)
\] 
is the map $s \mapsto (P(\Phi)(s), 0)$.
\end{proof}
\begin{Remark}
We remark that the term $(\Phi F_\omega )(y)$ is not the naive $F_\omega( \phi(y))$, where $\phi: X \rightarrow X$ is the Frobenius.
Rather, it is understood as the following explanation:

Since $\scrH^r$ is an overconvergent $F$-isocrystal, there is a Frobenius neighbourhood $\calW' \subset \calW$ (see \cite[\S 10]{C-pShimura}) and a horizontal map $\Phi: \phi^* \scrH^r \rightarrow \scrH^r|_{\calW'}$.
This induces a map $\Phi$ from $(\scrH^r)_{\phi(y)}$ to $(\scrH^r)_y$, which may be thought as the `parallel transport along the Frobenius' from $\phi(y)$ to $y$.
Hence $P(\Phi)(y)$ takes value in $\scrH^r(y)$.
\blackqed
\end{Remark}

Now we construct the desired cycle $\Delta_{\varphi}$ by specifying the coefficient $\theta$. 
Observe that, since $ \bfepsilon H_{\dR}^1(A)$ is self-dual (up to a twist), we have a natural perfect pairing (see, for example, \cite[(1.1.14)]{BDP})
\[
    \Sym^r \bfepsilon H_{\dR}^1(A) \otimes \Sym^r \bfepsilon H_{\dR}^1(A)(r) \rightarrow K, \quad \beta \otimes \gamma \mapsto \langle \beta, \gamma \rangle. 
\] This induces a perfect pairing between $\Sym^r \bfepsilon H_{\dR}^1(A) \otimes \Sym^r \bfepsilon H_{\dR}^1(A)(r)$ and itself. Hence, there is an element $\one_A\in \Sym^r \bfepsilon H_{\dR}^1(A) \otimes \Sym^r \bfepsilon H_{\dR}^1(A)(r)$ such that \[
    \langle \beta\otimes \gamma, \one_A \rangle = \langle \beta, \gamma \rangle.
\] 

\begin{Definition}\label{Definition: cycle attached to isogenies}
The \textbf{cycle attached to the isogeny $\varphi$} is defined to be \[
    \Delta_{\varphi} := \theta \cdot y \in A^1(X, \scrH^{r,r}(r)) = A^1(X, \scrH^{r,r}(r))_0,
\] where $\theta = \varphi_*\one_A$ and $\varphi_*$ is the pushforward map \[
    \varphi_*: \scrH^{r,r}(x) = \Sym^r \bfepsilon H_{\dR}^1(A)\otimes \Sym^r \bfepsilon H_{\dR}^1(A)(r) \rightarrow \Sym^r \bfepsilon H_{\dR}^1(A') \otimes \Sym^r \bfepsilon H_{\dR}^1(A)(r) = \scrH^{r,r}(r)(y)
\] induced by the isogeny $\varphi: (A, i, \psi) \rightarrow (A', i', \psi')$.
\end{Definition}

\begin{proof}[Proof of Theorem \ref{Theorem: cycle attached to isogenies}]
The proof now follows easily by combining \eqref{eq: preliminary computation of cycles attached to isogenies}, Lemma \ref{Lemma: pullback of omega viewed as Coleman integration} and the definition of $\Delta_{\varphi}$.
\end{proof}

\begin{Remark}
\normalfont As before, our strategy can be directly applied to the setting of elliptic modular curves as in \cite{BDP} after taking care of the log structure on the compactified modular curves. In this case, readers should be able to compare this to the proof of Proposition 3.21 in \emph{op. cit.}.
The pushforward of the fundamental class $[A^r] \in H^0_{\dR}(A^r)$ under the map 
$ (\varphi^r, \id^r): A \rightarrow X_r$ (notations as in \emph{op. cit.}) corresponds to our element $\one_A$.
\blackqed
\end{Remark}

\begin{Remark}
\normalfont Similarly as before, if we choose $x$ and $y$ to be \emph{Heegner points} on the Shimura curve (or modular curve), it is expected that the formula in Theorem \ref{Theorem: cycle attached to isogenies} computes special values of the anti-cyclotomic $p$-adic $L$-functions as in \cite{BDP}. Since we only wanted to demonstrate the use of finite polynomial cohomology with coefficients in practice, we do not intend to link our formula with special $L$-values.  
\blackqed
\end{Remark}

\begin{Remark}
As it could be seen in the two applications, one advantage of using the finite polynomial cohomology with coefficients is that our cycles are much simpler.
In particular, we can get rid of the idempotents used to define the cycles in \cite{DR} and \cite{BDP}, which simplifies the computations.
\blackqed
\end{Remark}

\printbibliography[heading=bibintoc]

\vspace{10mm}

\begin{tabular}{l}
    T.-H.H.\\
    Laboratoire Analyse, G\'{e}om\'{e}trie et Applications\\
    Universit\'{e} Paris XIII (Sorbonne Paris Nord)  \\
    Villetaneuse, France\\
    \textit{E-mail address: }\texttt{ting-han.huang@math.univ-paris13.fr}
\end{tabular}
\begin{tabular}{l}
    J.-F.W.\\
    School of Mathematics and Statistics\\
    University College Dublin\\
    Belfield Dublin 4, Ireland\\
    \textit{E-mail address: }\texttt{ju-feng.wu@ucd.ie}
\end{tabular}


\end{document}